\documentclass[11pt]{amsart}

\usepackage[
  top=1.3in,
  bottom=1in,
  left=1.2in,
  right=1.2in,
]{geometry}

\usepackage{latexsym}
\usepackage{comment}
\usepackage{amsmath}
\usepackage{amsfonts}
\usepackage{amssymb}
\usepackage{geometry} 
\usepackage{graphicx,color}
\usepackage{epstopdf}
\usepackage{caption}
\usepackage{diagbox}
\usepackage{subfig}
\usepackage{bm}
\usepackage{algorithm}
\usepackage{algorithmic}
\usepackage{tikz}
\usepackage{stmaryrd}

\graphicspath{{figures/}}
\newtheorem{theorem}{Theorem}[section]

\newtheorem{lemma}[theorem]{Lemma}

\theoremstyle{definition}
\newtheorem{definition}[theorem]{Definition}

\newtheorem{assumption}{Assumption}

\theoremstyle{remark}
\newtheorem{remark}[theorem]{Remark}

\numberwithin{equation}{section}

\def\bz{\bm{\zeta}}

\def\bn{\mathbf{n}}

\def\O{\Omega}

\def\Tb{{\bf T}}

\def\cB{\mathcal{B}}

\def\R{\mathcal{R}}
\def\S{\mathcal{S}}
\def\C{\mathscr{C}}
\def\M{\bf{M}}
\def\A{{\bf A}}
\def\tp{\widehat{p}}

\def\ty{\widehat{y}}
\def\tpb{\bf{p}}
\def\tqb{\bf{q}}

\def\H2{{H^2(\O)}}

\def\T{\mathcal{T}}
\def\B{\mathcal{B}}

\def\A{\bf{A}}

\def\l{\langle}
\def\r{\rangle}

\def\g{\bf g}
\def\u{\bf u}
\def\L{{\bf L}}
\def\w{{\bf w}}
\def\W{{\bf W}}
\def\S{{\bf S}}
\def\wS{\widehat{\bf S}}
\def\g{{\bf g}}

\def\U{{\bf U}}
\def\R{{\bf R}}
\def\L{{\bf L}}
\def\D{{\bf D}} 

\def\B{{\bf B}}
\def\E{{\bf E}}
\def\C{\bm{\mathcal C}}
\def\v{{\bf v}}
\def\z{{\bf z}}
\def\s{{\bm s}}

\def\a{{\bm a}}
\def\b{{\bf b}}

\def\H{{\mathcal H}}

\def\bq{\mathbf{q}}
\def\bp{\mathbf{p}}
\def\lambdae{\bm\lambda}
\def\Lambdae{\bm\Lambda}

\def\v{\bf v}
\def\Aad{{\bf A}^{\bf ad}}

\def\Aai{{{\bf A}^{(i)}_{\bf a}}}
\def\Aa{{{\bf A}_{\bf a}}}

\def\Aad{{\bf A}^{\bf ad}}

\def\Aaoi{{{\bf A}^{(i)}_{\bf a}}}

\def\Ma{{{\bf M}_{\bf a}}}
\def\La{{{\bf L}_{\bf a}}}
\def\ad{{\bf ad}}
\def\vpg{{\bf\varphi}_{\bf g}}
\def\vpgt{\widetilde{\bf\varphi}_{\bf g}}
\def\vps{{\bf\varphi}_{\bf s}}
\def\vpst{\widetilde{\bf\varphi}_{\bf s}}

\def\Z{{\bf Z}}

\def\xg{{\bm \xi}_{\bf g}}
\def\z{{\bf z}}
\def\Aadt{{A}^{ad}}
\def\LO2{{L^2(\Omega)}}
\def\G{{\bf G}}

\DeclareMathOperator*{\argmin}{argmin}

\begin{document}

\title{Convergence analysis of a balancing domain decomposition method for an elliptic  optimal control problem with HDG discretizations}

\author{Sijing Liu and Jinjin Zhang}
\address{Sijing Liu, Department of Mathematical Sciences\\
Worcester Polytechnic Institute\\
Worcester, MA\\
USA}
\email{sliu13@wpi.edu}

\address{Jinjin Zhang, Department of Mathematics\\
The Ohio State University\\
Columbus, OH\\
USA}
\email{zhang.14647@osu.edu}

\subjclass{49J20, 49M41, 65N30, 65N55}
\keywords{elliptic distributed optimal control problems, hybridizable discontinuous Galerkin methods, BDDC algorithms}
\date{\today}

\begin{abstract}
In this work, a balancing domain decomposition by constraints (BDDC) algorithm is applied to the nonsymmetric positive definite linear system arising from the hybridizable discontinuous Galerkin (HDG) discretization of an elliptic distributed optimal control problem. Convergence analysis for the BDDC preconditioned generalized minimal residual (GMRES) solver demonstrates that, when the subdomain size is small enough, the algorithm is robust with respect to the regularization parameter, and the number of iterations is independent of the number of subdomains and depends only slightly on the subdomain problem
size. Numerical experiments are performed to confirm the theoretical results.
\end{abstract}

\maketitle

\section{Introduction}

Elliptic distributed optimal control problems have been extensively studied using various discretization techniques. Standard $P_1$ finite element methods are applied to an elliptic optimal control problem constrained by a diffusion-reaction equation in \cite{Schoberl:2011:RobustMG} and to one constrained by a advection–diffusion–reaction equation in \cite{brenner2020multigrid}. Discontinuous Galerkin (DG) methods have been explored in~\cite{leykekhman2012investigation,leykekhman2012local,liu2025robust,liu2024discontinuous} for their applicability to optimal control problems as well. Moreover, HDG methods are investigated in~\cite{chen2018hdg,chen2019hdg,HuShenSinglerZhangZheng2018} and the error estimates in $L^2$ norm are provided. Streamline upwind/Petrov-Galerkin (SUPG) methods have also been developed and analyzed for the elliptic optimal control problems constrained by a advection-dominated state equation in~\cite{CollisHeinkenschloss2024_SUPGOptimalControl,HeinkenschlossLeykekhman2010_SUPGLocalError}. In particular, an error estimate in the energy norm for an HDG discretization has been established in \cite{liuzhang2025}, where the dependence on the regularization parameter is explicitly tracked.




In terms of existing solvers for the discretized system, multigrid methods provide an effective approach for solving elliptic optimal control problems. Typically, there are two main approaches (cf. \cite{Schoberl:2011:RobustMG}): the all-at-once approach (see \cite{Taasan1991,SimonZulehner2009,brenner2020multigrid}), in which the multigrid method is applied directly to the coupled optimality system, and the approach in \cite{Biros:2005:LNKS,BirosGhattas2005PartII,HazraSchulz2004}, where multigrid methods are applied separately to the equations of the state variable, the dual variable, and/or the control variable as components of the overall iterative scheme.
For example, in \cite{SimonZulehner2009}, an additive Schwarz-type iterative method is formulated as a preconditioned Richardson method, which enables the use of multigrid as a fully coupled all-at-once solver for the Karush–Kuhn–Tucker system, and in \cite{Biros:2005:LNKS}, a domain-decomposed Schur complement partial differential equation (PDE) solver with a Krylov–Schur preconditioner is proposed, which employs an approximate state/decision variable decomposition by replacing the forward PDE Jacobians with their respective preconditioners.

BDDC algorithms are popular nonoverlapping domain decomposition techniques that were first introduced in~\cite{Doh04}; they have been developed for solving symmetric positive definite problems with mixed and hybrid formulations~\cite{Tu2005mix,Tu2007BDDC}, DG methods~\cite{KimChungLee2014BDDC}, and HDG discretizations~\cite{TuWang2016HDG}. A variant of the BDDC method has been proposed for symmetric indefinite systems arising from finite element discretizations of the Helmholtz equation in~\cite{LiTu2009}. BDDC algorithm can also be extended to solve nonsymmetric positive definite problems. In~\cite{TZAD2021}, BDDC algorithms are applied to the nonsymmetric positive definite discretized system resulting from using HDG discretizations on the advection–diffusion equation.
In~\cite{TZOseen2022}, BDDC algorithms are developed for the Oseen equations, where the globally coupled unknowns are reduced to the numerical trace of the velocity on element boundaries and the mean pressure on each element. The resulting discretized nonsymmetric saddle point system, derived from HDG discretizations, is solved within a benign subspace framework using GMRES.
In~\cite{liuzhang2025}, we propose a BDDC preconditioner to solve an elliptic distributed optimal control problem discretized by HDG methods. In this approach, the state and adjoint state variables on the subdomain interfaces are coupled as the primary unknowns. The original problem is reduced to a global interface system. A BDDC preconditioner is constructed by assembling the subdomain Schur complements and is used to solve the coupled interface unknowns. Once the interface variables are computed, the remaining variables in the system can be efficiently recovered.

In this article, we present a convergence analysis for the BDDC-preconditioned GMRES method for the optimal control problem proposed in~\cite{liuzhang2025}. The analysis establishes both upper and lower bounds with respect to the regularization parameter $\beta$, as stated in Theorem~\ref{theorem:EES}, particularly for small values of $\beta$. This work is a continuation of \cite{liuzhang2025}. We show that when the subdomain size $H$ is sufficiently small, the convergence rate becomes independent of $\beta$, and the number of iterations is independent of the number of subdomains.
The analysis proceeds in several steps. First, given the nuemrical trace
$\lambdae$ defined on the union of element boundaries, we construct some extensions from the HDG discretized system of the optimal control problem. Using the results in~\cite{TZAD2021}, we derive estimates for the norms of these extensions by relating the optimal control equations to a pair of advection-diffusion equations.
Second, due to the presence of the $L^2$ inner product in the nonsymmetric part of the bilinear form, we introduce a new norm in Theorem~\ref{theorem:EES}. Unlike [28, Theorem 6.6], which only focuses on the symmetric part of the bilinear form, our norm incorporates the $L^2$ norm of the extension operators in addition to the symmetric component. Consequently, it is not equivalent to the triple-bar norm used in the advection-diffusion problem (see \cite{TZAD2021}), which behaves similarly to an $H^1$ semi-norm.
Instead, we propose a full norm rather than a semi-norm, which behaves similarly to an energy norm as defined in~\cite[Equation (3.3)]{liuzhang2025}. This difference makes it difficult to directly estimate the $h$-norm, an $L^2$ type norm, of the optimal control extension for certain key operators in the lower bound analysis. To address this issue, we estimate the $h$-norm of the corresponding advection-diffusion extension of these operators, then apply norm equivalence relations to obtain the desired lower bound.
Third, we establish upper and lower bounds that explicitly track the dependence on the regularization parameter $\beta$. These bounds are derived by relating extensions of the optimal control problem to extensions of the advection-diffusion problem and applying known estimates for the latter.

The rest of the paper is organized as follows. In Section \ref{sec:hdg}, we introduce the HDG discretizations of the optimal control problem and derive the corresponding matrix formulation. In Section \ref{sec:bddc}, the subdomain local problem, a reduced subdomain interface problem,
and the BDDC preconditioner are given. In Section \ref{sec:extennorms}, we define some extensions for the optimal control problem and introduce some useful blinear forms and norms. In Section \ref{sec:conv}, we give some estimates of these extensions and establish the upper and lower bounds for the BDDC algorithm with the regularization parameter $\beta$ explicitly tracked. In Section \ref{sec:lemmaproofs}, we provide the proofs of the lemmas used in Section \ref{sec:conv}. In Section \ref{sec:numerics}, we provide some numerical experiments to confirm the theoretical results. We end with some concluding remarks in Section \ref{sec:conclude}.

\section{Problem setting and HDG discretizations}\label{sec:hdg}
We consider the following elliptic optimal control problem  defined on a bounded convex polygonal domain $\Omega$ in $\mathbb{R}^2$ : 
\begin{equation}\label{optcon}
\mbox{Find}\quad (\bar{y},\bar{u})=\argmin_{(y,u)\in K}\left [ \frac{1}{2}\|y-y_d\|^2_{\LO2}+\frac{\beta}{2}\|u\|^2_{\LO2}\right],
\end{equation} 
where $(y,u)$ belongs to $K\subset H^1_0(\Omega)\times L^2(\Omega)$ and satisfies
\begin{equation*}
\int_{\Omega} \nabla y\cdot \nabla v\ dx+\int_{\Omega} (\bz\cdot\nabla y) v\ dx=\int_{\Omega}uv \ dx \quad \forall v\in H^1_0(\Omega),
\end{equation*}
with vector field $\bz\in [W^{1,\infty}(\Omega)]^2$ and $\nabla \cdot{\bz}=0$. Here $y$ denotes the state variable and $u$ is the control variable.

{{Let \(p\) denote the adjoint state variable. In~\cite{liuzhang2025}, we provide a detailed derivation of the following system from a saddle point problem~\eqref{optcon} by applying a scaling argument to the variables \(p\) and \(y\); further details can be found therein.}} The optimal control problem can be described as follows: find $(y,p)\in H^1_0(\Omega)\times H^1_0(\Omega)$ such that  
\begin{subequations}\label{Osystem}
 \begin{alignat}{3}
 \beta^{\frac12}(-\Delta y{+}\bz\cdot\nabla{y}){-}p&=f\quad&&\mbox{in}\quad\Omega,\\
 y&=0\quad&&\mbox{on}\quad\partial\Omega,\\
 \beta^{\frac12}(-\Delta p{-}\nabla\cdot(\bz p))+y&=g\quad&&\mbox{in}\quad\Omega,\\
 p&=0\qquad &&\mbox{on}\quad\partial\Omega.
 \end{alignat}
 \end{subequations}
 Let $\bq=-\nabla y$ and $\bp=-\nabla p$. we can write \eqref{Osystem} as 
\begin{subequations}\label{eq:fsystem}
 \begin{alignat}{3}
\bq+\nabla y&=0\quad&&\mbox{in}\quad\Omega,\\
 \beta^{\frac12}(\nabla\cdot\bq+\bz\cdot\nabla y){-}p&=f\quad&&\mbox{in}\quad\Omega,\\
 y&=0\quad&&\mbox{on}\quad\partial\Omega,\\
 \bp+\nabla p&=0\quad&&\mbox{in}\quad\Omega,\\
 \beta^{\frac12}(\nabla\cdot \bp-\nabla\cdot(\bz p))+y&=g\quad&&\mbox{in}\quad\Omega,\\
 p&=0\qquad &&\mbox{on}\quad\partial\Omega.
 \end{alignat}
 \end{subequations}
\subsection{HDG formulation and the matrix form}
We define a shape-regular and quasi-uniform triangulation $\mathcal{T}_h$ of the domain $\Omega$, where $h$ represents the characteristic element size. Each element in \(\mathcal{T}_h\) is denoted by \(K\).
Define the inner products
$
(\cdot, \cdot)_{\mathcal{T}_h} = \sum_{K \in \mathcal{T}_h} (\cdot, \cdot)_K $ {and}
$\langle \cdot, \cdot \rangle_{\partial \mathcal{T}_h} = \sum_{K \in \mathcal{T}_h} \langle \cdot, \cdot \rangle_{\partial K},
$
where \((\cdot, \cdot)_K\) and \(\langle \cdot, \cdot \rangle_{\partial K}\) represent the \(L^2\)-inner products for functions defined over the element \(K\) and its boundary \(\partial K\), respectively. Define the finite element space as follows: 
\begin{align*}
&{\bf V}_h=\{{\bf v}\in (L^2(\Omega))^n:{\bf v}|_K\in (P^k(K))^n,\forall K\in\mathcal{T}_h\},\\
&{W}_h=\{w\in L^2(\Omega): w|_K\in P^k(K),\forall K\in\mathcal{T}_h\},\\
&{\Lambda}^c_h=\{\mu\in L^2(\mathcal{E}_h):\mu|_{e}\in P^k(e),\forall e\in \mathcal{E}_h \},\\
&\Lambda_h=\{\mu\in {G}_h:\mu|_{e}=0,\forall e\in\partial\Omega\}.
\end{align*}
For notational simplicity, we will drop the subscript $h$ and denote the corresponding spaces by ${\bf V}, {W}$ and $\Lambda$.
The HDG method is to find $({\tqb}_h,{\tpb}_h,y_h,{p}_h, \widehat{y}_h,\widehat{p}_h) \in {\bf V}\times{\bf V}\times{W}\times{W}\times\Lambda\times\Lambda$ such that,
\begin{subequations}\label{eq:hdg2}
  \begin{alignat}{1}
-\beta^{\frac12}({\bf q}_h,{\bf r}_1)_{\T_h}+\beta^{\frac12}({y}_h,\nabla\cdot{\bf r}_1)_{\T_h}-\beta^{\frac12}\l\widehat{{y}}_h,{\bf r}_1\cdot{\bf n}\r_{\partial\T_h}&=0,\label{eq:qh1}\\
-\beta^{\frac12}({\bf p}_h,{\bf r}_2)_{\T_h}+\beta^{\frac12}({p}_h,\nabla\cdot{\bf r}_2)_{\T_h}-\beta^{\frac12}\l\widehat{{p}}_h,{\bf r}_2\cdot{\bf n}\r_{\partial\T_h}&=0,\label{eq:ph1}\\
\beta^{\frac12}(\nabla\cdot{\tqb}_h,w_1)_{\T_h}-\beta^{\frac12}({y}_h,\bz\cdot\nabla{w}_1)_{\T_h}
+\beta^{\frac12}\l\tau_1{y}_h,{ w}_1\r_{\partial \T_h}\label{Uh1}\\
\nonumber
-({p}_h,w_1)_{\T_h}
+\beta^{\frac12}\l(\bz\cdot{\bf n}-\tau_1)\widehat{y}_h,w_1\r_{\partial \T_h}&=(f,w_1)_{\T_h},\\
\beta^{\frac12}(\nabla\cdot{\tpb}_h,w_2)_{\T_h}+({y}_h,w_2)_{\T_h}+\beta^{\frac12}({p}_h,\bz\cdot\nabla w_2)_{\T_h}\label{Uh2}\\\nonumber
+\beta^{\frac12}\l\tau_2{p}_h,w_2\r_{\partial\T_h}
-\beta^{\frac12}\l(\tau_2+\bz\cdot{\bf n})\widehat{p}_h,w_2)\r_{\partial\T_h}&=(g,w_2)_{\T_h},\\
-\beta^{\frac12}\l{\bf q}_h\cdot{\bf n},\mu_1\r_{\partial \T_h}-\beta^{\frac12}\l\tau_1{y}_h,\mu_1\r_{\partial \T_h}-\beta^{\frac12}\l(\bz\cdot{\bf n}-\tau_1)\widehat{y}_h,\mu_1)\r_{\partial\T_h}&=0,\\
-\beta^{\frac12}\l{\bf p}_h\cdot{\bf n},\mu_2\r_{\partial \T_h}-\beta^{\frac12}\l\tau_2{p}_h,\mu_2\r_{\partial \T_h}+\beta^{\frac12}\l(\bz\cdot{\bf n}+\tau_2)\widehat{p}_h,\mu_2)\r_{\partial\T_h}&=0,
\end{alignat}
\end{subequations}
for all $({\bf r}_1,{\bf r}_2,w_1,w_2,\mu_1,\mu_2)\in{\bf V}\times{\bf V}\times {W}\times {W}\times\Lambda\times\Lambda$, where $\tau_1$ are $\tau_2$ are piecewise 
local stabilization parameters, see \cite{liuzhang2025,chen2018hdg} for more details. 
Define ${\bf G}=\begin{bmatrix}
\tqb\\
\tpb
\end{bmatrix},{\bf u}=\begin{bmatrix}
y\\
p
\end{bmatrix},{\lambdae}=\begin{bmatrix}
{\ty}\\
{\tp}
\end{bmatrix}$ and ${\bf Z}_h=\begin{bmatrix}
F_h\\
G_h
\end{bmatrix}$, where $F_h=(f,w_1)_{\T_h}$ and $G_h=(f,w_2)_{\T_h}$, and 
${\tqb, \tpb}, y, p, \ty$ and $\tp$ are the unknowns associated with ${\tqb}_h, {\tpb}_h, y_h, p_h, \ty_h$  and $\tp_h$, respectively. The matrix form of system \eqref{eq:hdg2} can be written as 
\begin{align}\label{ALu}
\A_{c}\begin{bmatrix}
{\bf G}\\
\u\\
\lambdae
\end{bmatrix}=\begin{bmatrix}
{\bf 0}\\
{\bf Z}_h\\
{\bf 0}
\end{bmatrix},
\end{align}
where 
$$
\A_c=\begin{bmatrix}
\A_{\G\G} &\A^T_{\u\G} & \A^T_{\lambdae\G}\\
\A_{\u\G} &\A_{\u\u} &\A_{\u\lambdae}\\
\A_{\lambdae\G}&\A_{\lambdae\u}&\A_{\lambdae\lambdae}
\end{bmatrix}.$$
For any $({\bf q}_1,u_1, \lambda_1,{\bf p}_1,v_1, \mu_1),({\bf q}_2,u_2, \lambda_2,{\bf p}_2,v_2, \mu_2)\in {\bf V}\times W\times\Lambda\times {\bf V}\times W\times\Lambda$, 
by the definition of $\A_c$, we can define the following blinear form:
 \begin{align*}
&\cB_h\big(({\bf q}_1,u_1, \lambda_1,{\bf p}_1,v_1, \mu_1),({\bf q}_2,u_2, \lambda_2,{\bf p}_2,v_2, \mu_2)\big)\\
&=({\bf q}_2,{\bf p}_2,u_2,v_2,\lambda_2,\mu_2){\A_c}({\bf q}_1,{\bf p}_1,u_1,v_1,\lambda_1,\mu_1)^T.
\end{align*}
We can reorder the unknowns by regrouping the equations in~\eqref{eq:hdg2} as follows:
\begin{align}
&\cB_h\big(({\bf q}_1,u_1, \lambda_1,{\bf p}_1,v_1, \mu_1),({\bf q}_2,u_2, \lambda_2,{\bf p}_2,v_2, \mu_2)\big)\label{eq:hdgconcise1}\\
&=({\bf q}_2,u_2,\lambda_2,{\bf p}_2,v_2,\mu_2){\A_a}({\bf q}_1,u_1,\lambda_1,{\bf p}_1,v_1,\mu_1)^T,\notag
\end{align}
where
\begin{equation}\label{Aa}
\A_a=\begin{bmatrix}
\beta^{\frac12} {\Aadt_1}  & -{L}\\
{L} &\beta^{\frac12}  {\Aadt_2}
\end{bmatrix}.
\end{equation}
Here $\Aadt_1$ is defined as the matrix in \cite[equation (2.23)]{TZAD2021} with viscosity $\epsilon = 1$, corresponding to the advection-diffusion equation for the state variable $y$. The matrix
$ \Aadt_2$ is obtained by replacing ${\bz}$ with $-\bz$ and $\tau_1$ with $\tau_2$ in $\Aadt_1$, corresponding to the advection-diffusion equation for the dual variable $p$.
We see that ${L}$ satisfies
\begin{equation}\label{Ldef}
({\bf q}_2,u_2,\lambda_2){L}({\bf q}_1,u_1,\lambda_1)^T=(u_1,u_2)_{\T_h},
\end{equation}
for all $({\bf q}_1,u_1, \lambda_1),({\bf q}_2,u_2, \lambda_2)\in {\bf V}\times W\times\Lambda$. Moreover, let
\begin{equation}\label{Ladef}
{\L_{\bf a}}=\begin{bmatrix}{L}& {\bf 0}\\ {\bf 0}&{L}\end{bmatrix},
\end{equation} 
then 
\begin{align*}
({\bf q}_2,u_2,\lambda_2,{\bf p}_2,v_2,\mu_2){\L_{\bf a}}({\bf q}_1,u_1,\lambda_1,{\bf p}_1,v_1,\mu_1)^T=(u_1,u_2)_{\T_h}+(v_1,v_2)_{\T_h},
\end{align*}
for any $({\bf q}_1,u_1, \lambda_1,{\bf p}_1,v_1, \mu_1),({\bf q}_2,u_2, \lambda_2,{\bf p}_2,v_2, \mu_2)\in {\bf V}\times W\times\Lambda\times {\bf V}\times W\times\Lambda$.
\begin{definition}
Let \(\Lambdae= \Lambda \times \Lambda\). Given \(\lambdae \in \Lambdae\), we solve~\eqref{ALu} locally on each element \(K\) with \({\bf Z}_h = \mathbf{0}\) and obtain the solution \(({\bf Q}\lambdae, {\bf U}\lambdae)\) defined as the extensions of \(\lambdae\) from the element boundary into the element interior for the optimal control problem. Solving~\eqref{ALu} over all elements in \(\T_h\) with ${\bf Z}_h = \mathbf{0}$ yields the global solution \(({\bf Q} \lambdae, {\bf U} \lambdae)\), which together with  $\lambdae$ satisfy
\begin{equation}\label{equation:AmatrixQU}
\begin{aligned}
\begin{bmatrix}
\A_{\G\G} &\A^T_{\u\G} & \A^T_{\lambdae\G}\\
\A_{\u\G} &\A_{\u\u} &\A_{\u\lambdae}\\
\end{bmatrix}
\begin{bmatrix}
{\bf Q}\lambdae\\
{\bf U}\lambdae\\
\lambdae
\end{bmatrix}
=\begin{bmatrix}
{\bf 0}\\
{\bf 0}
\end{bmatrix}.
\end{aligned}
\end{equation}
Denote ${\bf Q}\lambdae=(Q_{A_1}\lambdae, Q_{A_2}\lambdae)^T$,${\bf U}\lambdae=(U_{A_1}\lambdae,U_{A_2}\lambdae)^T$, we define  
 \begin{align}\label{defEA}
{\bf E}_{\C}(\lambdae)
=(Q_{A_1}\lambdae,U_{A_1}\lambdae,\lambda, Q_{A_2}\lambdae,U_{A_2}\lambdae,\mu)^T.
\end{align}
\end{definition}
Given $\lambdae = (\lambda, \mu)^T, \s = (s, t)^T \in \Lambdae$, using the identities 
$
\langle \bz \cdot \mathbf{n} \lambda, s \rangle_{\partial\T_h} = 0,  
\langle \bz \cdot \mathbf{n} \mu, t \rangle_{\partial\T_h} = 0,
$ we introduce the following symmetric and nonsymmetric bilinear forms:
\begin{align}
{\bf b}_h(\lambdae,\s)&=\beta^{\frac12}({Q}_{A_1}\lambdae,Q_{A_1}\s)_{\T_h}+\beta^{\frac12}\l(\tau_1-\frac12\bz\cdot{\bf n})(U_{A_1}\lambdae-\lambda),(U_{A_1}\s-s)\r_{\partial\T_h}\label{bilinearb}\\
&\quad+\beta^{\frac12}(Q_{A_2}\lambdae,Q_{A_2}\s)_{\T_h}+\beta^{\frac12}\l(\tau_2+\frac12\bz\cdot{\bf n})(U_{A_2}\lambdae-\mu),(U_{A_2}\s-t)\r_{\partial\T_h},\nonumber\\
{\bf z}_h(\lambdae,\s)&=\frac{\beta^{\frac12}}2(U_{A_1}\s,\bz\cdot\nabla U_{A_1}{\lambdae})_{\T_h}-\frac{\beta^{\frac12}}2(U_{A_1}\lambdae,\bz\cdot\nabla U_{A_1}\s)_{\T_h}\label{bilinearz}\\
&\quad-\frac{\beta^{\frac12}}2\l\bz\cdot{\bf n}s,U_{A_1}\lambdae\r_{\partial\T_h}+\frac{\beta^{\frac12}}2
\l\bz\cdot{\bf n}\lambda,U_{A_1}\s\r_{\partial\T_h}-(U_{A_1}\s,U_{A_2}\lambdae)_{\T_h}\nonumber\\
&\quad-\frac{\beta^{\frac12}}2(U_{A_2}\s,\bz\cdot\nabla U_{A_2}{\lambdae})_{\T_h}+\frac{\beta^{\frac12}}2(U_{A_2}\lambdae,\bz\cdot\nabla U_{A_2}\s)_{\T_h}\nonumber\\
&\quad+\frac{\beta^{\frac12}}2\l\bz\cdot{\bf n}t,U_{A_2}\lambdae\r_{\partial\T_h}-\frac{\beta^{\frac12}}2
\l\bz\cdot{\bf n}\mu,U_{A_2}\s\r_{\partial\T_h(\Omega_i)}+(U_{A_1}\lambdae,U_{A_2}\s)_{\T_h}.\nonumber
\end{align}
The full bilinear form is defined by
\begin{align*}
{\bm a}_h(\lambdae,\s)&={\bf b}_h(\lambdae,\s)+{\bf z}_h(\lambdae,\s)\label{bilineara}.
\end{align*}
By the definitions of ${\bm a}_h$,~\eqref{eq:hdgconcise1} and~\eqref{defEA}, we have 
\begin{equation}
\cB_h({\bf E}_{\C}(\lambdae),{\bf E}_{\C}(\s))={\bm a}_h(\lambdae,\s),\quad\forall \lambdae,\s\in\Lambdae.
\end{equation}
In order to establish the relations between the extensions of the optimal control problem and the extensions of the advection-diffusion problem, we introduce another definition.
\begin{definition}
Given $\lambdae=(\lambda,\mu)^T\in\Lambdae$, on each element $K$, we define 
\begin{align*}
{\bf Q}_{\Aad}(\lambdae)=(Q_{A^{ad}_1}\lambda, Q_{A^{ad}_2}\mu)^T,\quad {\bf U}_{\Aad}(\lambdae)=(U_{A^{ad}_1}\mu,U_{A^{ad}_2}\mu)^T,
\end{align*}
where $(Q_{A^{ad}_1}\lambda$, $ U_{A^{ad}_1}\lambda)$ denotes the extensions from the element boundary into the element interior for the advection-diffusion problem, defined similar to~\cite[Equation (2.14)]{TZAD2021}, using the extension operators associated with $\Aadt_1$.  Here $(Q_{A^{ad}_2}\mu, U_{A^{ad}_2}\mu)$ is defined analogously by applying the extension operators associated with $\Aadt_2$. Similar as~\cite[Equation (2.14)]{TZAD2021}, we obtain that 
\begin{align*}
\Aadt_1
\begin{bmatrix}
{Q}_{A^{ad}_1}\lambda\\
{U}_{A^{ad}_1}\lambda\\
\lambda
\end{bmatrix}=\begin{bmatrix}
0\\
0
\end{bmatrix}
\quad
\mbox{and}
\quad
\Aadt_2\begin{bmatrix}
{Q}_{A^{ad}_2}\mu\\
{U}_{A^{ad}_2}\mu\\
\mu
\end{bmatrix}=\begin{bmatrix}
0\\
0
\end{bmatrix}.
\end{align*}
We also define
\begin{equation}\label{defEAd}
{\bf E}_{\A^{ad}}(\lambdae)=(Q_{A^{ad}_1}\lambda,
U_{A^{ad}_1}\lambda,\lambda,Q_{A^{ad}_2}\mu,U_{A^{ad}_2}\mu,\mu)^T.
\end{equation}
Let $\Aad=\begin{bmatrix}
{\Aadt_1}& {\bf 0}\\
{\bf 0} & {\Aadt_2}
 \end{bmatrix}$. Given $\lambdae,\s\in\Lambdae$, we define the following bilinear form 
\begin{equation}\label{cadBh}
\begin{aligned}
\cB^{\ad}_h\big({\E}_{\Aad}(\lambdae),{\E}_{\Aad}(\s)\big)
&={\bf E}^T_{\Aad}(\s) \Aad{\bf E}_{\Aad}(\lambdae).
 \end{aligned}
\end{equation}
\end{definition}

Recall the system~\eqref{ALu}, we eliminate ${\G}$ and $\u$ in each element independently and obtain a system for 
$\lambdae$ as follows:
\begin{align}\label{Ab}
\A\lambdae={\bm b},
\end{align}
where 
\begin{align*}
\A=\A_{\lambdae\lambdae}-\begin{bmatrix}
\A_{\lambdae\G}&\A_{\lambdae\u}
\end{bmatrix}\begin{bmatrix}
\A_{\G\G}&\A^T_{\u\G}\\
\A_{\u\G}&\A_{\u\u}
\end{bmatrix}^{-1}
\begin{bmatrix}
\A^T_{\lambdae\G}\\
\A_{\u\lambdae}
\end{bmatrix},
\end{align*}
and 
\begin{align*}
{\bf b}=-\begin{bmatrix}
\A_{\lambdae\G}&\A_{\lambdae\u}
\end{bmatrix}
\begin{bmatrix}
\A_{\G\G} &\A^T_{\u\G}\\
\A_{\u\G}&\A_{\u\u}
\end{bmatrix}^{-1}
\begin{bmatrix}
{\bf 0}\\
{\bf Z}_h
\end{bmatrix}.
\end{align*}
Let $\A = \B + \Z$, where 
\begin{equation}\label{BZdef}
\B= \frac{\A + \A^T}{2}, \quad \Z = \frac{\A - \A^T}{2}
\end{equation}
denote the symmetric and nonsymmetric parts of $\A$, respectively. 
For any $\lambdae,\s\in\Lambdae$, we define the following bilinear forms
\begin{equation}\label{sumhat}
\begin{aligned}
&\langle\lambdae,{\s}\rangle_{{\B}}={\s}^T{\B}\lambdae={\bf b}_h(\lambdae,\s), \quad\langle\lambdae,{\s}\rangle_{{\Z}}={\s}^T{\Z}\lambdae={\bf z}_h(\lambdae,\s).
\end{aligned}
\end{equation}
In addition, we define the matrix $\L$ by 
\begin{equation}\label{defL}
\langle\lambdae,{\s}\rangle_{{\L}}={\s}^T{\L}{\lambdae}=(U_{A_1}\lambdae,U_{A_1}\s)_{\T_h}+(U_{A_2}\lambdae,U_{A_2}\s)_{\T_h}.
\end{equation}
\section{The BDDC algorithm}\label{sec:bddc}
\subsection{Domain decomposition and a reduced subdomain interface problem}
We decompose the domain $\Omega$ into $N$ nonoverlapping subdomains $\Omega_i(i=1,2,\cdots,N)$ and the diameter of each subdomain is $H_i$. Let $H=\max\limits_{i}H_i$ and $\Gamma=\left( \bigcup_{i \neq j} (\partial \Omega_{i} \cap \partial\Omega_{j}) \right) \setminus \partial \Omega$ be the subdomain interface.
The spaces of finite element functions on $\Omega_i$ are denoted by ${\bf V}^{(i)}$, $W^{(i)}$, and $\Lambda^{(i)}$.

Given $\lambdae^{(i)}=(\lambda^{(i)},\mu^{(i)})^T,\s^{(i)}=(s^{(i)},t^{(i)})^T\in\Lambdae^{(i)}$, we define the following bilinear forms on each subdomain $\Omega_i$:
\begin{align}
{\bf b}_h(\lambdae^{(i)}, \s^{(i)})
&= \beta^{\frac{1}{2}} (Q_{A_1} \lambdae^{(i)}, Q_{A_1} \s^{(i)})_{\T_h(\Omega_i)} \label{bh1}\\
&\quad + \beta^{\frac{1}{2}} \langle (\tau_1 - \tfrac{1}{2} \bz \cdot {\bf n})(U_{A_1} \lambdae^{(i)} - \lambda^{(i)}), (U_{A_1} \s^{(i)} - s^{(i)}) \rangle_{\partial\T_h(\Omega_i)} \notag\\
&\quad + \beta^{\frac{1}{2}} (Q_{A_2} \lambdae^{(i)}, Q_{A_2} \s^{(i)})_{\T_h(\Omega_i)} \notag\\
&\quad + \beta^{\frac{1}{2}} \langle (\tau_2 + \tfrac{1}{2} \bz \cdot {\bf n})(U_{A_2} \lambdae^{(i)} - \mu^{(i)}), (U_{A_2} \s^{(i)} - t^{(i)}) \rangle_{\partial \T_h(\Omega_i)},\notag\\
{\bf z}_h(\lambdae^{(i)}, \s^{(i)}) 
&= \frac{\beta^{\frac{1}{2}}}{2} (U_{A_1} \s^{(i)}, \bz \cdot \nabla U_{A_1} \lambdae^{(i)})_{\T_h(\Omega_i)}
- \frac{\beta^{\frac{1}{2}}}{2} (U_{A_1} \lambdae^{(i)}, \bz \cdot \nabla U_{A_1} \s^{(i)})_{\T_h(\Omega_i)} \label{zh2}\\
&\quad - \frac{\beta^{\frac{1}{2}}}{2} \langle \bz \cdot {\bf n} \, s^{(i)}, U_{A_1} \lambdae^{(i)} \rangle_{\partial \T_h(\Omega_i)}
+ \frac{\beta^{\frac{1}{2}}}{2} \langle \bz \cdot {\bf n} \, \lambda^{(i)}, U_{A_1} \s^{(i)} \rangle_{\partial \T_h(\Omega_i)} \notag\\
&\quad - (U_{A_1} \s^{(i)}, U_{A_2} \lambdae^{(i)})_{\T_h(\Omega_i)}
- \frac{\beta^{\frac{1}{2}}}{2} (U_{A_2} \s^{(i)}, \bz \cdot \nabla U_{A_2} \lambdae^{(i)})_{\T_h(\Omega_i)} \notag\\
&\quad + \frac{\beta^{\frac{1}{2}}}{2} (U_{A_2} \lambdae^{(i)}, \bz \cdot \nabla U_{A_2} \s^{(i)})_{\T_h(\Omega_i)}
+ \frac{\beta^{\frac{1}{2}}}{2} \langle \bz \cdot {\bf n} \, t^{(i)}, U_{A_2} \lambdae^{(i)} \rangle_{\partial \T_h(\Omega_i)} \notag\\
&\quad - \frac{\beta^{\frac{1}{2}}}{2} \langle \bz \cdot {\bf n} \, \mu^{(i)}, U_{A_2} \s^{(i)} \rangle_{\partial \T_h(\Omega_i)}
+ (U_{A_1} \lambdae^{(i)}, U_{A_2} \s^{(i)})_{\T_h(\Omega_i)},\notag\\
{\bm a}_h(\lambdae^{(i)}, \s^{(i)})&={\bf b}_h(\lambdae^{(i)}, \s^{(i)}) +{\bf z}_h(\lambdae^{(i)}, \s^{(i)}).\label{ah4}
\end{align}
The local bilinear form $\cB^{(i)}_h$ is defined by
\begin{align*}
\cB^{(i)}_h \big({\bf E}_{\C}(\lambdae^{(i)}),{\bf E}_{\C}(\s^{(i)})\big)={\bm a}_h(\lambdae^{(i)}, \s^{(i)})-\frac{1}{2} \langle {\bm \zeta} \cdot {\bf n} \, \lambda^{(i)}, s^{(i)} \rangle_{\partial \T_h(\Omega_i)}
+\frac{1}{2} \langle {\bm \zeta} \cdot {\bf n} \, \mu^{(i)}, t^{(i)} \rangle_{\partial \T_h(\Omega_i)}, 
\end{align*}
which is obtained by~\eqref{eq:hdgconcise1} and restricting $\Aa$ to the subdomain $\Omega_i$. 

To make the subdomain local problem positive definite, we introduce the Robin boundary conditions and let
\begin{equation}\label{cBi}
\begin{aligned}
&\cB^{(i)}\big({\bf E}_{\C}(\lambdae^{(i)}),{\bf E}_{\C}(\s^{(i)})\big) \\
&= \cB^{(i)}_h \big({\bf E}_{\C}(\lambdae^{(i)}),{\bf E}_{\C}(\s^{(i)})\big) 
+ \frac{1}{2} \langle {\bm \zeta} \cdot {\bf n} \, \lambda^{(i)}, s^{(i)} \rangle_{\partial \T_h(\Omega_i)}
- \frac{1}{2} \langle {\bm \zeta} \cdot {\bf n} \, \mu^{(i)}, t^{(i)} \rangle_{\partial \T_h(\Omega_i)} \\
&= {\bf E}_{\C}^T(\s^{(i)}) \Aaoi {\bf E}_{\C}(\lambdae^{(i)}) 
= {\bf E}_{\C}^T(\s^{(i)}) 
\begin{bmatrix}
\beta^{\frac{1}{2}} {A^{ad}_1}^{(i)}&-{L^{(i)}}\\
L^{(i)} & \beta^{\frac{1}{2}}{A^{ad}_2}^{(i)}
\end{bmatrix}
{\bf E}_{\C}(\lambdae^{(i)})
\end{aligned}
\end{equation}
where $\Aai, {A^{ad}_1}^{(i)}, {A^{ad}_2}^{(i)}$ denote the subdomain matrices of $\Aa$, $\Aadt_1$, and $\Aadt_2$ in \eqref{Aa}, restricted to the subdomain $\Omega_i$ with Robin boundary conditions, respectively. We then have
\begin{align}\label{Bhai}
\cB^{(i)}({\bf E}_{\C}(\lambdae^{(i)}),{\bf E}_{\C}(\s^{(i)}))={\bm a}^{(i)}_h(\lambdae^{(i)},\s^{(i)}),\quad\forall \lambdae^{(i)},\s^{(i)}\in\Lambdae^{(i)}.
\end{align}

To reduce the global problem in ~\eqref{Ab} to a subdomain interface problem, we perform the following decomposition
\begin{align*}
\Lambdae=\bigg(\bigoplus_{i=1}^N \Lambdae^{(i)}_{I}\bigg)\bigoplus \widehat{\Lambdae}_\Gamma,
\end{align*}
where $\widehat{\Lambdae}_\Gamma$ denotes the degrees of freedom on the subdomain interface and $\Lambdae^{(i)}_I$ denotes the  degrees of freedom in the interior of subdomain  $\Omega_i$ .

Given $\lambdae_\Gamma\in\widehat{\Lambdae}_\Gamma,$ the original global problem \eqref{Ab} can be decomposed as 
\begin{align}\label{AIGamma}
\begin{bmatrix}
{\A}_{II}& {\A}_{I\Gamma}\\
{\A}_{\Gamma I} &{\A}_{\Gamma\Gamma}
\end{bmatrix}
\begin{bmatrix}
\lambdae_I\\
\lambdae_\Gamma
\end{bmatrix}=\begin{bmatrix}
{\bf b}_I\\
{\bf b}_\Gamma
\end{bmatrix}.
\end{align}
Therefore, for each subdomain $\Omega_i$, the subdomain problem can be written as 
\begin{align}\label{Asubdomain}
\begin{bmatrix}
{\A}^{(i)}_{II}& {\A}^{(i)}_{I\Gamma}\\
{\A}^{(i)}_{\Gamma I} &{\A}^{(i)}_{\Gamma\Gamma}
\end{bmatrix}
\begin{bmatrix}
\lambdae^{(i)}_I\\
\lambdae^{(i)}_\Gamma
\end{bmatrix}=\begin{bmatrix}
{\bf b}^{(i)}_I\\
{\bf b}^{(i)}_\Gamma
\end{bmatrix}.
\end{align}
The subdomain local Schur complement  ${\bf S}^{(i)}_\Gamma$  can be defined as:
\begin{equation}\label{equationSi_gamma}
{\S}^{(i)}_\Gamma\lambdae^{(i)}_\Gamma={\g}^{(i)}_\Gamma,
\end{equation}
where
\begin{align*}
{\S}^{(i)}_\Gamma&={\A}^{(i)}_{\Gamma\Gamma}-{\A}^{(i)}_{\Gamma I}{\A}^{(i)^{-1}}_{II}{\A}^{(i)}_{I\Gamma},\\
{\g}^{(i)}_\Gamma&={\bf b}^{(i)}_\Gamma-{\A}^{(i)}_{\Gamma I}{\A}^{(i)^{-1}}_{II}{\bf b}^{(i)}_I.
\end{align*}
Denote $\R^{(i)}_\Gamma$ as the restriction operator from $\widehat\Lambdae_\Gamma$ to the subdomain interface ${\Lambdae}^{(i)}_\Gamma$.
By assembling the subdomain local Schur complement ${\S}^{(i)}_\Gamma$, we obtain the global Schur interface problem:
find  $\lambdae\in\widehat{\Lambdae}_{\Gamma}$ such that 
\begin{equation}\label{Shat_Gamma}
\widehat{\S}_\Gamma\lambdae_\Gamma={\bf g}_\Gamma,
\end{equation}
where 
\begin{align*}
\widehat{\S}_\Gamma=\sum\limits_{i=1}^N {\bf R}^{(i)^T}_{\Gamma}{\S}^{(i)}_\Gamma {\bf R}^{(i)}_\Gamma,\qquad {\g}_\Gamma=\sum\limits_{i=1}^N{\bf R}^{(i)^T}_\Gamma{\g}^{(i)}_\Gamma.
\end{align*}
\subsection{The BDDC precondtioner}
We define the partially assembled interface space $\widetilde{\bm\Lambda}_\Gamma$ as
\begin{align*}
\widetilde{\bm\Lambda}_\Gamma = \widehat{\bm\Lambda}_{\Pi} \bigoplus \bm\Lambda_{\Delta} = \widehat{\bm\Lambda}_{\Pi} \bigoplus \left( \prod_{i=1}^N \bm\Lambda^{(i)}_{\Delta} \right),
\end{align*}
where $\widehat{\bm\Lambda}_{\Pi}$ denotes the space of primal variables, which are typically continuous across subdomain interfaces. The space $\bm\Lambda_{\Delta}$, given by the direct sum of the local spaces $\Lambdae^{(i)}_{\Delta}$, consists of the remaining interface degrees of freedom, which are typically discontinuous across subdomain interfaces.

Let $\widetilde{\R}_\Gamma$ denote the injection operator from the space $\widehat\Lambdae_\Gamma$ into $\widetilde{\Lambdae}_\Gamma$. We also define
$
\widetilde{\R}_{\D,\Gamma} = \D \widetilde{\R}_\Gamma,
$
where $\D$ is a diagonal scaling matrix. The diagonal entries of $\D$ are set to 1 for rows corresponding to primal interface variables, and to $\delta_i^\dagger(x)$ for the others.
Here, $\delta_i^\dagger(x)$ is the inverse counting function defined for a subdomain interface node $x$ as:
\[
\delta^\dag_i(x) = \frac{1}{\text{card}(I_x)}, \quad \text{for } x \in \partial \Omega_i \cap \Gamma,
\]
where $N_x$ is the set of subdomain indices that contain $x$ on their boundaries. The function $\text{card}(I_x)$ gives the number of such subdomains. The scaling matrix provides a partition of unity such that
\begin{equation}\label{scalingD}
\widetilde{\R}^T_{\D,\Gamma}\widetilde{\R}_{\Gamma}=\widetilde{\R}^T_{\Gamma}\widetilde{\R}_{\D,\Gamma}={\bf I}.
\end{equation}
Various scaling strategies can be used in BDDC algorithms. Among them, deluxe scaling is particularly useful for problems with discontinuous coefficients, as it enhances the robustness of the preconditioner in the presence of parameter jumps in the model~\cite{ZT:2016:Darcy, Widlund:2020:BDDC, Deluxe, WD:2016:deluxe}.

Let $\overline{\R}_\Gamma$  be the direct sum of ${\R}^{(i)}_{\Gamma}$, which is the restriction operator from $\widetilde{\Lambdae}_\Gamma$ to $\Lambdae^{(i)}$.
Let ${\S}_\Gamma$ be the direct sum of the subdomain local Schur complement  $\S^{(i)}_\Gamma$ defined in equation \eqref{equationSi_gamma}. The partially assembled interface Schur complement is defined as
\begin{align}\label{Sgamma}
\widetilde{\S}_\Gamma=\overline{\R}_\Gamma^T{\S}_\Gamma\overline{\R}_\Gamma.
\end{align}
With \eqref{Sgamma}, we can define BDDC preconditioner as  
\begin{align*}
{\bf P}^{-1}_{\bf BDDC}=\widetilde{\R}^T_{\D,\Gamma}\widetilde{\S}^{-1}_\Gamma
\widetilde{\R}_{\D,\Gamma}.
\end{align*}
Multiplying the BDDC preconditioner to the global interface problem~\eqref{Shat_Gamma} both sides, we have
\begin{align}\label{bddc}
\widetilde{\R}^T_{\D,\Gamma}\widetilde{\S}^{-1}_\Gamma\widetilde{\R}_{\D,\Gamma}\widehat{\S}_\Gamma\lambdae_\Gamma=\widetilde{\R}^T_{\D,\Gamma}\widetilde{\S}^{-1}_\Gamma\widetilde{\R}_{\D,\Gamma}{\g}_\Gamma.
\end{align}
Since $\widehat{\S}_\Gamma$ defined in~\eqref{Shat_Gamma} is nonsymmetric but positive definite, we employ the GMRES solver to solve~\eqref{bddc}.
\section{Extensions and Norms}\label{sec:extennorms}
\subsection{Subdomain extensions}
Given $\lambdae^{(i)}_\Gamma=(\lambda^{(i)}_\Gamma,\mu^{(i)}_\Gamma)^T\in\Lambdae^{(i)}_\Gamma$, by solving the subdomain problem \eqref{Asubdomain} with ${\bf b}^{(i)}_I={\bf 0}$, we obtain  the subdomain interior $\lambdae^{(i)}_I\in\Omega_i$ as 
$
\lambdae^{(i)}_{I}=
-{{\A}^{(i)}_{II}}^{-1}{{\A}^{(i)}_{I\Gamma}}{\lambdae}^{(i)}_\Gamma.
$
Let $(\lambda^{(i)}_I,\mu^{(i)}_I)^T=\lambdae^{(i)}_I,$  we denote
$\lambda^{(i)}_{A,\Gamma}=(\lambda^{(i)}_I,\lambda^{(i)}_\Gamma)^T, \mu^{(i)}_{A,\Gamma}=(\mu^{(i)}_I,\mu^{(i)}_{\Gamma})^T$ and define
$\lambdae^{(i)}_{\A,\Gamma}=(\lambda^{(i)}_{A,\Gamma},\mu^{(i)}_{A,\Gamma})$.
Similarly, given $\lambdae_\Gamma\in\widehat{\Lambdae}_\Gamma$ or $\widetilde{\Lambdae}_\Gamma$, we obtain $\lambdae_I$ as
$\lambdae_{I}=
-{{\A}^{-1}_{II}}{{\A}_{I\Gamma}}{\lambdae}_\Gamma$
and denote
$\lambda_{A,\Gamma}=(\lambda_I,\lambda_{\Gamma})^T, \mu_{A,\Gamma}=(\mu_I,\mu_{\Gamma})^T$, where $(\lambda_I,\mu_I)^T=\lambdae_I$. 
Let 
\begin{equation}\label{lambdaIdef}
\lambdae_{\A,\Gamma}=(\lambda_{A,\Gamma},\mu_{A,\Gamma})^T
\end{equation}
denote the optimal control extension of $\lambdae_\Gamma$. We can also define the advection-diffusion extension of $\lambdae_\Gamma$ as
\begin{equation}\label{lambdaad}
\lambdae_{\A^{ad},\Gamma}=(\lambda_{{A^{ad}_1},\Gamma},\mu_{{A^{ad}_2},\Gamma})^T,
\end{equation} 
where $\lambda_{{A^{ad}_1},\Gamma}$ is the advection-diffusion extension of $\lambda_\Gamma$ defined similarly as \cite[Equation (5.1)]{TZAD2021} corresponding to $A^{ad}_1$ in~\eqref{Aa}, $\mu_{A^{ad}_2,\Gamma}$ can be obtained similarly as $\lambda_{A^{ad}_1,\Gamma}$ corresponding to $A^{ad}_2$ in~\eqref{Aa}. Let
$\lambdae^{(i)}_{\A^{ad},\Gamma}=(\lambda^{(i)}_{{A^{ad}_1},\Gamma},\mu^{(i)}_{{A^{ad}_2},\Gamma})^T$, 
where $\lambda^{(i)}_{{A^{ad}_1},\Gamma}$, $\mu^{(i)}_{{A^{ad}_2},\Gamma}$ are the restrictions of  $\lambda_{{A^{ad}_1},\Gamma}$  and $ \mu_{{A^{ad}_2},\Gamma}$ in subdomain $\Omega_i$, respectively.

\subsection{A partial assembled finite element space and norms}
We define a partially sub-assembled space $\widetilde\Lambdae$ as
\begin{align*}
\widetilde\Lambdae = \widehat\Lambdae_{\Pi} \bigoplus \left( \prod_{i=1}^N \left( \Lambdae^{(i)}_I \bigoplus \Lambdae^{(i)}_\Delta \right) \right).
\end{align*}
The functions in the space $\widetilde\Lambdae$ consist of a continuous primal component and a generally discontinuous dual component. 
Moreover, we have $\Lambdae \subset \widetilde\Lambdae$.

Let $\widetilde{\B},\widetilde{\Z},\widetilde\A$ be the partially sub-assembled matrices corresponding to the bilinear forms in \eqref{bh1}, \eqref{zh2} and~\eqref{ah4}, respectively.
Recall $\L$ defined in~\eqref{defL}, let $\widetilde\L$ be the partially sub-assembled matrix corresponding to $\L$ and $\widetilde{\R}$ be the injection vector from $\Lambdae$ to $\widetilde\Lambdae$. Then, we have  $
\B=\widetilde\R^T\widetilde\B\widetilde\R,\Z=\widetilde\R^T\widetilde\Z\widetilde\R,{\A}={\widetilde\R}^T\widetilde\A\widetilde\R,{\L}=\widetilde\R^T\widetilde\L\widetilde\R$,
where $\B,\Z,\A$ are defined in~\eqref{BZdef} and~\eqref{Ab}.
$\forall \lambdae,{\s}\in\widetilde{\Lambdae}$, we define the following {bilinear forms}
\begin{equation}\label{sumoflocalbl}
\begin{aligned}
&\langle\lambdae,{\s}\rangle_{\widetilde{\B}}={\s}^T\widetilde{\B}\lambdae=\sum\limits_{i=1}^N{\b}_h(\lambdae^{(i)}, \s^{(i)}), \quad\langle\lambdae,{\s}\rangle_{\widetilde{\bm Z}}={\s}^T\widetilde{\bm Z}\lambdae=\sum\limits_{i=1}^N{\z}_h(\lambdae^{(i)}, \s^{(i)}),\\
&\langle\lambdae,{\s}\rangle_{\widetilde{\A}}=\s^T\widetilde\A\lambdae=\sum\limits_{i=1}^N{\a}_h(\lambdae^{(i)}, \s^{(i)}),\\ 
&{{\langle\lambdae,{\s}\rangle_{\widetilde{\L}}={\s}^T\widetilde{\L}{\lambdae}=(U_{A_1}\lambdae,U_{A_1}\s)_{\T_h}+(U_{A_2}\lambdae,U_{A_2}\s)_{\T_h}.}}
\end{aligned}
\end{equation}

Given the definition $\cB^{(i)}$ in \eqref{cBi} in each subdomain, for any $\lambdae,\s\in\widetilde\Lambdae$, we can define the following bilinear form 
\begin{align}
\widetilde\cB\big({\bf E}_{\C}(\lambdae);{\bf E}_{\C}(\s)\big)&=\sum\limits_{i=1}^N\cB^{(i)}\big({\bf E}_{\C}(\lambdae^{(i)});{\bf E}_{\C}(\s^{(i)})\big)={\bf E}^T_{\C}(\s)\Aa {\bf E}_{\C}(\lambdae)\label{cB}\\
&= {\bf E}^T_{\C}(\s) \begin{bmatrix}
\beta^{\frac12}{\Aadt_1}& -{L}\\
 {L} & \beta^{\frac12}{\Aadt_2}
 \end{bmatrix}{\bf E}_{\C}(\lambdae),\nonumber\\
\widetilde\cB^{\ad}\big({\bf E}_{\C}(\lambdae);{\bf E}_{\C}(\s)\big)
&={\bf E}^T_{\C}(\s)\Aad{\bf E}_{\C}(\lambdae),\label{cBad}
 \end{align}
where $\Aa, \Aadt_1, \Aadt_2$ are obtained by assembling the local matrices $\Aai,{\Aadt_1}^{(i)},{\Aadt_2}^{(i)}$ in~\eqref{cBi} on each subdomain respectively.
Here, the matrices $\Aa, \Aadt_1, \Aadt_2,\Aad$ have been adjusted by the Robin boundary conditions.
 
Let
\begin{equation}\label{Mdeft}
\M=\B+\L,\quad\widetilde\M=\widetilde\B+\widetilde\L,
\end{equation}
by using the optimal control extension in~\eqref{lambdaIdef}, we can define the following bilinear forms
\begin{align}
\forall \lambdae_\Gamma,{\s}_\Gamma\in\widehat{\Lambdae}_\Gamma,\quad&\langle\lambdae_\Gamma,{\s}_\Gamma\rangle_{{\B}_\Gamma}
= \s^T_{\A,\Gamma} {\bf B} {\lambdae}_{\A,\Gamma},\langle\lambdae_\Gamma,{\s}_\Gamma\rangle_{{\bf Z}_\Gamma}
= \s^T_{\A,\Gamma} {\bf Z}{\lambdae}_{\A,\Gamma},\label{lambdauBhat} \\
&\langle\lambdae_\Gamma,{\s}_\Gamma\rangle_{{\bf L}_\Gamma}
= {\s}^T_{\A,\Gamma}{\L}\lambdae_{\A,\Gamma},
\langle\lambdae_\Gamma,{\s}_\Gamma\rangle_{{\bf M}_\Gamma}
= {\s}^T_{\A,\Gamma}{\M}\lambdae_{\A,\Gamma};\notag\\
\forall \lambdae_\Gamma,{\s}_\Gamma\in\widetilde{\Lambdae}_\Gamma,\quad&\langle\lambdae_\Gamma,{\s}_\Gamma\rangle_{\widetilde{\B}_\Gamma}
= \s^T_{\A,\Gamma}\widetilde{\bf B} {\lambdae}_{\A,\Gamma},\langle\lambdae_\Gamma,{\s}_\Gamma\rangle_{\widetilde{\bf Z}_\Gamma}
= \s^T_{\A,\Gamma}\widetilde{\bf Z}{\lambdae}_{\A,\Gamma},\label{lambdauBwilde} \\
&\langle\lambdae_\Gamma,{\s}_\Gamma\rangle_{\widetilde{\bf L}_\Gamma}
= {\s}^T_{\A,\Gamma}\widetilde{\L}\lambdae_{\A,\Gamma},\langle\lambdae_\Gamma,{\s}_\Gamma\rangle_{\widetilde{\bf M}_\Gamma}
= {\s}^T_{\A,\Gamma}\widetilde{\M}\lambdae_{\A,\Gamma}.\notag
\end{align}
By~\eqref{sumhat},~\eqref{defL},~\eqref{Bhai},~\eqref{sumoflocalbl},~\eqref{cB}, ~\eqref{Mdeft},~\eqref{lambdauBhat} and~\eqref{lambdauBwilde}, we obtain the following lemma.
\begin{lemma}\label{MSrelation}
Given $\lambdae_\Gamma,{\s}_\Gamma\in\widetilde{\Lambdae}_\Gamma$, we have the following relations
\begin{align}
\langle\lambdae_\Gamma,{\s}_\Gamma\rangle_{\widetilde{\S}_\Gamma}&=\l\lambdae_{\A,\Gamma},\s_{\A,\Gamma}\r_{\widetilde\A}=
\langle\lambdae_\Gamma,{\s}_\Gamma\rangle_{\widetilde{\B}_\Gamma}+
\langle\lambdae_\Gamma,{\s}_\Gamma\rangle_{\widetilde{\Z}_\Gamma}.\label{eqlambdaAu}
\end{align}
When $\lambdae_\Gamma={\s}_\Gamma$,
$
\langle\lambdae_\Gamma,\lambdae_\Gamma\rangle_{\widetilde{\bf Z}_\Gamma}=0,
\langle\lambdae_\Gamma,\lambdae_\Gamma\rangle_{\widetilde{\bf S}_\Gamma}=
\l\lambdae_{\A,\Gamma},\lambdae_{\A,\Gamma}\r_{\widetilde{\A}}
={\bf E}^T_{\C}(\lambdae_{\A,\Gamma})\Aa{\bf E}_{\C}(\lambdae_{\A,\Gamma})=\l\lambdae_{\A,\Gamma},\lambdae_{\A,\Gamma}\r_{\widetilde{\B}}=\langle\lambdae_\Gamma,\lambdae_\Gamma\rangle_{\widetilde{\bf B}_\Gamma},
$
and
$\langle\lambdae_\Gamma,\lambdae_\Gamma\rangle_{\widetilde{\bf M}_\Gamma}=\langle\lambdae_\Gamma,\lambdae_\Gamma\rangle_{\widetilde{\bf B}_\Gamma}+\langle\lambdae_\Gamma,\lambdae_\Gamma\rangle_{\widetilde{\bf L}_\Gamma}
=\langle\lambdae_\Gamma,\lambdae_\Gamma\rangle_{\widetilde{\bf S}_\Gamma}+\langle\lambdae_\Gamma,\lambdae_\Gamma\rangle_{\widetilde{\bf L}_\Gamma}.
$
The same results also hold for $\lambdae,\s\in \widehat{\Lambdae}_\Gamma$ and their corresponding bilinear forms.
\end{lemma}
\subsection{Norms}
Recall $\Aa$ defined in~\eqref{Aa} and $\La$ defined in~\eqref{Ladef}. Let $$\Ma=\Aa+\La.$$
For any $ \lambdae = (\lambda, \mu)^T \in \widetilde\Lambdae $, we can define the following norms
\begin{align}
\|{\bf E}_{\C}(\lambdae)\|^2_{\Aa}&={\bf E}^T_{\C}(\lambdae)\Aa{\bf E}_{\C}(\lambdae)\label{defAanorm}\\
&= \beta^{\frac12}\|Q_{A_1}\lambdae\|^2_{\T_h}+\beta^{\frac12}\big\||\tau_1-\frac12{\bf\zeta}\cdot{\bf n}|^{\frac12}(U_{A_1}\lambdae-\lambda)\big\|^2_{\partial\T_h}\notag\\
&\quad+\beta^{\frac12}\|Q_{A_2}\lambdae\|^2_{\T_h}+ \beta^{\frac12}\big\||\tau_2+\frac12{\bf\zeta}\cdot{\bf n}|^{\frac12}(U_{A_2}\lambdae-\mu)\big\|^2_{\partial\T_h},\notag \\
\|{\bf U}\lambdae\|^2_{\T_h}&=\|{U}_{A_1}\lambdae\|^2_{\T_h}+\|{U}_{A_2}\lambdae\|^2_{\T_h}={\bf E}^T_{\C}(\lambdae)\La{\bf E}_{\C}(\lambdae),\label{eqUnorm}\\
\|{\bf E}_{\C}(\lambdae)\|^2_{\Ma}&={\bf E}^T_{\C}(\lambdae)\Ma{\bf E}_{\C}(\lambdae)=\|{\bf E}_{\C}(\lambdae)\|^2_{\Aa}+\|{\bf U}\lambdae\|^2_{\T_h}.
\label{eq:Manorm}
\end{align}
By Lemma~\ref{MSrelation} and ~\eqref{eq:Manorm}, we have
\begin{align}\label{MaMGamma}
\forall\lambdae_\Gamma\in\widetilde\Lambdae_\Gamma, \|\lambdae_\Gamma\|^2_{\widetilde\M_\Gamma}=\|{\bf E}_{\C}(\lambdae_{\A,\Gamma})\|^2_{\Ma},\quad
\forall\lambdae_\Gamma\in\widehat\Lambdae_\Gamma, \|\lambdae_\Gamma\|^2_{\M_\Gamma}=\|{\bf E}_{\C}(\lambdae_{\A,\Gamma})\|^2_{\Ma}.
\end{align}
Additionally, by~\eqref{cBad}, we define the norms associated with ${\bf E}_{\Aad}(\lambdae)$ defined in~\eqref{defEAd}  as follows:
\begin{align}
\|{\bf E}_{\A^{\bf ad}}(\lambdae)\|^2_{\Aad}&={\bf E}^T_{\Aad}(\lambdae)\Aad{\bf E}_{\Aad}(\lambdae)\label{Madefinition}\\
&= \|Q_{A^{ad}_1}\lambda\|^2_{\T_h}+\||\tau_1-\frac12{\bf\zeta}\cdot{\bf n}|^{\frac12}(U_{A^{ad}_1}\lambda-\lambda)\|^2_{\partial\T_h}\notag\\
&\quad+\|Q_{A^{ad}_2}\mu\|^2_{\T_h}+\beta^{\frac12}\big\||\tau_2+\frac12{\bf\zeta}\cdot{\bf n}|^{\frac12}(U_{A^{ad}_2}\mu-\mu)\big\|^2_{\partial\T_h},\notag\\
\|{\bf U}_{\Aad}(\lambdae)\|^2_{\T_h}&=\|{U}_{A^{ad}_1}\lambda\|^2_{\T_h}+\|{U}_{A^{ad}_2}\mu\|^2_{\T_h}.
\end{align}
Moreover,  we introduce several useful norms as defined in \cite{TZAD2021}, for any domain $D$ and $\lambdae=(\lambda,\mu)^T\in\widetilde\Lambdae(D)$, suppose $h$ is the diameter of $D$, then 

\begin{equation}
\|\lambda\|^2_{h,D} := \sum\limits_{K \in \T_h(D)} \|\lambda\|^2_{\partial K}\frac{|D|}{|\partial D|},
\label{eq:norm1}
\end{equation}
\begin{equation}
\interleave\lambda\interleave^2_{D} := \sum\limits_{K \in \T_h(D)} \|\lambda - m_K(\lambda)\|^2_{\partial K} \left(\frac{|D|}{|\partial D|}\right)^{-1},\text{ where } m_K(\lambda)=\frac1{|\partial K|}\int_{\partial K}\lambda ds.
\label{eq:norm2}
\end{equation}
For the domain $\Omega$, we simplify the notation by omitting the subscript.  Specifically, we use $\|\cdot\|_h$ and $\interleave\cdot\interleave$ to denote $\|\cdot\|_{h,\Omega}$ and $\interleave\cdot\interleave_{\Omega}$ respectively. Note that  in this paper, the norm of $\lambdae=(\lambda,\mu)^T$, $\|\lambdae\|_{h,D}$, $\interleave\lambdae\interleave_{h,D}$
and  $\|\lambdae\|_{\partial D}$, are defined as follows: 
\begin{align*}
\|\lambdae\|_{h,D}:=(\|\lambda\|^2_{h,D}+\|\mu\|^2_{h,D})^{\frac12},\quad \interleave\lambdae \interleave_{D}:=( \interleave\lambda \interleave^2_{D}+ \interleave\mu \interleave^2_{D})^{\frac12},\quad \|\lambdae\|_{\partial D}:=(\|\lambda\|^2_{\partial D}+\|\mu\|^2_{\partial D})^{\frac12}.
\end{align*}
\subsection{Some norm estimates}
\begin{lemma}\cite[Lemma 7.8]{TZAD2021}\label{citeadlemma7.8}
If $\mu \in \Lambda^{(i)}$ and $\mu$ is edge average zero on $\partial \Omega_i$, then
\[
\|\mu\|_{h,\Omega_i} \leq C H \interleave \mu \interleave_{\Omega_i}.
\]
\end{lemma}
\begin{lemma}\cite[Lemma 7.10]{TZAD2021}\label{citeadlemma7.10}
If $\mu \in \widetilde{\Lambda}$, then
\[
\|\mu\|_h \leq C \, \interleave \mu \interleave.
\]
\end{lemma}
From the definitions of $\|\cdot\|_h$ and $\interleave\cdot\interleave$, we see that 
\begin{lemma}\label{normtriplebarhnorm}
If $\mu \in \widetilde{\Lambda}$, then
\[
\interleave\mu\interleave\leq Ch^{-1}\|\mu\|_{h}.
\]
\end{lemma}

For any $\lambdae=(\lambda,\mu)^T\in\widetilde\Lambdae$, by 
\cite[Lemma 7.2]{TZAD2021} with $\epsilon=1$, 
 we have 
{
\begin{equation}\label{UA1L2}
\begin{aligned}
&\|U_{A^{ad}_1}\lambda\|_{K} \leq C\|\lambda\|_{h,K}.
\end{aligned}
\end{equation}
 Replacing $\bz$ with $-\bz$, $\tau_1$ with $\tau_2$ and by similar proofs as \cite[lemma 7.2]{TZAD2021}, we obtain
\begin{equation}\label{UA2L2}
\begin{aligned}
&\|U_{A^{ad}_2}\mu\|_{K}\leq C\|\mu\|_{h,K}.
\end{aligned}
\end{equation}}
Also, for $\lambdae\in\widetilde\Lambdae$,
by \cite[Lemma 7.6]{TZAD2021} we have
\begin{equation}\label{ADestimateE}
c\interleave\lambdae\interleave\leq \|{\bf E}_{\A^{\bf ad}}(\lambdae)\|_{\Aad}\leq C\interleave\lambdae\interleave,
\end{equation}
where $c$ and $C$ are two constants such that $c\leq C$.

By~\cite[Lemma 7.3, Lemma 7.4]{TZAD2021} with $\epsilon=1$, we obtain the following lemma.
\begin{lemma}
For $\lambda$ defined on $\partial K$, we have 
\begin{align}
&\|\nabla U_{A^{ad}_1}\lambda\|_K\leq C\interleave\lambda\interleave_K, \|\nabla U_{A^{ad}_2}\lambda\|_K\leq C\interleave\lambda\interleave_K\label{gradientUestimate},\\
&\|U_{A^{ad}_1}\lambda-\lambda\|_{\partial K}\leq Ch^{\frac12}\interleave\lambda\interleave_K, 
\|U_{A^{ad}_2}\lambda-\lambda\|_{\partial K}\leq Ch^{\frac12}\interleave\lambda\interleave_K.\label{Uminuslambdaestimate}
\end{align}
\end{lemma}

\section{The Convergence rate of the BDDC preconditioned system}\label{sec:conv}
In this section, we examine the dependence of the GMRES convergence rate on the regularization parameter $\beta$, the subdomain size $H$ and the mesh size $h$. Throughout this section, we use $C$ to denote any constants independent of these parameters. 
{{Following~\cite[A(1)-A(3)]{chen2018hdg} and~\cite[Assumption 6.1]{TZAD2021}, we make the assumptions below.}}
\begin{assumption}\label{eq:assump}
For the stabilizers $\tau_1$ and $\tau_2$, we assume
\begin{enumerate}
\item\label{assump1} $\tau_1$ is a piecewise positive constant on $\partial\T_h$ and there exists a constant $C_1$ such that $\tau_1\leq C_1$;
\item\label{assump2} $\tau_1=\tau_2+\bm{\zeta}\cdot{\bf n}$;
\item\label{assump3} $\max\limits_{\mathcal{E}\subset\partial K}\inf\limits_{x\in\mathcal{E}}(\tau_1-\frac12\bm{\zeta}\cdot{\bf n})=C_K,\quad\forall K\in\T_h$;
\item\label{assump4} $\inf\limits_{x\in\mathcal{E}}(\tau_1-\frac12\bz\cdot{\bf n})\geq C^*\max_{x\in\mathcal{E}}|\bz(x)\cdot{\bf n}|,\forall\mathcal{E}\in\partial K$ and $\forall K\in\T_h.$
\end{enumerate}
\end{assumption}

We also make the following assumption as in~\cite[Assumption 6.2]{TZAD2021}.
\begin{assumption}\label{assumpconstrain}
For each two adjacent subdomains $\Omega_i$ and $\Omega_j$, which share the edge $\mathcal{E}_{ij}$, for any $\lambdae^{(i)}\in{\Lambdae}^{(i)}$, we assume that 
\begin{align*}
\int_{\mathcal{E}_{ij}}\lambdae^{(i)} ds,
\quad  \int_{{\mathcal{E}}_{ij}}{\bm\zeta}\cdot{\bf n}\lambdae^{(i)}  ds,\quad \int_{{\mathcal{E}}_{ij}}{\bm\zeta}\cdot{\bf n}\lambdae^{(i)}s  ds,
\end{align*}
are contained in the coarse level primal subspace $\widehat{\Lambdae}_{\Pi}$, which are the same or different by a $-1$ factor due to the normal direction.
\end{assumption}

Let  ${\bf T}=\widetilde{\R}^T_{\D,\Gamma}\widetilde{\S}^{-1}_\Gamma\widetilde{\R}_{\D,\Gamma}{\S}_\Gamma$ denote the BDDC preconditioned operator.
The convergence behavior of the GMRES method applied to equation \eqref{bddc} is analyzed using the framework established in \cite[Theorem 6.5]{TZAD2021}, whose proof is detailed in \cite{eielsc}. {{Based on this result, we derive the following theorem on the convergence rate of the BDDC preconditioned GMRES iterations, which constitutes the main result of this paper.}}
{{\begin{theorem}\label{theorem:EES}
Given $\lambdae_\Gamma \in \widehat{\Lambdae}_\Gamma$, let $C_u$ and $c_l$ be the constants established in Lemma~\ref{upperbbound} and Lemma~\ref{cl}, respectively. Then, the following bounds hold:
\begin{align*}
c_l \langle \lambdae_\Gamma, \lambdae_\Gamma \rangle_{\M_\Gamma} &\leq \langle \lambdae_\Gamma, \mathbf{T} \lambdae_\Gamma \rangle_{\M_\Gamma}, \\
\langle \mathbf{T} \lambdae_\Gamma, \mathbf{T} \lambdae_\Gamma \rangle_{\M_\Gamma} &\leq C_u^2 \langle \lambdae_\Gamma, \lambdae_\Gamma \rangle_{\M_\Gamma}.
\end{align*}
As a result, the residual $\mathbf{r}_m$ of the GMRES iteration at iteration $m$ satisfies
\begin{align*}
\frac{\| \mathbf{r}_m \|_{\M_\Gamma}}{\| \mathbf{r}_0 \|_{\M_\Gamma}} \leq \left(1 - \frac{c_l^2}{C_u^2}\right)^{m/2}.
\end{align*}
\end{theorem}}}

%
\subsection{Average operator estimates}
{{The relationship among the norms \( \interleave \cdot \interleave \) defined in~\eqref{eq:norm2}, \( \|\cdot\|_{\Aa} \)  defined in~\eqref{defAanorm}, and \( \|\cdot\|_{\Ma} \) defined ~\eqref{eq:Manorm} is established in the following lemma.

}}
\begin{lemma}\label{lambdaetriple}
For any $K\in\T_h$ and $\lambdae=(\lambda,\mu)^T\in\widetilde\Lambdae$ we have 
\begin{equation}\label{Ktriplebar}
C\interleave\lambda\interleave_K\leq \|{Q}_{A_1}{\lambdae}\|_K,\quad C \interleave\mu\interleave_K\leq \|{Q}_{A_2}{\lambdae}\|_K
\end{equation}
which implies that 
\begin{equation}\label{Anormtriplrbar}
C\beta^{\frac{1}{4}} \interleave \lambdae \interleave \leq \|{\bf E}_{\C}(\lambdae) \|_{\Aa}\leq 
 \| {\bf E}_{\C}(\lambdae) \|_{\Ma}.
\end{equation}
\end{lemma}
\begin{proof}
\eqref{Ktriplebar} can be proved similarly as \cite[Lemma 3.8]{CDGT2014}. ~\eqref{Anormtriplrbar} follows  from assembling ~\eqref{Ktriplebar} over all elements $K\in\T_h$, with the definitions~\eqref{defAanorm} and ~\eqref{eq:Manorm}.
\end{proof}

Estimates for the $L^2$-norm of $\U\lambdae$ is given in the lemma below.
\begin{lemma}\label{UL2K}
For any $\lambdae=(\lambda,\mu)^T\in\widetilde\Lambdae$ , we have the following estimates:
\begin{align}
&\|\U\lambdae\|_{K} \leq Cc_0 \|\lambdae\|_{h,K},\label{Uestimate2}\\
&\|\U\lambdae\|_{\T_h} \leq Cc_0 \|\lambdae\|_{h} \label{Uestimate3},
\end{align}
where $c_0= h \beta^{-\frac{1}{2}} + 1$. 
Therefore,
\begin{align}\label{Uestimate2Th}
\|\U\lambdae\|_{\T_h}\leq Cc_0\beta^{-\frac14}\|{\bf E}_{\C}(\lambdae)\|_{\Aa}.
\end{align}

\end{lemma}

The following lemma provides the gradient estimates and other auxiliary bounds for the extensions of {the optimal control problem.}
\begin{lemma}\label{UUADL2difference2}
For $\lambdae=(\lambda,\mu)^T\in\widetilde{\Lambdae}$, we have 
\begin{align}
&\|U_{A_1}\lambdae - U_{A^{ad}_1}\lambda\|_{\T_h} + \|U_{A_2}\lambdae - U_{A^{ad}_2}\mu\|_{\T_h}\leq Ch\beta^{-\frac12}\|{\U}\lambdae\|_{\T_h}\label{UUaddifference},\\
&\beta^{\frac12}\|\nabla U_{A_1}\lambdae\|_{\T_h}\leq C(1+\beta^{\frac14})\|{\bf E}_{\C}(\lambdae)\|_{\Ma},\quad \beta^{\frac12}\|\nabla U_{A_2}\lambdae\|_{\T_h}\leq C(1+\beta^{\frac14})\|{\bf E}_{\C}(\lambdae)\|_{\Ma}.\label{gradient1ThM}
\end{align}
Moreover, we obtain that 
\begin{align}
&\beta^{\frac{1}{4}} \||\tau_1 - \frac{1}{2} \bz \cdot \mathbf{n}|^{\frac{1}{2}} 
(U_{A_1}\lambdae - \lambda)\|_{\partial\mathcal{T}_h} 
+ \beta^{\frac{1}{4}} \||\tau_2 + \frac{1}{2} \bz \cdot \mathbf{n} |^{\frac{1}{2}} 
(U_{A_2}\lambdae - \mu) \|_{\partial\mathcal{T}_h}\label{UAlambdaMestimate} \\
&\leq C h^{\frac{1}{2}}(1+\beta^{-\frac14})\|{\bf E}_{\C}(\lambdae)\|_{\Ma}.\notag
\end{align}
\end{lemma}

Lemma~\ref{abilinearlessthan1} establishes the relation between ${\bf E}_{\C}(\lambdae_{\A,\Gamma})$, as defined in~\eqref{defEA} and~\eqref{lambdaIdef}, and ${\bf E}_{\Aad}(\lambdae_{\Aad,\Gamma})$, as defined in~\eqref{defEAd} and~\eqref{lambdaad}.
\begin{lemma}\label{abilinearlessthan1}
Given $\lambdae_\Gamma=(\lambda_\Gamma,\mu_\Gamma)^T\in\widetilde{\bf\Lambda}_\Gamma,$ we have 
\begin{equation}\label{less1AAd1}
\|{\bf E}_{\C}(\lambdae_{\A,\Gamma})\|_{\Aa}
\leq C\beta^{\frac14}c_0\|{\bf E}_{\Aad}(\lambdae_{\Aad,\Gamma})\|_{\Aad}+C(c_0\beta^{-\frac14})\|\lambdae_{\Aad,\Gamma}\|_h,
\end{equation}
and
\begin{equation}\label{great1AAd}
\beta^{\frac{1}{4}} \|{\bf E}_{\Aad}(\lambdae_{\Aad,\Gamma})\|_{\Aad} \leq C \|{\bf E}_{\C}(\lambdae_{\A,\Gamma})\|_{\Aa}. 
\end{equation}
\end{lemma}

We define two Schur complements, $\widetilde{S}^{ad}_{1,\Gamma}$ and $\widetilde{S}^{ad}_{2,\Gamma}$, for $A^{ad}_1$ and $A^{ad}_2$ in~\eqref{cB}.
$\widetilde{S}^{ad}_{1,\Gamma}$ is defined similarly to \cite[Equation (4.1)]{TZAD2021} with $\epsilon=1$ and $\tau=\tau_1$.  
 $\widetilde{S}^{ad}_{2,\Gamma}$ is defined analogously with $\tau=\tau_2$ and $\bz$ replaced by $-\bz$ in $\widetilde{S}^{ad}_{1,\Gamma}$. 
Let
$\widetilde{\S}^{\ad}_{\Gamma}=\begin{bmatrix}\widetilde{S}^{ad}_{1,\Gamma}& {\bf 0}\\ {\bf 0}&\widetilde{S}^{ad}_{2,\Gamma}\end{bmatrix}$.
Then, we have 
\begin{equation}\label{lambdaAadAS}
\|\lambdae_{\Gamma}\|_{\widetilde{\S}^{\ad}_\Gamma}=\|{\bf E}_{\Aad}(\lambdae_{\Aad,\Gamma})\|_{\Aad}.
\end{equation}
Let $\E_{\D}\lambdae_\Gamma$ be the average operator defined as
$\E_{\D}\lambdae_\Gamma = \widetilde{\R}_\Gamma\widetilde{\R}_{\D,\Gamma}\lambdae_\Gamma.$ 
An estimate for the average operator of the advection-diffusion problem is provided in~\cite[Lemma 6.9]{TZAD2021} with $\epsilon=1$.  This leads to the following lemma:

\begin{lemma}\label{CEDAad}
For $\lambdae_\Gamma\in \widetilde{\Lambdae}_\Gamma$, we have
\begin{align*}
\|{\bf E_D}\lambdae_\Gamma\|_{\widetilde{\S}^{\ad}_\Gamma}\leq C\left(1+\log\frac Hh\right)\|\lambdae_\Gamma\|_{\widetilde{\S}^{\ad}_\Gamma}.
\end{align*}
\end{lemma}
We are now ready to provide the estimates for the average operator.
\begin{lemma}\label{averagelarger1}
Given $\lambdae_\Gamma \in \widetilde{\Lambdae}_\Gamma$, there exists a positive constant $C$, independent of $\beta, H, h$, such that for all $\lambdae_\Gamma \in \widetilde{\Lambdae}_\Gamma$, we have
\begin{align}
\|{\bf E_D}\lambdae_\Gamma\|_{\widetilde{\S}_\Gamma}&\leq CC_{ED}\|\lambdae_\Gamma\|_{\widetilde{\S}_\Gamma},\label{CED} \\
 \|{\bf E_D}\lambdae_\Gamma\|_{\widetilde{\M}_\Gamma}&\leq C_{ED,M}\|\lambdae_\Gamma\|_{\widetilde{\M}_\Gamma} ,\label{ED2}
\end{align}
where $C_{ED} =c_0(1+H\beta^{-\frac12})\left(1+\log\frac Hh\right), C_{ED,M}=C(1+c_0H\beta^{-\frac14})C_{ED},$ and $c_0$ is defined in Lemma~\ref{UL2K}.
\end{lemma}
\begin{proof}
Let ${\bf v}_{\Gamma}=\E_{\D}\lambdae_\Gamma-\lambdae_\Gamma$, then we have 
$
\|{{\E}_{\D}}\lambdae_\Gamma\|^2_{\widetilde{\S}_\Gamma}\leq 
C\|{\bf v}_\Gamma\|^2_{\widetilde{\S}_\Gamma}
+C\|{\lambdae}_\Gamma\|^2_{\widetilde{\S}_\Gamma}.$
Since
\begin{equation}\label{less1SGammav}
\begin{aligned}
&\|{\bf v}_{\Gamma}\|_{\widetilde{\S}_\Gamma}=\|{\E}_{\C}({\v}_{\A,\Gamma})\|_{\Aa}\leq Cc_0\beta^{\frac14}\|{\bf E}_{\Aad}({\v}_{\Aad,\Gamma})\|_{\Aad}+C{(c_0\beta^{-\frac14})}\|{\v}_{\Aad,\Gamma}\|_h\\
&\leq Cc_0\beta^{\frac14}\|{\bf E}_{\Aad}({\v}_{\Aad})\|_{\Aad}+C{{(c_0\beta^{-\frac14})}}H\interleave{\v}_{\Aad,\Gamma}\interleave\\
&\leq C(c_0\beta^{\frac14}+c_0\beta^{-\frac14}H)\|{\bf E}_{\Aad}({\v}_{\Aad,\Gamma})\|_{\Aad} \\
&=C(c_0\beta^{\frac14}+c_0\beta^{-\frac14}H)\|{\v}_{\Gamma}\|_{\widetilde{\S}^{\ad}_\Gamma}\leq C (c_0\beta^{\frac14}+c_0\beta^{-\frac14}H)\left(1+\log\frac Hh\right)\|{\lambdae_\Gamma}\|_{\widetilde{\S}^{\bf ad}_\Gamma}\\
&=C(c_0\beta^{\frac14}+c_0\beta^{-\frac14}H)\left(1+\log\frac Hh\right)\|{\bf E}_{\Aad}({\lambdae}_{\Aad,\Gamma})\|_{\Aad}\\
&\leq C (c_0+c_0\beta^{-\frac12}H)\left(1+\log\frac Hh\right)\|{\bf E}_{\C}({\lambdae_{\A,\Gamma}})\|_{\Aa}\\
&= C(c_0+c_0\beta^{-\frac12}H)\left(1+\log\frac Hh\right) \|{\lambdae_\Gamma}\|_{\widetilde{\S}_\Gamma}:=CC_{ED} \|{\lambdae_\Gamma}\|_{\widetilde{\S}_\Gamma},  
\end{aligned}
\end{equation}
where $C_{ED}:=c_0(1+\beta^{-\frac12}H).$ We use Lemma~\ref{MSrelation} for the first equality,
~\eqref{less1AAd1} for the first inequality.
The second inequality follows from Lemma \ref{citeadlemma7.8} with the fact that 
\begin{equation}\label{vadhH}
\|{\v}_{\Aad,\Gamma}\|_h=\big(\sum\limits_{i=1}^N\|{\v}^{(i)}_{\Aad,\Gamma}\|^2_{h,\Omega_i}\big)^{\frac12}\leq C \big(\sum\limits_{i=1}^NH^2\interleave{\v}^{(i)}_{\Aad,\Gamma}\interleave^2_{\Omega_i}\big)^{\frac12}\leq CH\interleave{\v}_{\Aad,\Gamma}\interleave.
\end{equation}
The third inequality follows from~\eqref{ADestimateE}, 
{{the second and third equalities follow from~\eqref{lambdaAadAS}}}, 
the second to last inequality follows from Lemma~\ref{CEDAad}, and the last inequality follows from \eqref{great1AAd}. The second to last equality follows from~Lemma~\ref{MSrelation}. 
 
Moreover, we have
\begin{equation}\label{UAestimateED}
\begin{aligned}
\|{\bf U}{\bf v}_{\A,\Gamma}\|_{\T_h}&\leq Cc_0 \|{\bf v}_{\A,\Gamma}\|_{h}\leq  Cc_0H\interleave{\bf v}_{\A,\Gamma}\interleave\leq Cc_0 H\beta^{-\frac14}\|{\bf v}_\Gamma\|_{\widetilde{\S}_\Gamma}\\
&\leq Cc_0 H\beta^{-\frac14}C_{ED}\|\lambdae_\Gamma\|_{\widetilde{\S}_\Gamma},
\end{aligned}
\end{equation}
where ~\eqref{Uestimate3} is used for the first inequality. The second inequality can be obtained similarly as ~\eqref{vadhH}, the second to last inequality follows from~\eqref{Anormtriplrbar}, and the last inequality follows from~\eqref{less1SGammav}.

Therefore, by~Lemma~\ref{MSrelation},\eqref{less1SGammav} and~\eqref{UAestimateED}, we obtain that 
\begin{align*}
\|{\bf v}_{\Gamma}\|_{\widetilde{\M}_\Gamma}\leq C (1+c_0 H\beta^{-\frac14})C_{ED}\|\lambdae_\Gamma\|_{\widetilde{\S}_\Gamma},
\end{align*}
and 
\begin{align*}
\|{\bf E}_{\D}{\lambdae}_\Gamma\|_{\widetilde\M_\Gamma}\leq C\|{\bf v}_\Gamma\|_{\widetilde\M_\Gamma}+C\|{\lambdae}_\Gamma\|_{\widetilde\M_\Gamma}
\leq  C(1+c_0 H\beta^{-\frac14})C_{ED}\|\lambdae_\Gamma\|_{\widetilde{\M}_\Gamma}.
\end{align*}
\end{proof}

\subsection{Upper bound estimate for $C_u$ and lower bound estimate for $c_l$}
We make use of the dual argument together with the norm equivalence between the extensions of the advection-diffusion problem and the extensions of the optimal problem stated in Lemma~\ref{abilinearlessthan1}, to derive the following estimates for the nonsymmetric bilinear form.
\begin{lemma}\label{Zhatestimatetilde}
Given $\lambdae_\Gamma,\s_\Gamma\in \widetilde{\Lambdae}_\Gamma$, we have 
\begin{align}
\langle\lambdae_\Gamma,{\s}_\Gamma\rangle_{\widetilde{\bf Z}_\Gamma} &\leq C\gamma_1\|\lambdae_\Gamma\|_{\M_\Gamma}\left(C(H\beta^{-\frac14})\|\s_{\Gamma}\|_{\widetilde\S_\Gamma}+\|\s_{\Aad,\Gamma}\|_{h}\right),\label{ZNL2estimate2} 
\end{align}
where $\gamma_1=(1+\beta^{\frac14})c_0$ and $c_0$ is defined in~Lemma~\ref{UL2K}.
The same estimate holds for $\lambdae_\Gamma, \s_\Gamma \in \widehat{\Lambdae}_\Gamma$ analogously.
\end{lemma}

Given $\lambdae_\Gamma\in \widehat\Lambdae_\Gamma$, define $ \w_\Gamma=\widetilde{\S}^{-1}_\Gamma\widetilde{\R}_{\D,\Gamma}\widehat\S_\Gamma\lambdae_\Gamma$. Different from the extension $\U_{\Aad}\lambdae$ studied in~\cite{TZAD2021}, the extension operator $\U$ of the optimal control problem does not preserve constants; that is, even if \(\lambdae\) is constant on the mesh element, \(\U\lambdae\) is not a constant, as seen from~\eqref{QUAAad}. This leads to the difficulty of estimating the $h$-norm $\|\w_{\A,\Gamma}-\lambdae_{\A,\Gamma}\|_h$directly using the approach in \cite[Lemma 7.20]{TZAD2021}. To overcome this difficulty, 
we will use a dual argument and  estimate the $h$-norm $\|{\bf w}_{\Aad,\Gamma} - \lambdae_{\Aad,\Gamma}\|_h $ as in the following lemma. 
\begin{lemma}\label{wdnormS}
Given $\lambdae_\Gamma\in \widehat\Lambdae_\Gamma$, let $\mathcal{H}_{\Aad}(\lambdae)=\lambdae_{\Aad,\Gamma}$, we have
\begin{align}
\|\mathcal{H}_{\Aad}({\bf w}_\Gamma - \lambdae_\Gamma)\|_h 
&\leq C \alpha_1 \left( \|\w_\Gamma\|_{\widetilde\M_\Gamma} + \|\lambdae_\Gamma\|_{\widehat\S_\Gamma} \right),  \label{TAdnormSw1}\\
\|\mathcal{H}_{\Aad}({\bf T} \lambdae_\Gamma - \lambdae_\Gamma)\|_h 
&\leq C \alpha_2 \left( \|\w_\Gamma\|_{\widetilde\M_\Gamma} + \|\lambdae_\Gamma\|_{\widehat\S_\Gamma} \right), \label{TAdnormSw2}
\end{align}
where $\alpha_1=H\big((1+\beta^{-\frac12})c^2_0+\beta^{-\frac14}\big), \alpha_2=H\big((1+\beta^{-\frac12})c_0^2+\beta^{-\frac14}C_{ED}\big)$, $c_0$ is defined in Lemma~\ref{UL2K} and $C_{ED}$ is defined in Lemma~\ref{averagelarger1}.
\end{lemma}
\begin{lemma}\label{Chigamma}
For any  $\lambdae_\Gamma\in\widehat{\Lambdae}_\Gamma,$when $H$ is sufficiently small,
let $\z_\Gamma=\widetilde{\L}^{-1}_\Gamma\widetilde{\bf R}_{\D,\Gamma}{\L}_\Gamma\lambdae_\Gamma$, we have 
the following estimates
\begin{align}
|\l\w_\Gamma,\z_\Gamma-
\widetilde{\R}_\Gamma \lambdae_\Gamma\r_{\widetilde\L_\Gamma}|&\leq Cc_1\|\lambdae_\Gamma\|_{{\L}_\Gamma}\|{\bf w}_\Gamma\|_{\widetilde{\M}_\Gamma}\label{wlambdaL1},\\
|\l\lambdae_\Gamma, {\bf T}\lambdae_{\Gamma}-\lambdae_{\Gamma}\r_{\Z_\Gamma}|&\leq Cc_2\|\lambdae\|_{\M_\Gamma} (\|\w_\Gamma\|_{\widetilde\M_\Gamma}+\|\lambdae_\Gamma\|_{\M_\Gamma}),\label{wlambdaL4}\\
\|\w_\Gamma-\widetilde{\R}_\Gamma\lambdae_{\Gamma}\|_{\widetilde\L_\Gamma}
&\leq Cc_3\|\w_\Gamma\|_{\widetilde{\M}_\Gamma},
\label{wlambdaL2}\\
|\langle\w_\Gamma,\widetilde{\R}_\Gamma{\lambdae}_\Gamma-\w_\Gamma\rangle_{\widetilde{\Z}_\Gamma}|&\leq 
Cc_4\|\w\|_{\widetilde{\M}_\Gamma}(\|\w\|_{\widetilde{\M}_\Gamma}+\|\lambdae_{\Gamma}\|_{\M_\Gamma}),
\label{wlambdaL3}
\end{align}
where $c_1=c_0H\beta^{-\frac14}C_{ED},c_2=\gamma_1\alpha_2, c_3=c_0H(1+\beta^{-\frac12}), c_4=\gamma_1 \alpha_1$. $\alpha_1, \alpha_2$ are defined in Lemma~\ref{wdnormS}. $\gamma_1$ is defined in Lemma~\ref{Zhatestimatetilde}. 
\end{lemma}
{{Recall that ${\bf T}=\widetilde{\R}^T_{\D,\Gamma}\widetilde{\S}^{-1}_\Gamma\widetilde{\R}_{\D,\Gamma}\widehat{\S}_\Gamma$, we proceed to derive the upper bound $C_u$ and lower bound $c_l$.}}
\begin{lemma}\label{upperbbound}
Given $\lambdae_\Gamma\in \widehat\Lambdae_\Gamma$, when $H$ is sufficient small, we have
\begin{align*}
\langle{\bf T}\lambdae_\Gamma,{\bf T}{\lambdae}_\Gamma\rangle_{{\M}_\Gamma} &\leq C^2_u\l\lambdae_\Gamma,\lambdae_\Gamma\r_{\M_\Gamma},
\end{align*}
where $C_u=C{{C}^2_{ED,M}}$. $C_{ED,M}$ is defined in Lemma~\ref{averagelarger1}.
\end{lemma}
\begin{proof}
By~\eqref{lambdauBwilde} and~\eqref{wlambdaL2}, we have
\begin{align*}
\|\w_\Gamma\|^2_{\widetilde\M_\Gamma}&= \|\w_\Gamma\|^2_{\widetilde\B_\Gamma}+\|\w_\Gamma\|^2_{\widetilde\L_\Gamma}\leq 
\|\w_\Gamma\|^2_{\widetilde\B_\Gamma}+\|\w_\Gamma-\widetilde\R_\Gamma\lambdae_\Gamma\|^2_{\widetilde\L_\Gamma}+\|\widetilde\R_\Gamma\lambdae_\Gamma\|^2_{\widetilde\L_\Gamma}\\
&\leq  \|\w_\Gamma\|^2_{\widetilde\B_\Gamma}+Cc^2_3\|\w_\Gamma\|^2_{\widetilde{\M}_\Gamma}+\|\lambdae\|^2_{\M_\Gamma}, 
\end{align*}
when $H$ is sufficiently small, we can make $c_3$ small enough and obtain that 
\begin{equation}\label{wBestimate}
\|\w_\Gamma\|^2_{\widetilde\M_\Gamma}\leq { {C}}\big(\|\w_\Gamma\|^2_{\widetilde\B_\Gamma}+\|\lambdae_\Gamma\|^2_{\M_\Gamma}\big).
\end{equation}
Also, by~Lemma~\ref{MSrelation}, ~\eqref{wlambdaL3} and~\eqref{wBestimate}, we have
\begin{equation}\label{wBTM}
\begin{aligned}
\|\w_{\Gamma}\|^2_{\widetilde\B_\Gamma}&=\|\w_{\Gamma}\|^2_{\widetilde\S_\Gamma}=\w^T_{\Gamma}\widetilde\S_\Gamma\w^T_{\Gamma}
=\w^T_{\Gamma}\widetilde\S_\Gamma\widetilde\S^{-1}_\Gamma\widetilde\R_{\D,\Gamma}\widehat{\S}_\Gamma\lambdae_\Gamma
=\l\lambdae_\Gamma, \widetilde\R^T_{\D,\Gamma}\w_{\Gamma}\r_{\widehat{\S}_\Gamma}\\
&=\l\lambdae_\Gamma, {\bf T}\lambdae_{\Gamma}\r_{\widehat{\S}_\Gamma}=\l\lambdae_\Gamma, {\bf T}\lambdae_{\Gamma}\r_{{\B}_\Gamma}
+\l\lambdae_\Gamma, {\bf T}\lambdae_{\Gamma}-\lambdae_{\Gamma}\r_{{\Z}_\Gamma}\\
&\leq \|\lambdae_\Gamma\|_{\B_\Gamma}\|{\bf T}\lambdae_\Gamma\|_{\B_\Gamma}+ Cc_2\|\lambdae\|_{\M_\Gamma} (\|\w_\Gamma\|_{\widetilde\M_\Gamma}+\|\lambdae_\Gamma\|_{\M_\Gamma})\\
&\leq \|\lambdae_\Gamma\|_{\M_\Gamma}\|{\bf T}\lambdae_\Gamma\|_{\M_\Gamma}+{Cc_2}\|\lambdae\|_{\M_\Gamma}\|\w_\Gamma\|_{\widetilde\B_\Gamma}+ {Cc_2}\|\lambdae\|^2_{\M_\Gamma}.
\end{aligned}
\end{equation}
By Young's inequality, we obtain that
\begin{align}\label{wTlambdafinal}
\|\w_{\Gamma}\|^2_{\widetilde\B_\Gamma}\leq C\big(\|\lambdae_\Gamma\|_{\M_\Gamma}\|{\bf T}\lambdae_\Gamma\|_{\M_\Gamma}+(Cc^2_2+Cc_2)\|\lambdae_\Gamma\|^2_{\M_\Gamma}\big).
\end{align}
Since
\begin{align*}
\langle {\bf T}\lambdae_{\Gamma}, {\bf T}\lambdae_{\Gamma} \rangle_{\M_{\Gamma}} &= \langle \widetilde{\R}^T_{\D,\Gamma} \widetilde{\S}_{\Gamma}^{-1} \widetilde{\R}_{\D,\Gamma} \S_{\Gamma} \lambdae_{\Gamma}, \widetilde{\R}^T_{\D,\Gamma} \widetilde{\S}_{\Gamma}^{-1} \widetilde{\R}_{\D,\Gamma} \S_{\Gamma} \lambdae_{\Gamma} \rangle_{{\M}_{\Gamma}} \\
&= \langle \E_{\D}\w_{\Gamma}, \E_{\D}\w_{\Gamma} \rangle_{\widetilde{\M}_{\Gamma}}= \|\E_{\D}\w_{\Gamma}\|_{\widetilde{\M}_{\Gamma}}^2,
\end{align*}
we see that
\begin{align*}
\|{\bf T}\lambdae_\Gamma\|^2_{\M_\Gamma}&=\|\E_{\D}\w_{\Gamma}\|_{\widetilde{\M}_{\Gamma}}^2\leq {CC}^2_{ED,M}\|\w_{\Gamma}\|^2_{\widetilde\M_\Gamma}\leq {C{C}^2_{ED,M}}\big(\|\w_\Gamma\|^2_{\widetilde\B_\Gamma}+\|\lambdae_\Gamma\|^2_{\M_\Gamma}\big),\\
&\leq C{{C}^2_{ED,M}}\|\lambdae_\Gamma\|_{\M_\Gamma}\|{\bf T}\lambdae_\Gamma\|_{\M_\Gamma}
+C{{C}^2_{ED,M}(c^2_2+c_2+1)}\|\lambdae_\Gamma\|^2_{\M_\Gamma},
\end{align*}
where ~\eqref{ED2} is used for the first inequality,~\eqref{wBestimate} for the second inequality, and ~\eqref{wTlambdafinal} for the last inequality.
Consequently, when $H$ is sufficiently small, $c_2$ can be small enough, therefore applying Young's inequality yields
\begin{equation}\label{TlambdaM}
\begin{aligned}
\|{\bf T}\lambdae_\Gamma\|^2_{\M_\Gamma}&\leq C\big(\big({{C}^2_{ED,M}}\big)^2+{{C}^2_{ED,M}(c^2_2+c_2+1)}\big)\|\lambdae_\Gamma\|^2_{\M_\Gamma}\\
&\leq C\big({{C}^2_{ED,M}}\big)^2\|\lambdae_\Gamma\|^2_{\M_\Gamma}=C^2_u\|\lambdae_\Gamma\|^2_{\M_\Gamma},
\end{aligned}
\end{equation}
where $C_u:=C{{C}^2_{ED,M}}.$
\end{proof}
\begin{lemma}\label{cl}
For any  $\lambdae_\Gamma\in\widehat{\Lambdae}_\Gamma$, when $H$ is sufficiently small, we have 
\begin{align*}
c_{l}\langle\lambdae_\Gamma,\lambdae_\Gamma\rangle_{\M_\Gamma}\leq  \l\lambdae_\Gamma,{\bf T}\lambdae_\Gamma\r_{\M_\Gamma},
\end{align*}
where  $
c_{l}=C\big(1-Cc_2C_{ED,M}-Cc_3C^2_{ED,M}\big)
$. $c_2$ and $c_3$ are defined in Lemma~\ref{Chigamma}. $C_{ED,M}$ is defined in Lemma~\ref{upperbbound}.
\end{lemma}
\begin{proof}
We have
\begin{align*}
\langle\lambdae_\Gamma,\lambdae_\Gamma\rangle_{{\M}_\Gamma}&=\langle\lambdae_\Gamma,\lambdae_\Gamma\rangle_{{\B}_\Gamma}+\langle\lambdae_\Gamma,\lambdae_\Gamma\rangle_{{\L}_\Gamma}
=\langle\lambdae_\Gamma,\lambdae_\Gamma\rangle_{{\wS}_\Gamma}+\langle\lambdae_\Gamma,\lambdae_\Gamma\rangle_{{\L}_\Gamma}\\
&=\langle\w_\Gamma,\widetilde{\R}_\Gamma\lambdae_\Gamma\rangle_{\widetilde{\S}_\Gamma}
+\langle\lambdae_\Gamma,\lambdae_\Gamma\rangle_{{\L}_\Gamma}\\
&=\langle \w_\Gamma, \widetilde{\R}_\Gamma\lambdae_\Gamma\rangle_{\widetilde\M_\Gamma}-\langle \w_\Gamma, \widetilde{\R}_\Gamma\lambdae_\Gamma\rangle_{\widetilde{\bf L}_\Gamma}+\langle\w_\Gamma,\widetilde{\R}_\Gamma{\lambdae}_\Gamma\rangle_{\widetilde{\Z}_\Gamma}
+\langle\widetilde{\R}_\Gamma\lambdae_\Gamma,\widetilde{\R}_\Gamma\lambdae_\Gamma\rangle_{\widetilde{\bf L}_\Gamma}\\
&=\langle \w_\Gamma, \widetilde{\R}_\Gamma\lambdae_\Gamma\rangle_{\widetilde\M_\Gamma}-\langle \w_\Gamma-\widetilde{\R}_\Gamma\lambdae_\Gamma,\widetilde{\R}_\Gamma\lambdae_\Gamma\rangle_{\widetilde{\L}_\Gamma}+\langle\w_\Gamma,\widetilde{\R}_\Gamma{\lambdae}_\Gamma-\w_\Gamma\rangle_{\widetilde{\Z}_\Gamma}
\\
&\leq \frac12\|\w_\Gamma\|^2_{\widetilde{\M}_\Gamma}+\frac12\|\lambdae_\Gamma\|^2_{\M_\Gamma}-\langle \w_\Gamma-\widetilde{\R}_\Gamma\lambdae_\Gamma,\widetilde{\R}_\Gamma\lambdae_\Gamma\rangle_{{\widetilde{\bf L}_\Gamma}}+\langle\w_\Gamma,\widetilde{\R}_\Gamma{\lambdae}_\Gamma-\w_\Gamma\rangle_{\widetilde{\Z}_\Gamma},
\end{align*}
where the first and the second equalities follow from Lemma~\ref{MSrelation}. The third equality follows from~\eqref{scalingD} with the fact that
\[
\langle\w_\Gamma,\widetilde{\R}_\Gamma\lambdae_\Gamma\rangle_{\widetilde{\S}_\Gamma}=\lambdae^T_\Gamma\widetilde{\R}_\Gamma^T\widetilde{\S}_\Gamma\w_\Gamma
=\lambdae^T_\Gamma\widetilde{\R}_\Gamma^T\widetilde{\S}_\Gamma\widetilde{\S}^{-1}_\Gamma\widetilde{\R}_{\D,\Gamma}\widehat\S_\Gamma\lambdae_\Gamma
=\langle\lambdae_\Gamma,\lambdae_\Gamma\rangle_{\widehat{\S}_\Gamma}.\]
The last inequality follows from
$$\langle \w_\Gamma, \widetilde{\R}_\Gamma\lambdae_\Gamma\rangle_{\widetilde\M_\Gamma}\leq \|\w_\Gamma\|_{\widetilde\M_\Gamma}\| \widetilde{\R}_\Gamma\lambdae_\Gamma\|_{\widetilde\M_\Gamma}\leq \frac12 (\|\w_\Gamma\|_{\widetilde\M_\Gamma}^2+\|\lambdae_\Gamma\|^2_{\M_\Gamma}).$$
Therefore, we obtain that
\begin{align}\label{lambdawgamma}
\l\lambdae_\Gamma,\lambdae_\Gamma\r_{{\M}_\Gamma}
&\leq \|\w_\Gamma\|^2_{\widetilde{\bf M}_\Gamma}
+C|\langle \w_\Gamma-\widetilde{\R}_\Gamma\lambdae_\Gamma,\widetilde{\R}_\Gamma\lambdae_\Gamma\rangle_{\widetilde{\L}_\Gamma}|+C|\langle\w_\Gamma,\widetilde{\R}_\Gamma{\lambdae}_\Gamma-\w_\Gamma\rangle_{\widetilde{\Z}_\Gamma}
|.
\end{align}
In order to estimate $\|\lambdae_\Gamma\|_{\M_\Gamma}$, we need to estimate  $\|\w_\Gamma\|_{\widetilde\M_\Gamma}$. Define ${\z}_\Gamma=\widetilde{\L}^{-1}_\Gamma\widetilde{\bf R}_{\D,\Gamma}{\L}_\Gamma\lambdae_\Gamma$, we have
\begin{equation}\label{wc123}
\begin{aligned}
&\l\w_\Gamma,\w_\Gamma\r_{\widetilde{\M}_\Gamma}\\
&=\l\w_\Gamma,\w_\Gamma\r_{\widetilde{\B}_\Gamma}+\l\w_\Gamma,\w_\Gamma\r_{\widetilde{\L}_\Gamma}
=\l\lambdae_\Gamma,\Tb\lambdae_\Gamma\r_{\widehat\S_\Gamma}+\l\w_\Gamma,\w_\Gamma\r_{\widetilde{\L}_\Gamma}\\
&=\l\lambdae_\Gamma,\Tb\lambdae_\Gamma\r_{\M_\Gamma}-\l\lambdae_\Gamma,\Tb\lambdae_\Gamma\r_{{\L}_\Gamma}
+\l\lambdae_\Gamma,\Tb\lambdae_\Gamma\r_{\Z_\Gamma}+\l\w_\Gamma,\w_\Gamma\r_{\widetilde{\L}_\Gamma}\\
&=\l\lambdae_\Gamma,\Tb\lambdae_\Gamma\r_{\M_\Gamma}+\l\lambdae_\Gamma,\Tb\lambdae_\Gamma\r_{\Z_\Gamma}-\l\w_\Gamma,\z_\Gamma-\w_\Gamma\r_{\widetilde{\L}_\Gamma}\\
&=\l\lambdae_\Gamma,\Tb\lambdae_\Gamma\r_{\M_\Gamma}+\l\lambdae_\Gamma,\Tb\lambdae_\Gamma-\lambdae_\Gamma\r_{\Z_\Gamma}-\l\w_\Gamma,\z_\Gamma-\w_\Gamma\r_{\widetilde{\L}_\Gamma}\\
&=\l\lambdae_\Gamma,\Tb\lambdae_\Gamma\r_{\M_\Gamma}+\l\lambdae_\Gamma,\Tb\lambdae_\Gamma-\lambdae_\Gamma\r_{\Z_\Gamma}-\l\w_\Gamma,\z_\Gamma-
\widetilde{\R}_\Gamma \lambdae_\Gamma\r_{\widetilde\L_\Gamma}+\l\w_\Gamma,\w_\Gamma-
\widetilde{\R}_\Gamma \lambdae_\Gamma\r_{\widetilde\L_\Gamma},
\end{aligned}
\end{equation}
where the first equality follows from Lemma~\ref{MSrelation}, the second equality follows from $\|\w_\Gamma\|^2_{\widetilde\B_\Gamma}=\l\lambdae_\Gamma,{\bf T}\lambdae_\Gamma\r_{\widehat{\S}_\Gamma}$ which is proved in~\eqref{wBTM},
and the third equality follows from the symmetry of the inner product $\langle \cdot,\cdot \rangle_{\mathbf{L}_\Gamma}$ with the fact that
\begin{align*}
\l\lambdae_\Gamma,{\bf T}\lambdae_\Gamma\r_{\bf L_\Gamma}
=\lambdae^T_\Gamma{\L}_\Gamma \widetilde{\bf R}^T_{\D,\Gamma}\widetilde{\S}^{-1}_\Gamma\widetilde{\bf R}_{\D,\Gamma}{\widehat\S}_\Gamma\lambdae_\Gamma
=\lambdae^T_\Gamma{\L}_\Gamma \widetilde{\bf R}^T_{\D,\Gamma}\widetilde{\L}_\Gamma^{-1}\widetilde{\L}_\Gamma\widetilde{\S}^{-1}_\Gamma\widetilde{\bf R}_{\D,\Gamma}{\widehat\S}_\Gamma\lambdae_\Gamma=\l{\bf w}_\Gamma,{\bf z}_\Gamma\r_{\widetilde{\L}_\Gamma}.
\end{align*}
Also, when $H$ is sufficiently small, by~\eqref{wBestimate},~\eqref{wTlambdafinal},~\eqref{TlambdaM}, we obtain 
\begin{equation}\label{TMCEDM}
\|{\w}_\Gamma\|_{\widetilde\M_\Gamma}\leq CC_{ED,M}\|\lambdae\|_{\M_\Gamma},
\end{equation}
where $C_{ED,M}$ is defined in Lemma~\ref{upperbbound}.
Then, we have
\begin{equation}\label{wTCEDMc2}
\begin{aligned}
&\l\w_\Gamma,\w_\Gamma\r_{\widetilde{\M}_\Gamma}\\
&\leq \l\lambdae_\Gamma,\Tb\lambdae_\Gamma\r_{\M_\Gamma}+\l\lambdae_\Gamma,\Tb\lambdae_\Gamma-\lambdae_\Gamma\r_{\Z_\Gamma}+|\l\w_\Gamma,\z_\Gamma-
\widetilde{\R}_\Gamma \lambdae_\Gamma\r_{\widetilde\L_\Gamma}|\\
&\quad+|\l\w_\Gamma,\w_\Gamma-
\widetilde{\R}_\Gamma \lambdae_\Gamma\r_{\widetilde\L_\Gamma}|\\
&\leq  \l\lambdae_\Gamma,\Tb\lambdae_\Gamma\r_{\M_\Gamma}+Cc_2\|\lambdae_\Gamma\|_{\M_\Gamma}(\|\w_\Gamma\|_{\widetilde\M_\Gamma}+\|\lambdae_\Gamma\|_{\M_\Gamma})+Cc_1\|\lambdae_\Gamma\|_{\M_\Gamma}\|{\bf w}_\Gamma\|_{\widetilde{\M}_\Gamma}\\
&\quad+Cc_3\|\w_\Gamma\|^2_{\widetilde{\M}_\Gamma}\\
&\leq\l\lambdae_\Gamma,\Tb\lambdae_\Gamma\r_{\M_\Gamma}+Cc_2C_{ED,M}\|\lambdae_\Gamma\|^2_{\M_\Gamma}+Cc_3C^2_{ED,M}\|\lambdae_\Gamma\|^2_{\M_\Gamma},
\end{aligned}
\end{equation}
where the first inequality follows from~\eqref{wc123}, the second inequality follows from~\eqref{wlambdaL4}, ~\eqref{wlambdaL1} and ~\eqref{wlambdaL2}, the  last inequality follows from~\eqref{TMCEDM}, and $c_1\leq c_2$.
Therefore,
\begin{equation*}
\begin{aligned}
&\l\lambdae_\Gamma,\lambdae_\Gamma\r_{\M_\Gamma}\\
&\leq \|\w_\Gamma\|^2_{\widetilde{\M}_\Gamma}+C\|\w_\Gamma-
\widetilde{\R}_\Gamma \lambdae_\Gamma\|_{\widetilde\L_\Gamma}\|\lambdae_\Gamma\|_{\L_\Gamma}+C|\langle\w_\Gamma,\widetilde{\R}_\Gamma{\lambdae}_\Gamma-\w_\Gamma\rangle_{\widetilde{\Z}_\Gamma}|\\
&\leq \|\w_\Gamma\|^2_{\widetilde{\M}_\Gamma}+
Cc_3\|\w_\Gamma\|_{\widetilde{\M}_\Gamma}\|\lambdae_\Gamma\|_{\M_\Gamma}+Cc_4\|\w\|_{\widetilde{\M}_\Gamma}(\|\w\|_{\widetilde{\M}_\Gamma}+\|\lambdae_{\Gamma}\|_{\M_\Gamma})\\
&\leq (C+Cc_4)\|\w_\Gamma\|^2_{\widetilde{\M}_\Gamma}+C(c^2_3+c^2_4)\|\lambdae_\Gamma\|^2_{\M_\Gamma}\\
&\leq {(C+Cc_4)}\bigg(
\l\lambdae_\Gamma,\Tb\lambdae_\Gamma\r_{\M_\Gamma}
+Cc_2C_{ED,M}\|\lambdae_\Gamma\|^2_{\M_\Gamma}+Cc_3C^2_{ED,M}\|\lambdae_\Gamma\|^2_{{\M}_\Gamma}\bigg)+C(c^2_3+c^2_4)\|\lambdae_\Gamma\|^2_{\M_\Gamma}\\
&\leq (C+Cc_4)
\l\lambdae_\Gamma,\Tb\lambdae_\Gamma\r_{\M_\Gamma}
+{(C+Cc_4)\big(c_2C_{ED,M}+c_3C^2_{ED,M}\big)}\|\lambdae_\Gamma\|^2_{\M_\Gamma},
\end{aligned}
\end{equation*}
where the first inequality follows from~\eqref{lambdawgamma}, the second inequality follows from~\eqref{wlambdaL2} and~\eqref{wlambdaL3}, the third inequality follows from Young's inequality, the second to last inequality follows from~\eqref{wTCEDMc2}, and the last inequality follows from
the fact that $c_3\leq C_{ED,M}$, $c_4\leq c_2$, and $C_{ED,M}\geq 1$.

Therefore, when $H$ is sufficiently small, we have $c_2, c_3, c_4$ small enough such that 
\[
\l\lambdae_\Gamma,\lambdae_\Gamma\r_{\M_\Gamma}\leq C\l\lambdae_\Gamma,\Tb\lambdae_\Gamma\r_{\M_\Gamma}+C\big(c_2C_{ED,M}+c_3C^2_{ED,M}\big)\|\lambdae_\Gamma\|^2_{\M_\Gamma},
\]
and 
\begin{align*}
c_{l}\l\lambdae_\Gamma,\lambdae_\Gamma\r_{\M_\Gamma}\leq \l\lambdae_\Gamma,{\bf T}\lambdae_\Gamma\r_{\M_\Gamma},
\end{align*}
where $c_{l}:=C\big(1-Cc_2C_{ED,M}-Cc_3C^2_{ED,M}\big)$.
\end{proof}

\begin{remark}
Based on the estimates in Lemma~\ref{upperbbound} and Lemma~\ref{cl}, we obtain the following:
when $\beta=O(1)$, the bounds $C_u, c_l$ can be simplified as 
\begin{align*}
C_u&=C{\bigg(1+\log \frac Hh\bigg)^2},\quad c_l=C\bigg(1-CH\bigg(1+\log \frac Hh\bigg)^2\bigg),
\end{align*}
and these estimates are consistent with the results in~\cite{TZAD2021} for the advection–diffusion problem.
 When $\beta$ is not too small compared to $h$, namely $\beta=O(h^{2-\delta})$ for any fixed $\delta\in (0,2)$,
 the bounds $C_u, c_l$ can be simplified as 
\begin{align*}
C_u&={C\bigg(1+\log \frac Hh \bigg)^2},\quad c_l=C\bigg(1-CHh^{-1+\frac\delta 2}\bigg(1+\log \frac Hh\bigg)^2\bigg).
\end{align*}
In this case, our results are comparable to those established in~\cite[Theorem 4.2]{Schoberl:2011:RobustMG}, which are based on standard finite element discretizations combined with multigrid methods. The $\mathcal{O}(h^{2-\delta})$ threshold can be reflected in Lemma~\ref{wdnormS}, which utilizes specific properties of the HDG discretizations and plays a crucial role in establishing the lower bound. 
In both cases, the convergence rates are independent of $\beta$. By selecting appropriate values of $H$ and $ h $, $c_l$ can be ensured to be positive.  
The convergence rate  of the proposed BDDC algorithms is independent of the number of subdomains and depends slightly on subdomain problem size $ H/h$.  
When $ \beta$ is very small, $H$ is required to be tiny to maintain the positivity of $ c_l$, which might not be practical.  
Nevertheless, numerical experiments still exhibit good convergence behavior in such case.
\end{remark}

\section{Proof of lemmas}\label{sec:lemmaproofs}
In this section, we provide the proofs of Lemma~\ref{UL2K},  \ref{UUADL2difference2}, \ref{abilinearlessthan1}, \ref{Zhatestimatetilde}, \ref{wdnormS} and \ref{Chigamma}, along with the supporting lemmas required for these proofs.


\begin{lemma}\label{bdulambda}
Given $\lambdae=(\lambda,\mu)^T,{\s}=(s,t)^T\in \widetilde{\Lambdae}$, under Assumption \ref{assumpconstrain},
 we have the following estimates
\begin{align*}
|\langle{\bm\zeta}\cdot{\bf n}\lambda,{s}\rangle_{\partial\T_h}|\leq CH\beta^{-\frac12}\|{\bf E}_{\C}(\lambdae)\|_{\Aa}\|{\bf E}_{\C}({\s})\|_{\Aa},\quad |\langle{\bm\zeta}\cdot{\bf n}t,{\mu}\rangle_{\partial\T_h}|\leq CH\beta^{-\frac12}\|{\bf E}_{\C}(\lambdae)\|_{\Aa}\|{\bf E}_{\C}({\s})\|_{\Aa}.
\end{align*}
\end{lemma}
\begin{proof}
Suppose that $\Omega_i$ and $\Omega_j$ are two adjacent subdomains sharing the interface edge $\mathcal{E}_{ij}$. Let $\lambdae^{(i)}=(\lambda^{(i)},\mu^{(i)})^T$ and ${\s}^{(i)}=(s^{(i)},t^{(i)})^T$ denote the restrictions of $\lambdae$ and $\s$ to subdomain $\Omega_i$, respectively. Define
$\overline{\lambda}_{\mathcal{E}_{ij}}=\frac1{|{\mathcal{E}_{ij}}|}\int_{{\mathcal{E}_{ij}}}\lambda^{(i)}ds$ and 
$\overline{s}_{\mathcal{E}_{ij}}=\frac1{|{\mathcal{E}_{ij}}|}\int_{{\mathcal{E}_{ij}}}s^{(i)}ds$, we obtain the following result 
\begin{equation}\label{lambdasHboundary}
\begin{aligned}
&|\langle{\bm\zeta}\cdot{\bf n}\lambda,{s}\rangle_{\partial\T_h}|=
\bigg|\sum\limits_{i=1}^N\sum\limits_{\mathcal{F}\subset \partial\Omega_i}\langle{\bm\zeta}\cdot{\bf n}\lambda^{(i)},{s}^{(i)}\rangle_{\mathcal{F}}\bigg|
=\bigg|\sum\limits_{i=1}^N\sum\limits_{\mathcal{F}\subset \partial\Omega_i\cap\partial\Omega_j}\langle{\bm\zeta}\cdot{\bf n}\lambda^{(i)},{s}^{(i)}\rangle_{\mathcal{F}}\bigg|\\
&=\bigg|\sum\limits_{i=1}^N\sum\limits_{\mathcal{F}\subset \partial\Omega_i\cap\partial\Omega_j}\langle{\bm\zeta}\cdot{\bf n}(\lambda^{(i)}-\overline{\lambda}_{\mathcal{E}_{ij}}),({s}^{(i)}-\overline{s}_{\mathcal{E}_{ij}})\rangle_{\mathcal{F}}\bigg|\\
&\le C\sum\limits_{i=1}^N\sum\limits_{\mathcal{F}\subset\partial\Omega_i\cap\partial\Omega_i}
\|\lambda^{(i)}-\overline{\lambda}_{\mathcal{E}_{ij}}\|_{\mathcal{F}}\|{s}^{(i)}-\overline{s}_{\mathcal{E}_{ij}}\|_{\mathcal{F}}
\le C\sum\limits_{i=1}^N H\interleave\lambda^{(i)}\interleave_{\T_h(\Omega_i)}\interleave{s}^{(i)}\interleave_{\T_h(\Omega_i)}\\
&\leq CH\beta^{-\frac12}\|{\bf E}_{\C}(\lambdae)\|_{\Aa}\|{\bf E}_{\C}({\s})\|_{\Aa},
\end{aligned}
\end{equation}
where the last equality follows from Assumption \ref{assumpconstrain}, the second  inequality follows from equation \cite[Equation (7.49)]{TZAD2021}, and the last inequality follows from  \eqref{Anormtriplrbar}. A Similar result holds for $(\mu,t)^T\in \widetilde\Lambdae$. Therefore, we obtain the results.
\end{proof}

\subsection{Proof of Lemma~\ref{UL2K}}
\begin{proof}
We will use a similar method as the proof in \cite[Lemma 7.2]{TZAD2021}. By equation \eqref{equation:AmatrixQU}, 
\cite[equation (2.16)]{TZAD2021} with $\epsilon=1$, 
we have 
\begin{equation}\label{systemA}
\begin{aligned}
&(Q_{A_1}\lambdae,{\bf r})_K-(U_{A_1}\lambdae,\nabla\cdot{\bf r})_K+\langle\lambda,{\bf r}\cdot{\bf n}\rangle_{\partial K}=0,\\
&\beta^{\frac12}(\nabla\cdot{Q}_{A_1}\lambdae,w)_{K}+\beta^{\frac12}\langle(\tau_1-\frac12\bz\cdot{\bf n})(U_{A_1}\lambdae-\lambda),w\rangle_{\partial K}\\
&\quad-\frac{\beta^{\frac12}}2(\bz\cdot\nabla w,U_{A_1}\lambdae)_{K}+\frac{\beta^{\frac12}}2(\bz\cdot\nabla U_{A_1}\lambdae,w)_{K}
+\frac{\beta^{\frac12}}2\langle \bz\cdot{\bf n}\lambda,w\rangle_{K}-(U_{A_2}\lambdae,w)_K=0,
\end{aligned}
\end{equation}
and
\begin{equation}\label{systemAad} 
\begin{aligned}
&(Q_{A^{ad}_1}\lambda,{\bf r})_K-(U_{A^{ad}_1}\lambda,\nabla\cdot{\bf r})_K+\langle\lambda,{\bf r}\cdot{\bf n}\rangle_{\partial K}=0,\\
&\beta^{\frac12}(\nabla\cdot{Q}_{A^{ad}_1}\lambda,w)_{K}+\beta^{\frac12}\langle(\tau_1-\frac12\bz\cdot{\bf n})(U_{A^{ad}_1}\lambda-\lambda),w\rangle_{\partial K}\\
&\quad-\frac{\beta^{\frac12}}2(\bz\cdot\nabla w,U_{A^{ad}_1}\lambda)_{K}+\frac{\beta^{\frac12}}2(\bz\cdot\nabla U_{A^{ad}_1}\lambda,w)_{K}
+\frac{\beta^{\frac12}}2\langle \bz\cdot{\bf n}\lambda,w\rangle_{K}=0.
\end{aligned}
\end{equation}
Subtracting \eqref{systemAad} from \eqref{systemA}, we obtain that 
\begin{equation}\label{QUAAad}
\begin{aligned}
&\beta^{\frac12}(Q_{A_1}\lambdae-Q_{A^{ad}_1}\lambda,{\bf r})_K-\beta^{\frac12}(U_{A_1}\lambdae-U_{A^{ad}_1}\lambda,\nabla\cdot{\bf r})_K=0\\
&\beta^{\frac12}(\nabla\cdot({Q}_{A_1}\lambdae-{Q}_{A^{ad}_1}\lambda),w)_{K}+\beta^{\frac12}\langle(\tau_1-\frac12\bz\cdot{\bf n})(U_{A_1}\lambdae-U_{A^{ad}_1}\lambda),w\rangle_{\partial K}\\
&\quad-\frac{\beta^{\frac12}}2(\bz\cdot\nabla w,U_{A_1}\lambdae-U_{A^{ad}_1}\lambda)_{K}
+\frac{\beta^{\frac12}}2(\bz\cdot\nabla (U_{A_1}\lambdae-U_{A^{ad}_1}\lambda),w)_{K}\\
&\quad-(U_{A_2}\lambdae,w)_K=0.
\end{aligned}
\end{equation}
Let \({\bf r}=Q_{A_1}\lambdae-Q_{A^{ad}_1}\lambda,\ w=U_{A_1}\lambdae-U_{A^{ad}_1}\lambda\), then 
\begin{equation}\label{eqQA1}
\begin{aligned}
&\beta^{\frac12}\|Q_{A_1}\lambda-Q_{A^{ad}_1}\lambdae\|^2_{K}
+\beta^{\frac12}\||\tau_1-\frac12\bz\cdot{\bf n}|^{\frac12}(U_{A_1}\lambdae-U_{A^{ad}_1}\lambda)\|^2_{\partial K}\\
&=(U_{A_2}\lambdae,U_{A_1}\lambdae-U_{A^{ad}_1}\lambda)_{K}.
\end{aligned}
\end{equation}
Similarly, by~\eqref{equation:AmatrixQU}, 
and \cite[Equation (2.16)]{TZAD2021} with $\epsilon=1$, $\tau_1$ replaced by $\tau_2$ and $\bz$ replaced by $-\bz$, we obtain that 
\begin{equation}\label{eqQA2}
\begin{aligned}
&\beta^{\frac12}\|Q_{A_2}\lambdae-Q_{A^{ad}_2}\mu\|^2_{K}
+\beta^{\frac12}\||\tau_2+\frac12\bz\cdot{\bf n}|^{\frac12}(U_{A_2}\lambdae-U_{A^{ad}_2}\mu)\|^2_{\partial K}\\
&=-(U_{A_1}\lambdae,U_{A_2}\lambdae-U_{A^{ad}_2}\mu)_{K}.
\end{aligned}
\end{equation}
By adding~\eqref{eqQA1} and~\eqref{eqQA2}, we have
\begin{align*}
&\beta^{\frac12}\|Q_{A_1}\lambda-Q_{A^{ad}_1}\lambda\|^2_{K}
+\beta^{\frac12}\||\tau_1-\frac12\bz\cdot{\bf n}|^{\frac12}(U_{A_1}\lambda-U_{A^{ad}_1}\lambda)\|^2_{\partial K}\\
&+\beta^{\frac12}\|Q_{A_2}\lambdae-Q_{A^{ad}_2}\mu\|^2_{K}
+\beta^{\frac12}\||\tau_2+\frac12\bz\cdot{\bf n}|^{\frac12}(U_{A_2}\lambdae-U_{A^{ad}_2}\mu)\|^2_{\partial K}\\
&=(U_{A_2}\lambdae,-U_{A^{ad}_1}\lambda)_{K}+(U_{A_1}\lambdae,U_{A^{ad}_2}\mu)_{K}
=(U_{A_2}\lambdae-U_{A^{ad}_2}\mu,-U_{A^{ad}_1}\lambda)_{K}+(U_{A_1}\lambdae-U_{A^{ad}_1}\lambda,U_{A^{ad}_2}\mu)_{K}.
\end{align*}
Therefore, we obtain that 
\begin{equation}\label{usefulQU}
\begin{aligned}
&\beta^{\frac12}\|Q_{A_1}\lambdae-Q_{A^{ad}_1}\lambda\|^2_{K}
+\beta^{\frac12}\||\tau_1-\frac12\bz\cdot{\bf n}|^{\frac12}(U_{A_1}\lambdae-U_{A^{ad}_1}\lambda)\|^2_{\partial K}\\
&+\beta^{\frac12}\|Q_{A_2}\lambdae-Q_{A^{ad}_2}\mu\|^2_{K}
+\beta^{\frac12}\||\tau_2+\frac12\bz\cdot{\bf n}|^{\frac12}(U_{A_2}\lambdae-U_{A^{ad}_2}\mu)\|^2_{\partial K}\\
&\leq \|U_{A_2}\lambdae-U_{A^{ad}_2}\mu\|_K\|U_{A^{ad}_1}\lambda\|_{K}+\|U_{A_1}\lambdae-U_{A^{ad}_1}\lambda\|_K\|U_{A^{ad}_2}\mu\|_{K}\\
&\leq \left( \|U_{A_2}\lambdae-U_{A^{ad}_2}\mu\|_K+\|U_{A_1}\lambdae-U_{A^{ad}_1}\lambda\|_K\right)\|\lambdae\|_{h,K},
\end{aligned}
\end{equation}
where the last inequality follows from ~\eqref{UA1L2} and ~\eqref{UA2L2}.

Next, to estimate $\|{\U} \lambdae\|_K$, we invoke the result from \cite[Lemma 3.3]{CDGT2014}, which states that
\begin{align}
&\|U_{A_1}\lambdae-U_{A^{ad}_1}\lambda\|_K\leq h\|{\mathcal{B}}^t(U_{A_1}\lambdae-U_{A^{ad}_1}\lambda)\|_K
+h^{\frac12}\min_{{\mathcal{E}}\subset \partial K}\|U_{A_1}\lambdae-U_{A^{ad}_1}\lambda\|_{\mathcal{E}},\label{UA1UAD1L2K}\\
&\|U_{A_2}\lambdae-U_{A^{ad}_2}\mu\|_K\leq h\|{\mathcal{B}}^t(U_{A_2}\lambdae-U_{A^{ad}_2}\mu)\|_K
+h^{\frac12}\min_{{\mathcal{E}}\subset \partial K}\|U_{A_2}\lambdae-U_{A^{ad}_2}\mu\|_{\mathcal{E}}.\label{UA2UAD2L2K}
\end{align}
Then, by~\eqref{usefulQU}, together with the identities 
${\mathcal{B}}^t(U_{A_1}\lambdae - U_{A^{ad}_1}\mu) = Q_{A_1}\lambdae - Q_{A^{ad}_1}\lambda$,  
${\mathcal{B}}^t(U_{A_2}\lambdae - U_{A^{ad}_2}\mu) = Q_{A_2}\lambdae - Q_{A^{ad}_2}\lambda$,
and~\eqref{assump3} in Assumption~\ref{eq:assump}, it is straightforward to see that
\begin{align*}
&\|U_{A_1}\lambdae - U_{A^{ad}_1}\lambda\|_K + \|U_{A_2}\lambdae - U_{A^{ad}_2}\mu\|_K \\
&\leq C h \|Q_{A_1}\lambdae - Q_{A^{ad}_1}\lambda\|_K 
+ \frac{C}{{C}_K} h^{\frac{1}{2}} \big\| |\tau_1 - \frac{1}{2} \bz \cdot \mathbf{n} \big|^{\frac{1}{2}} 
(U_{A_1}\lambdae - U_{A^{ad}_1}\lambda) \big\|_{\partial K} \\
&\quad + C h \|Q_{A_2}\lambdae - Q_{A^{ad}_2}\lambda\|_K 
+ \frac{C}{C_K} h^{\frac{1}{2}} \big\| \big|\tau_2 + \frac{1}{2} \bz \cdot \mathbf{n} \big|^{\frac{1}{2}} 
(U_{A_2}\lambdae - U_{A^{ad}_2}\lambda) \big\|_{\partial K} \\
&\leq C h^{\frac{1}{2}} \beta^{-\frac{1}{4}} 
\sqrt{ \left( \|U_{A_2}\lambdae - U_{A^{ad}_2}\mu\|_K 
+ \|U_{A_1}\lambdae - U_{A^{ad}_1}\lambda\|_K \right) \|\lambdae\|_{h,K} },
\end{align*}
which implies that
\begin{equation}\label{U12partialK}
\|U_{A_1}\lambdae - U_{A^{ad}_1}\lambda\|_K 
+ \|U_{A_2}\lambdae - U_{A^{ad}_2}\mu\|_K 
\leq C h \beta^{-\frac{1}{2}} \|\lambdae\|_{h,K},
\end{equation}
and
\begin{align*}
\|\U\lambdae\|_K 
&\leq C \left( \|U_{A_1}\lambdae\|_K + \|U_{A_2}\lambdae\|_K \right) \\
&\leq C \left( h \beta^{-\frac{1}{2}} \|\lambdae\|_{h,K} 
+ \|U_{A^{ad}_1}\lambda\|_K + \|U_{A^{ad}_2}\mu\|_K \right) \\
&\leq C (h \beta^{-\frac{1}{2}} + 1) \|\lambdae\|_{h,K}=Cc_0\|\lambdae\|_{h,K},
\end{align*}
where $c_0:=h \beta^{-\frac{1}{2}} + 1$. The last inequality follows from ~\eqref{UA1L2} and ~\eqref{UA2L2}. Therefore, we see that 
\begin{align*}
\|\U\lambdae\|_{\T_h}=\sqrt{\sum\limits_{K\in\T_h}\|\U\lambdae\|^2_K}\leq Cc_0\|\lambdae\|_h\leq Cc_0\interleave\lambdae\interleave\leq Cc_0\beta^{-\frac14}\|{\bf E}_{\C}(\lambdae)\|_{\Aa},
\end{align*}
where the second inequality follows from~Lemma~\ref{citeadlemma7.10}, the last inequality follows from~\eqref{Anormtriplrbar}.
 \end{proof}

\subsection{Proof of Lemma~\ref{UUADL2difference2}}
\begin{proof}
By~\eqref{eqQA1} and ~\eqref{eqQA2} we have 
\begin{equation}\label{usefulQU2}
\begin{aligned}
&\beta^{\frac12}\|Q_{A_1}\lambdae-Q_{A^{ad}_1}\lambda\|^2_{K}
+\beta^{\frac12}\||\tau_1-\frac12\bz\cdot{\bf n}|^{\frac12}(U_{A_1}\lambdae-U_{A^{ad}_1}\lambda)\|^2_{\partial K}\\
&+\beta^{\frac12}\|Q_{A_2}\lambdae-Q_{A^{ad}_2}\mu\|^2_{K}
+\beta^{\frac12}\||\tau_2+\frac12\bz\cdot{\bf n}|^{\frac12}(U_{A_2}\lambdae-U_{A^{ad}_2}\mu)\|^2_{\partial K}\\
&\leq \|U_{A_2}\lambdae-U_{A^{ad}_2}\mu\|_K\|U_{A_1}\lambdae\|_{K}+\|U_{A_1}\lambdae-U_{A^{ad}_1}\lambda\|_K\|U_{A_2}\lambdae\|_{K}\\
&\leq \left( \|U_{A_2}\lambdae-U_{A^{ad}_2}\mu\|_K+\|U_{A_1}\lambdae-U_{A^{ad}_1}\lambda\|_K\right)\|\U\lambdae\|_K.
\end{aligned}
\end{equation}
By~\eqref{UA1UAD1L2K}, ~\eqref{UA2UAD2L2K} and~\eqref{usefulQU2}, we have 
\begin{align*}
&\|U_{A_1}\lambdae - U_{A^{ad}_1}\lambda\|_K + \|U_{A_2}\lambdae - U_{A^{ad}_2}\mu\|_K \\
&\leq C h \|Q_{A_1}\lambdae - Q_{A^{ad}_1}\lambda\|_K 
+ \frac{C}{\widetilde{C}_K} h^{\frac{1}{2}} \big\| |\tau_1 - \frac{1}{2} \bz \cdot \mathbf{n} \big|^{\frac{1}{2}} 
(U_{A_1}\lambdae - U_{A^{ad}_1}\lambda) \big\|_{\partial K} \\
&\quad + C h \|Q_{A_2}\lambdae - Q_{A^{ad}_2}\lambda\|_K 
+ \frac{C}{\widetilde{C}_K} h^{\frac{1}{2}} \big\| \big|\tau_2 + \frac{1}{2} \bz \cdot \mathbf{n} \big|^{\frac{1}{2}} 
(U_{A_2}\lambdae - U_{A^{ad}_2}\lambda) \big\|_{\partial K} \\
&\leq C h^{\frac{1}{2}} \beta^{-\frac{1}{4}} 
\sqrt{ \left( \|U_{A_2}\lambdae - U_{A^{ad}_2}\mu\|_K 
+ \|U_{A_1}\lambdae - U_{A^{ad}_1}\lambda\|_K \right) \|{\U}\lambdae\|_{K} },
\end{align*}
which implies that 
\begin{align*}
&\|U_{A_1}\lambdae - U_{A^{ad}_1}\lambda\|_K + \|U_{A_2}\lambdae - U_{A^{ad}_2}\mu\|_K\leq Ch\beta^{-\frac12}\|{\U}\lambdae\|_{K},\\
&\beta^{\frac14}\||\tau_1-\frac12\bz\cdot{\bf n}|^{\frac12}(U_{A_1}\lambdae-U_{A^{ad}_1}\lambda)\|_{\partial K}
+\beta^{\frac14}\||\tau_2+\frac12\bz\cdot{\bf n}|^{\frac12}(U_{A_2}\lambdae-U_{A^{ad}_2}\mu)\|_{\partial K}\leq Ch^{\frac12}\beta^{-\frac14}\|\U\lambdae\|_K.
\end{align*}
Therefore, we obtain 
\begin{align}
&\|U_{A_1}\lambdae - U_{A^{ad}_1}\lambda\|_{\T_h} + \|U_{A_2}\lambdae - U_{A^{ad}_2}\mu\|_{\T_h}\leq Ch\beta^{-\frac12}\|{\U}\lambdae\|_{\T_h},\label{UUADl0}\\
&\beta^{\frac14}\||\tau_1-\frac12\bz\cdot{\bf n}|^{\frac12}(U_{A_1}\lambdae-U_{A^{ad}_1}\lambda)\|_{\partial\T_h}
+\beta^{\frac14}\||\tau_2+\frac12\bz\cdot{\bf n}|^{\frac12}(U_{A_2}\lambdae-U_{A^{ad}_2}\mu)\|_{\partial\T_h}\label{UUADl}\\
&\leq h^{\frac12}\beta^{-\frac14}\|\U\lambdae\|_{\T_h}\nonumber.
\end{align}
Moreover, we obtain that 
\begin{equation}\label{gradientUA1}
\begin{aligned}
\beta^{\frac12}\|\nabla U_{A_1}\lambdae\|_{\T_h}
&\leq \beta^{\frac12}\|\nabla U_{A_1}\lambdae-\nabla U_{A^{ad}_1}\lambda\|_{\T_h}+\beta^{\frac12}\|\nabla U_{A^{ad}_1}\lambda\|_{\T_h}\\
&\leq \beta^{\frac12}h^{-1}\|U_{A_1}\lambdae- U_{A^{ad}_1}\lambda\|_{\T_h}+\beta^{\frac12}\interleave\lambda\interleave_K\\
&\leq  \beta^{\frac12}h^{-1}(h\beta^{-\frac12})\|{\U}\lambdae\|_{\T_h}+\beta^{\frac14}\|{\bf E}_{\C}(\lambdae)\|_{{\Aa}}\\
&\leq C(1+\beta^{\frac14})\|{\bf E}_{\C}(\lambdae)\|_{\Ma},
\end{aligned}
\end{equation}
where the second inequality follows from an inverse inequality and~\eqref{gradientUestimate}, the second to last inequality follows from~\eqref{UUADl0} and~\eqref{Anormtriplrbar}, and the last inequality follows from the definition of $\|\cdot\|_{\Ma}$ in~\eqref{eq:Manorm}.
Moreover, we have 
\begin{align*}
&\beta^{\frac14}\||\tau_1-\frac12\bz\cdot{\bf n}|^{\frac12}(U_{A_1}\lambdae-\lambda)\|_{\partial\T_h}
+\beta^{\frac14}\||\tau_2+\frac12\bz\cdot{\bf n}|^{\frac12}(U_{A_2}\lambdae-\mu)\|_{\partial\T_h}\\
&\leq\beta^{\frac14}\||\tau_1-\frac12\bz\cdot{\bf n}|^{\frac12}(U_{A_1}\lambdae-U_{A^{ad}_1}\lambda)\|_{\partial\T_h}+\beta^{\frac14}\||\tau_1-\frac12\bz\cdot{\bf n}|^{\frac12}(U_{A^{ad}_1}\lambda-\lambda)\|_{\partial\T_h}\\
&\quad+\beta^{\frac14}\||\tau_2+\frac12\bz\cdot{\bf n}|^{\frac12}(U_{A_2}\lambdae-U_{A^{ad}_2}\mu)\|_{\partial\T_h}+\beta^{\frac14}\||\tau_2+\frac12\bz\cdot{\bf n}|^{\frac12}(U_{A^{ad}_2}\mu-\mu)\|_{\partial\T_h}\\
&\leq h^{\frac12}\beta^{-\frac14}\|\U\lambdae\|_{\T_h}+\beta^{\frac14}h^{\frac12}\interleave\lambdae\interleave\leq h^{\frac12}(\beta^{-\frac14}+1)\|{\bf E}_{\C}(\lambdae)\|_{\Ma},
\end{align*}
where the second to last inequality follows from~\eqref{UUADl} and~\eqref{Uminuslambdaestimate}, and the last inequality follows from ~\eqref{eq:Manorm} and ~\eqref{Anormtriplrbar}.
\end{proof}


\subsection{Proof of Lemma \ref{abilinearlessthan1}}
\begin{proof}
To prove~\eqref{less1AAd1}, notice that ${\bf E}_{\C}(\lambdae_{\A,\Gamma})-{\bf E}_{\Aad}(\lambdae_{\Aad,\Gamma})$ has {\bf 0} value on $\widetilde{\bf\Lambda}_\Gamma$, therefore 
\begin{equation}\label{EqAAhat1}
\begin{aligned}
\left({\bf E}_{\C}(\lambdae_{\A,\Gamma})-{\bf E}_{\Aad}(\lambdae_{\Aad,\Gamma})\right)^T{\Aa}{\bf E}_{\C}({\lambdae}_{\A,\Gamma})={0}.
\end{aligned}
\end{equation}
This implies 
\begin{align*}
&\|{\bf E}_{\C}(\lambdae_{\A,\Gamma})\|^2_{\Aa}={\bf E}^T_{\Aad}(\lambdae_{\Aad,\Gamma})\Aa
{\bf E}_{\C}(\lambdae_{\A,\Gamma})=
{\bf E}^T_{\Aad}(\lambdae_{\Aad,\Gamma})
\begin{bmatrix}
\beta^{\frac12} {A}^{ad}_1  & -{L}\\
{L} &\beta^{\frac12} {A}^{ad}_2
\end{bmatrix}{\bf E}_{\C}(\lambdae_{\A,\Gamma})\\
&=(Q_{A_1}\lambda_{A^{ad}_1,\Gamma},U_{A_1}\lambda_{A^{ad}_1,\Gamma},\lambda_{A^{ad}_1,\Gamma})
\begin{bmatrix}
\beta^{\frac12}A^{ad}_1& -L \end{bmatrix}{\bf E}_{\C}(\lambdae_{\A,\Gamma})\\
&\quad+(Q_{A_2}\mu_{A^{ad}_2,\Gamma},U_{A_2}\mu_{A^{ad}_2,\Gamma},\mu_{A^{ad}_2,\Gamma})
\begin{bmatrix}
L & \beta^{\frac12}A^{ad}_2\end{bmatrix}{\bf E}_{\C}(\lambdae_{\A,\Gamma}):=I+II
\end{align*}
We will estimate I and II can be estimated similarly. Since
{\begin{align*}
I&=\beta^{\frac12}({Q}_{A_1}\lambdae_{\A,\Gamma},Q_{A^{ad}_1}\lambda_{A^{ad}_1,\Gamma})_{\T_h}+\beta^{\frac12}\l(\tau_1-\frac12\bz\cdot{\bf n})(U_{A_1}\lambdae_{\A,\Gamma}-\lambda_{A,\Gamma}),U_{A_1}\lambda_{A^{ad}_1,\Gamma}-\lambda_{A^{ad}_1,\Gamma}\r_{\partial\T_h}\\
&\quad-\frac{\beta^{\frac12}}2(U_{A_1}\lambdae_{\A,\Gamma},\bz\cdot\nabla U_{A^{ad}_1}\lambda_{A^{ad}_1,\Gamma})_{\T_h}+\frac{\beta^{\frac12}}2(U_{A^{ad}_1}\lambdae_{A^{ad}_1,\Gamma},\bz\cdot\nabla U_{A_1}\lambda_{\A,\Gamma})_{\T_h}\\
&\quad+\frac{\beta^{\frac12}}2\l\bz\cdot{\bf n}\lambda_{A,\Gamma},U_{A^{ad}_1}\lambda_{A^{ad}_1,\Gamma}\r_{\partial\T_h}-\frac{\beta^{\frac12}}2
\l\bz\cdot{\bf n}(\lambda_{A^{ad}_1,\Gamma}),U_{A_1}\lambdae_{\A,\Gamma}\r_{\partial\T_h}-(U_{A_2}\lambdae_{\A,\Gamma},U_{A^{ad}_1}\lambda_{A^{ad}_1,\Gamma})_{\T_h}\\
&=\beta^{\frac12}({Q}_{A_1}\lambdae_{\A,\Gamma},Q_{A^{ad}_1}\lambda_{A^{ad}_1,\Gamma})_{\T_h}+\beta^{\frac12}\l(\tau_1-\frac12\bz\cdot{\bf n})(U_{A_1}\lambdae_{\A,\Gamma}-\lambda_{A,\Gamma}),U_{A_1}\lambda_{A^{ad}_1,\Gamma}-\lambda_{A^{ad}_1,\Gamma}\r_{\partial\T_h}\\
&\quad-\beta^{\frac12}(U_{A_1}\lambdae_{\A,\Gamma},\bz\cdot\nabla U_{A^{ad}_1}\lambda_{A^{ad}_1,\Gamma})_{\T_h}+\frac{\beta^{\frac12}}2\langle(\bz\cdot{\bf n})(U_{A^{ad}_1}\lambda_{A^{ad}_1,\Gamma}-\lambda_{A^{ad}_1,\Gamma}),U_{A_1}\lambdae_{A,\Gamma}\rangle_{\partial\T_h}\\
&\quad+\frac{\beta^{\frac12}}2\langle{\bz\cdot{\bf n}}\lambda_{A,\Gamma},U_{A^{ad}_1}\lambda_{A^{ad}_1,\Gamma}-\lambda_{A^{ad}_1,\Gamma}\rangle_{\partial\T_h}
+\frac{\beta^{\frac12}}2\langle(\bz\cdot{\bf n})\lambda_{A,\Gamma},\lambda_{A^{ad}_1,\Gamma}\rangle_{\partial\T_h}-(U_{A_2}\lambdae_{\A,\Gamma},U_{A^{ad}_1}\lambda_{A^{ad}_1,\Gamma})_{\T_h},
\end{align*}
where the second equality follows from the divergence theorem with $\nabla\cdot{\bz}=0$ such that 
$$(U_{A^{\text{ad}}_1} \lambda_{A^{\text{ad}}_1,\Gamma},\, \bz \cdot \nabla U_{A_1} \lambdae_{\A,\Gamma})_{\T_h}
= \l(\bz\cdot{\bf n})U_{A^{{ad}}_1} \lambda_{A^{{ad}}_1,\Gamma},\, U_{A_1} \lambdae_{\A,\Gamma} \r_{\partial \T_h}
- (\nabla U_{A^{{ad}}_1} \lambda_{A^{{ad}}_1,\Gamma}, U_{A_1} \lambdae_{\A,\Gamma})_{\T_h}.$$
Then, we have
\begin{align*}
|I|&\leq  \beta^{\frac14}\|{\bf E}_{\C}(\lambdae_{\A,\Gamma})\|_{\Aa}\|{\bf E}_{\Aad}(\lambdae_{\Aad,\Gamma})\|_{\Aad}+{\beta^{\frac12}}\|\nabla U_{A^{ad}_1}\lambda_{A^{ad}_1,\Gamma}\|_{\T_h}\|U_{A_1}\lambdae_{A_1,\Gamma}\|_{\T_h}\\
&\quad+C{\beta^{\frac12}}\||\bz\cdot{\bf n}|^{\frac12}(\lambda_{A^{ad}_1,\Gamma}-U_{A_1}\lambda_{A^{ad}_1,\Gamma})\|_{\partial\T_h}\big(\|U_{A_1}\lambdae_{\A,\Gamma}\|_{\partial\T_h}+\|(\bz\cdot{\bf n})\lambda_{A,\Gamma}\|_{\partial\T_h}\big)\\
&\quad+\big|\frac{\beta^{\frac12}}2\langle(\bz\cdot{\bf n})\lambda_{A,\Gamma},\lambda_{A^{ad}_1,\Gamma}\rangle_{\partial\T_h}\big|+\|U_{A_2}\lambdae_{\A,\Gamma}\|_{\T_h}\|U_{A_1}\lambda_{A^{ad}_1,\Gamma}\|_{\T_h}\\
&\leq C\beta^{\frac14}\|{\bf E}_{\C}(\lambdae_{\A,\Gamma})\|_{\Aa}\|{\bf E}_{\Aad}(\lambdae_{\Aad,\Gamma})\|_{\Aad}+C\beta^{\frac12}\|{\bf E}_{\Aad}(\lambdae_{\Aad,\Gamma})\|_{\Aad}\|U_{A_1}\lambdae_{\A,\Gamma}\|_{\T_h}\\
&\quad+C\|{\bf E}_{\Aad}(\lambdae_{\Aad,\Gamma})\|_{\Aad}\beta^{\frac12}\big(\|U_{A_1}\lambdae_{\A,\Gamma}\|_{\T_h}+h^{\frac12}\|\lambdae_{\A,\Gamma}\|_{\partial\T_h}\big)\\
&\quad+\beta^{\frac14}H\|{\bf E}_{\C}(\lambdae_{\A,\Gamma})\|_{\Aa}\|{\bf E}_{\Aad}(\lambdae_{\Aad,\Gamma})\|_{\Aad}+\|U_{A_2}\lambdae_{\A,\Gamma}\|_{\T_h}\|\lambdae_{\Aad,\Gamma}\|_h\\
&\leq C(c_0\beta^{\frac14})\|{\bf E}_{\C}(\lambdae_{\A,\Gamma})\|_{\Aa}\|{\bf E}_{\Aad}(\lambdae_{\A^{ad},\Gamma})\|_{\Aad}+C(c_0\beta^{-\frac14})\|{\bf E}_{\C}(\lambdae_{\A,\Gamma})\|_{\Aa}\|\lambdae_{\Aad,\Gamma}\|_h,
\end{align*}
where the second inequality follows from ~\eqref{gradientUestimate} with ~\eqref{ADestimateE},~\eqref{Uminuslambdaestimate}  with a trace inequality,~\eqref{UA1L2}, and the 
fact that 
\begin{equation}\label{lambdaAAd}
\bigg|\frac{\beta^{\frac12}}2\langle(\bz\cdot{\bf n})\lambda_{A,\Gamma},\lambda_{A^{ad}_1,\Gamma}\rangle_{\partial\T_h}\bigg|\leq C\beta^{\frac14}H\|{\bf E}_{\C}(\lambdae_{\A,\Gamma})\|_{\Aa}\|{\bf E}_{\Aad}(\lambdae_{\Aad,\Gamma})\|_{\Aad},
\end{equation}
where the inequality can be obtained following a similar process as~\eqref{lambdasHboundary} with Assumption~\ref{assumpconstrain},~\cite[Equation (7.49)]{TZAD2021}, ~\eqref{Anormtriplrbar} and ~\eqref{ADestimateE}.
The last inequality follows from~\eqref{Uestimate2Th}, Lemma~\ref{citeadlemma7.10} and ~\eqref{Anormtriplrbar}.} 

To prove \eqref{great1AAd}, we first consider the associated elliptic problem. Given $\lambdae=(\lambda,\mu)\in\widetilde\Lambdae$, we define 
 ${E}_{A^e}(\lambda)=(Q_{A^{e}}\lambda,U_{A^e}\lambda,\lambda)^T, {E}_{A^e}(\mu)=(Q_{A^e}\mu,U_{A^e}\mu,\mu)^T$. Here,
$(Q_{A^{e}}\lambda,U_{A^e}\lambda)$ and $(Q_{A^e}\mu,U_{A^e}\mu)$ satisfy
\begin{align*}
A^e\begin{bmatrix}Q_{A^{e}}\lambda \\ U_{A^e}\lambda\end{bmatrix}
=\begin{bmatrix} 0\\ 0\end{bmatrix},\quad 
A^e\begin{bmatrix}Q_{A^{e}}\mu \\ U_{A^e}\mu\end{bmatrix}
=\begin{bmatrix} 0\\ 0\end{bmatrix},
\end{align*}
where $A^{e}$ is the matrix  defined as \cite[Equation (2.5)]{TuWang2016HDG} with $\tau_k=\tau_1$. We then define 
$${\bf E}_{\A^e}(\lambdae)=(Q_{A^{e}}\lambda,U_{A^e}\lambda,\lambda, Q_{A^e}\mu,U_{A^e}\mu,\mu)^T$$ and the corresponding norm 
\begin{align*}
\|{\bf E}_{\A^e}(\lambdae)\|^2_{\A^{e}}=\|{E}_{A^e}(\lambda)\|^2_{A^e}+\|{E}_{A^e}(\mu)\|^2_{A^e}.
\end{align*}
Given $\lambdae_\Gamma=(\lambda_\Gamma,s_\Gamma)^T\in \widetilde{\Lambdae}_\Gamma$, let $\lambda_{A^{e},\Gamma}$ be the harmonic extension 
of $\lambda_\Gamma$. For each subdomain $\Omega_i$, let $\lambda^{(i)}_{A^e,\Gamma}$ denote the restriction of  $\lambda_{A^e,\Gamma}$ to $\Omega_i$, which satisfies
\begin{equation}\label{minharmonic}
\| E_{A^e}(\lambda^{(i)}_{A^{e},\Gamma}) \|^2_{A^{e{(i)}}} = 
\min_{\substack{\mu^{(i)} \in \Lambda^{(i)}, \mu^{(i)} = \mu^{(i)}_\Gamma \text{ on } \partial \Omega_i}} 
\| E_{A^e}(\mu^{(i)}) \|^2_{A^{e(i)}}.
\end{equation} 
Similarly, let $\mu_{A^{e},\Gamma}$ be the harmonic extension of $\mu_\Gamma$ and set $\lambdae_{\A^{e},\Gamma}= (\lambda_{A^{e},\Gamma}, \mu_{A^{e},\Gamma})^T$.
{{Following~\cite[Lemma 5.1]{TuWang2016HDG} and ~\eqref{assump1} in Assumption~\ref{eq:assump}, we have}
\begin{equation}\label{Aeinequality}
c\interleave\lambdae_{\A^e,\Gamma}\interleave\leq \|{\bf E}_{\A^e}(\lambdae_{\A^e,\Gamma})\|_{\A^e}\leq C\interleave\lambdae_{\A^e,\Gamma}\interleave,
\end{equation}
where $c$ and $C$ are two constants such that $c\leq C$.}

We have
\begin{align*}
\interleave\lambdae_{\A^{e},\Gamma}-\lambdae_{\Aad,\Gamma}\interleave^2&\leq C\|{\bf E}_{\Aad}(\lambdae_{\A^{e},\Gamma}-\lambdae_{\Aad,\Gamma})\|^2_{\Aad}\\
&={\bf E}^T_{\Aad}(\lambdae_{\A^{e},\Gamma}-\lambdae_{\Aad,\Gamma})\Aad({\bf E}_{\Aad}(\lambdae_{\A^{e},\Gamma}-\lambdae_{\Aad,\Gamma}))\\
&={\bf E}^T_{\Aad}(\lambdae_{\A^{e},\Gamma}-\lambdae_{\Aad,\Gamma})\Aad({\bf E}_{\Aad}(\lambdae_{\A^{e},\Gamma}))\\
&\leq C \interleave\lambdae_{\A^{e},\Gamma}-\lambdae_{\Aad,\Gamma}\interleave\interleave\lambdae_{\A^{e},\Gamma}\interleave,
\end{align*}
where the first follows from ~\eqref{ADestimateE}, the second equality follows from the fact that the difference $\lambdae_{\Aad,\Gamma} - \lambdae_{\A^{e},\Gamma}$ vanishes on the subdomain boundary, and the last inequality follows from \cite[Lemma 7.9]{TZAD2021} with $\epsilon=1$ and~\eqref{ADestimateE}.
Therefore, we obtain
\begin{equation}\label{AadAe}
\interleave\lambdae_{\A^{e},\Gamma}-\lambdae_{\Aad,\Gamma}\interleave\leq C\interleave\lambdae_{\A^{e},\Gamma}\interleave.
\end{equation}
Then, we have
\begin{align*}
\beta^{\frac14}\|{\bf E}_{\Aad}(\lambdae_{\Aad,\Gamma})\|_{\Aad}&\leq C\beta^{\frac14}\interleave\lambdae_{\Aad,\Gamma}\interleave\leq C\beta^{\frac14}\interleave\lambdae_{\A^e,\Gamma}\interleave\leq C\beta^{\frac14}\|{\bf E}_{\A^{e}}(\lambdae_{\A^{e},\Gamma})\|_{\A^{e}}\\
&\leq C\beta^{\frac14}\|{\bf E}_{\A^e}(\lambdae_{\A,\Gamma})\|_{\A^e}\leq C\beta^{\frac14}\interleave\lambdae_{\A,\Gamma}\interleave \leq C\|{\bf E}_{\C}(\lambdae_{\A,\Gamma})\|_{\Aa},
\end{align*}
where the first inequality follows from~\eqref{ADestimateE}, the second inequality follows from~\eqref{AadAe}, {{the third inequality follows from~\eqref{Aeinequality}, 
the third to last inequality follows from~\eqref{minharmonic}, the second to last inequality follows from~\eqref{Aeinequality},}} and the last inequality follows from~\eqref{Anormtriplrbar}.
\end{proof}
\begin{lemma}\label{UlambdaAboundary}
For any $\lambdae_{\Gamma}=(\lambda_{\Gamma},\mu_{\Gamma})^T\in \widetilde{\Lambdae}_\Gamma,$ we have 
\begin{align*}
&\|\lambdae_{\A,\Gamma}\|_{h}
\le  C(H\beta^{-\frac14})\|{\bf E}_{\C}(\lambdae_{\A,\Gamma})\|_{\Aa}+\|\lambdae_{\Aad,\Gamma}\|_{h}.
\end{align*}
\end{lemma}
\begin{proof}
We have
\begin{equation*}
\begin{aligned}
&\|\lambda_{A,\Gamma}\|_{h}=(h^{\frac12}\beta^{-\frac14})\beta^{\frac14}\|\lambda_{A,\Gamma}\|_{\partial\T_h}\leq (h^{\frac12}\beta^{-\frac14})\big(\beta^{\frac14}\|\lambda_{A,\Gamma}-\lambda_{A^{ad}_1,\Gamma}\|_{\partial\T_h}+
\beta^{\frac14}\|\lambda_{A^{ad}_1,\Gamma}\|_{\partial\T_h}\big)\\
&\leq (h^{\frac12}\beta^{-\frac14})\big(\beta^{\frac14}h^{-\frac12}\|\lambda_{A,\Gamma}-\lambda_{A^{ad}_1,\Gamma}\|_{h}+
\beta^{\frac14}\|\lambda_{A^{ad}_1,\Gamma}\|_{\partial\T_h}\big)\\
&\leq (h^{\frac12}\beta^{-\frac14})\big(\beta^{\frac14}h^{-\frac12}CH\interleave\lambda_{A,\Gamma}-\lambda_{A^{ad}_1,\Gamma}\interleave+\beta^{\frac14}h^{-\frac12}\|\lambda_{A^{ad}_1,\Gamma}\|_{h}\big)\\
&\leq CH\interleave\lambda_{A,\Gamma}-\lambda_{A^{ad}_1,\Gamma}\interleave+\|\lambda_{A^{ad}_1,\Gamma}\|_{h}\\
&\leq C(H\beta^{-\frac14})\|{\bf E}_{\C}(\lambdae_{\A,\Gamma})\|_{{\Aa}}+\|\lambda_{A^{ad}_1,\Gamma}\|_{h},
\end{aligned}
\end{equation*}
where the third inequality follows from Lemma~\ref{citeadlemma7.8} due to the fact that $\lambda_{A,\Gamma}-\lambda_{A^{ad}_1,\Gamma}$ has ${\bf 0}$ value on the subdomain boundary,  and the last inequality follows from a triangle inequality, ~\eqref{Anormtriplrbar},~\eqref{ADestimateE}, and
\eqref{great1AAd}. A similar argument yields
\[
\|\mu_{A,\Gamma}\|_{h} \leq C(H \beta^{-\frac{1}{4}}) \|{\bf E}_{\C}(\lambdae_{\A,\Gamma})\|_{\Aa} +  \|\mu_{A^{\text{ad}}_2,\Gamma}\|_h.
\]
This completes the proof.
\end{proof}
\subsection{Proof of Lemma \ref{Zhatestimatetilde}}
\begin{proof}
By~\eqref{sumoflocalbl} and~\eqref{lambdauBwilde}, we have
\begin{align*}
&\langle\lambdae_\Gamma,{\s}_\Gamma\rangle_{\widetilde{\bf Z}_\Gamma}=\left(-\frac{\beta^{\frac12}}2(U_{A_1}\lambdae_{\A,\Gamma},\bz\cdot\nabla U_{A_1}{\s}_{\A,\Gamma})_{\T_h}+\frac{\beta^{\frac12}}2(U_{A_1}\s_{\A,\Gamma},\bz\cdot\nabla U_{A_1}\lambdae_{\A,\Gamma})_{\T_h}\right.\\
&\quad\left.+\frac{\beta^{\frac12}}2\l\bz\cdot{\bf n}\lambda_{A,\Gamma},U_{A_1}\s_{\A,\Gamma}\r_{\partial\T_h}-\frac{\beta^{\frac12}}2
\l\bz\cdot{\bf n}s_{A,\Gamma},U_{A_1}\lambdae_{\A,\Gamma}\r_{\partial\T_h}-(U_{A_1}\s_{\A,\Gamma},U_{A_2}\lambdae_{\A,\Gamma})_{\T_h}\right)\\
&\quad+\left(\frac{\beta^{\frac12}}2(U_{A_2}\lambdae_{\A,\Gamma},\bz\cdot\nabla U_{A_2}{\s}_{\A,\Gamma})_{\T_h}-\frac{\beta^{\frac12}}2(U_{A_2}\s_{\A,\Gamma},\bz\cdot\nabla U_{A_2}\lambdae_{\A,\Gamma})_{\T_h}\right.\\
&\quad-\left.\frac{\beta^{\frac12}}2\l\bz\cdot{\bf n}\mu_{A,\Gamma},U_{A_2}\s_{\A,\Gamma}\r_{\partial\T_h}+\frac{\beta^{\frac12}}2
\l\bz\cdot{\bf n}t_{A,\Gamma},U_{A_2}\lambdae_{\A,\Gamma}\r_{\partial\T_h}+(U_{A_2}\s_{\A,\Gamma},U_{A_1}\lambdae_{\A,\Gamma})_{\T_h}\right)=I+II.
\end{align*}
We will consider the estimate of $I$ and $II$ can be estimated similarly. By the divergence theorem with $\nabla\cdot\bz=0$, we have
\[
\l(\bz\cdot{\bf n})U_{A_1}\lambdae_{\A,\Gamma},U_{A_1}\s_{\A,\Gamma}\r_{\partial\T_h}=({\bz}\cdot\nabla U_{A_1}\lambdae_{\A,\Gamma},U_{A_1}\s_{\A,\Gamma})_{\T_h}
+(U_{A_1}\lambdae_{\A,\Gamma},{\bz}\cdot\nabla U_{A_1}\s_{\A,\Gamma})_{\T_h},
\]
then, we obtain that  
\begin{equation}\label{Ifirst}
 \begin{aligned}
I&=-\frac{\beta^{\frac12}}2\l(\bz\cdot{\bf n})(U_{A_1}\lambdae_{\A,\Gamma}-\lambda_{A,\Gamma}),{U}_{A_1}\s_{\A,\Gamma}\r_{\partial\T_h}+\beta^{\frac12}(U_{A_1}\s_{\A,\Gamma},\bz\cdot\nabla U_{A_1}\lambdae_{\A,\Gamma})_{\T_h}\\
&\quad-\frac{\beta^{\frac12}}2\l(\bz\cdot{\bf n})s_{A,\Gamma},U_{A_1}\lambdae_{\A,\Gamma}-\lambda_{A,\Gamma}\r_{\partial\T_h}-(U_{A_1}\s_{\A,\Gamma},U_{A_2}\lambdae_{\A,\Gamma})_{\T_h}-\frac{\beta^{\frac12}}2\l(\bz\cdot{\bf n}) s_{A,\Gamma},\lambda_{A,\Gamma}\r_{\partial\T_h}.
\end{aligned}
\end{equation}
Therefore, we have
\begin{equation}\label{usefulI}
 \begin{aligned}
|I|&\leq Ch^{\frac12}(1+\beta^{-\frac14})\|{\bf E}_{\C}(\lambdae_{\A,\Gamma})\|_{\Ma}(\beta^{\frac14}h^{-\frac12})\|U_{A_1}\s_{\A,\Gamma}\|_{\T_h}+C(1+\beta^{\frac14})\|{\bf E}_{\C}(\lambdae_{\A,\Gamma})\|_{\Ma}\|U_{A_1}\s_{\A,\Gamma}\|_{\T_h}\\
&\quad+Ch^{\frac12}(1+\beta^{-\frac14})\|{\bf E}_{\C}(\lambdae_{\A,\Gamma})\|_{\Ma}(\beta^{\frac14}\|\s_{\A,\Gamma}\|_{\partial\T_h})+\|{\bf E}_{\C}(\lambdae_{\A,\Gamma})\|_{\Ma}\|U_{A_1}\s_{\A,\Gamma}\|_{\T_h}\\
&\quad+CH\|{\bf E}_{\C}(\lambdae_{\A,\Gamma})\|_{\Aa}\|{\bf E}_{\C}(\s_{\A,\Gamma})\|_{\Aa}\\
&\leq C\|{\bf E}_{\C}(\lambdae_{\A,\Gamma})\|_{{\Ma}}\left((1+\beta^{\frac14})\|\s_{\A,\Gamma}\|_{h}
+(1+\beta^{\frac14})\|U_{A_1}\s_{\A,\Gamma}\|_{\T_h}+H\|{\bf E}_{\C}(\s_{\A,\Gamma})\|_{\Aa}\right)\\
&\leq C(1+\beta^{\frac14})\|{\bf E}_{\C}(\lambdae_{\A,\Gamma})\|_{{\Ma}}\left(\|\s_{\A,\Gamma}\|_{h}+\|U_{A_1}\s_{\A,\Gamma}\|_{\T_h}+H\beta^{-\frac14}\|{\bf E}_{\C}(\s_{\A,\Gamma})\|_{\Aa}\right)\\
&\leq C(1+\beta^{\frac14})\|{\bf E}_{\C}(\lambdae_{\A,\Gamma})\|_{{\Ma}}\left(Cc_0\|\s_{\A,\Gamma}\|_{h}+H\beta^{-\frac14}\|{\bf E}_{\C}(\s_{\A,\Gamma})\|_{\Aa}\right)\\
&\leq C(1+\beta^{\frac14})c_0\|{\bf E}_{\C}(\lambdae_{\A,\Gamma})\|_{{\Ma}}\left(CH\beta^{-\frac14}\|{\bf E}_{\C}(\s_{\A,\Gamma})\|_{\Aa}+\|\s_{\Aad,\Gamma}\|_{h}\right)\\
&=C(1+\beta^{\frac14})c_0\|\lambdae_\Gamma\|_{\widetilde{\M}_\Gamma}\left(C(H\beta^{-\frac14})\|\s_{\Gamma}\|_{\widetilde\S_\Gamma}+\|\s_{\Aad,\Gamma}\|_{h}\right),
\end{aligned}
\end{equation}
where the first inequality follows from ~\eqref{UAlambdaMestimate}, a trace inequality,~\eqref{gradient1ThM}, ~\eqref{assump1} and~\eqref{assump4} in Assumption~\ref{eq:assump}, and~Lemma~\ref{bdulambda}, the second to last inequality follows from~\eqref{Uestimate3}, the last inequality follows from~Lemma~\ref{UlambdaAboundary}, and the last equality follows from~\eqref{MaMGamma} and ~Lemma~\ref{MSrelation}.
The same estimate holds for $\lambdae_\Gamma, \s_\Gamma \in \widehat{\Lambdae}_\Gamma$ analogously with the fact that $\l\bz\cdot{\bf n}\lambda,s\r_{\partial\T_h}=0$ and 
$\l\bz\cdot{\bf n}\mu,t\r_{\partial\T_h}=0$.
\end{proof}
\subsection{A dual argument}
Denote ${\bf W}= {W}\times{W}$,
given any  ${\bf g}=(g,z_d)^T\in {\bf W}$, let $(y^*_g,p^*_{z_d})$ be the solutions of the dual problem of  system \eqref{Osystem}, then we have 

\begin{subequations}\label{eq:sualfsystem}
 \begin{alignat}{3}
\bq^*_g+\nabla y^*_g&=0\quad&&\mbox{in}\quad\Omega,\\
 \beta^{\frac12}(\nabla\cdot\bq^*_g-\bz\cdot\nabla y^*_g){+}p^*_{z_d}&=g\quad&&\mbox{in}\quad\Omega,\\
 y^*_g&=0\quad&&\mbox{on}\quad\partial\Omega,\\
 \bp^*_{z_d}+\nabla p^*_{z_d}&=0\quad&&\mbox{in}\quad\Omega,\\
 \beta^{\frac12}(\nabla\cdot \bp^*_{z_d}+\nabla\cdot(\bz p^*_{z_d}))-y^*_{g}&=z_d\quad&&\mbox{in}\quad\Omega,\\
 p^*_{z_d}&=0\qquad &&\mbox{on}\quad\partial\Omega.
 \end{alignat}
 \end{subequations}
 The following regularity assumption holds for the dual problem(\cite[Lemma 2.2]{brenner2020multigrid}):
 \begin{align}\label{dual}
 \beta^{\frac12}\|y^*_g\|_{H^2(\Omega)}+ \beta^{\frac12}\|p^*_{z_d}\|_{H^2(\Omega)}\leq C(\|g\|_{L^2(\Omega)}+\|z_d\|_{L^2(\Omega)}).
 \end{align}
Given $\s=(s,t)^T\in\widetilde\Lambdae$, denote $\g=\U\s$. Let $\vpg=(\varphi_g,\varphi_{z_d})^T$ and $\vpgt=(\widetilde\varphi_g,\widetilde\varphi_{z_d})^T$ be the solutions of the following equations, 
\begin{align}
\mathcal{B}_h\big({\bf E}_{\C}(\lambdae),{\bf E}_{\C}(\vpg) \big)
&= (\U\lambdae, \g)_{\T_h}, \quad \forall \lambdae=(\lambda,\mu)^T \in\Lambdae, \label{eqdual1} \\
\widetilde{\mathcal{B}}\big({\bf E}_{\C}(\widetilde\lambdae),{\bf E}_{\C}(\vpgt) \big)
&= (\U\widetilde\lambdae, \g)_{\T_h}, \quad \forall \widetilde\lambdae=(\widetilde\lambda, \widetilde\mu)^T \in \widetilde\Lambdae, \label{eqdual2}
\end{align}
where $\cB_h$ is defined as~\eqref{eq:hdgconcise1} and $\widetilde\cB$ is defined as~\eqref{cB}.
 
{{By adopting a similar proof strategy as in~\cite[Lemma 7.17]{TZAD2021}, and utilizing the regularity property~\eqref{dual}, Assumption~\ref{eq:assump}, and Assumption~\ref{assumpconstrain}, 
we obtain the following result:
}}
\begin{lemma}\label{BygH2}
Given $\lambdae=(\lambda,\mu)^T\in\widetilde{\Lambdae}$, there exists a positive constant $C$ independent of $\beta,H$ and $h$, such that 
\begin{align*}
&\bigg|\widetilde{\mathcal{B}}\big({\bf E}_{\C}(\lambdae),{\bf E}_{\C}(\vpg)\big)
-(\U{\lambdae},\g)_\LO2 \bigg|\leq CH\beta^{-\frac14}\|{\bf E}_{\C}(\lambdae)\|_{\Aa}\|\g\|_{\LO2}.
\end{align*}
\end{lemma}
Using similar proofs as~\cite[Lemma 7.19]{TZAD2021}, \cite[Lemma 7.11]{TuLi:2008:BDDCAD}, we obtain the two following lemmas.
\begin{lemma}\label{Adualestimate}
Given $\g\in {\bf W}$, under the Assumption~\ref{assumpconstrain},  
we have 
$$
\big\|{\bf E}_{\C}(\vpgt-\vpg)\big\|_{\Aa}\leq C{H}{\beta^{-\frac14}}\|\g\|_{\LO2}.
$$
\end{lemma}
\begin{lemma}\label{Bwlambdae0}
For any $\lambdae_\Gamma\in\widehat{\Lambdae}_\Gamma$, recall that $\w_\Gamma=\widetilde{\S}^{-1}_\Gamma\widetilde{\R}_{\D,\Gamma}\widehat{\S}_\Gamma\lambdae_\Gamma$, we have 
\[
\widetilde{\mathcal{B}}\big({\bf E}_{\C}(\w_{\A,\Gamma}-\widetilde{\R}(\lambdae_{\A,\Gamma})),{\bf E}_{\C}({\bm \xi})\big)=0,
\]
for any ${\bm\xi}\in\widetilde{\R}(\Lambdae)$.
\end{lemma}

\begin{lemma}\label{Ulambdaw}
Given $\lambdae_\Gamma \in \widehat{\Lambdae}_\Gamma$, recall $\w_\Gamma = \widetilde{\S}_\Gamma^{-1} \widetilde{\R}_{\D,\Gamma} \S_\Gamma \lambdae_\Gamma.$
There exists a positive constant $C$ independent of $\beta, H$ and $h$, such that
\begin{align}
\|{\U}(\w_{\A,\Gamma} - \lambdae_{\A,\Gamma})\|_{\T_h}
\leq CHc_0(1+\beta^{-\frac12})\|\w_{\Gamma}\|_{\widetilde\M_\Gamma}\label{UA1wlambdae}.
\end{align}
\end{lemma}
\begin{proof}
Given ${\g}\in {\W}$, we first consider the estimate of $({\U}({\w}_{\A,\Gamma}-\lambdae_{\A,\Gamma}),{\g})_{\T_h}$. Since
\begin{align*}
&({\U}({\w}_{\A,\Gamma}-\lambdae_{\A,\Gamma}),{\g})_{\T_h}
=\widetilde{\mathcal{B}}\big({\bf E}_{\C}(\w_{\A,\Gamma}),{\bf E}_{\C}(\vpgt)\big)-{\mathcal{B}}_h\big({\bf E}_{\C}(\lambdae_{\A,\Gamma}),{\bf E}_{\C}(\vpg)\big)\\
&=\widetilde{\mathcal{B}}\big({\bf E}_{\C}(\w_{\A,\Gamma}),{\bf E}_{\C}(\vpgt)\big)-\widetilde{\mathcal{B}}\big({\bf E}_{\C}(\lambdae_{\A,\Gamma}),{\bf E}_{\C}(\vpg)\big)
=\widetilde{\mathcal{B}}\big({\bf E}_{\C}(\w_{\A,\Gamma}),{\bf E}_{\C}(\vpgt-\vpg)\big),
\end{align*}
where the first inequality follows from~\eqref{eqdual1} and \eqref{eqdual2}, the last equality follows from Lemma \ref{Bwlambdae0} and the fact that 
$\widetilde{\R}$ is the injection operator from ${\Lambdae}$ into $\widetilde{\Lambdae}$ with ${\Lambdae} \subset \widetilde{\Lambdae}$. Then, we obtain that 
\begin{align*}
&({\U}({\w}_{\A,\Gamma}-\lambdae_{\A,\Gamma}),{\g})_{\T_h}=\widetilde{\mathcal{B}}\big({\bf E}_{\C}(\w_{\A,\Gamma}),{\bf E}_{\C}(\vpgt-\vpg)\big)\\
&=(Q_{A_1}\xg,U_{A_1}\xg,\widetilde\varphi_g-\varphi_g)
\begin{bmatrix}
\beta^{\frac12}A^{ad}_1& -L \end{bmatrix}{\bf E}_{\C}(\w_{\A,\Gamma})\\
&\quad+(Q_{A_2}\xg,U_{A_2}\xg,\widetilde\varphi_{z_d}-\varphi_{z_d})\begin{bmatrix}
L & A^{ad}_2\end{bmatrix}{\bf E}_{\C}(\w_{\A,\Gamma})\\
&:=I+{II}.
\end{align*}
 Next we will estimate the term ${I}$, and term ${II}$ can be obtained similarly. Denote ${\bf\xg}=\vpgt-\vpg$ and 
 assume $\w_{\A,\Gamma}=(w_{A,\Gamma},x_{A,\Gamma})^T$, we have
\begin{align*}
|{I}|&=\bigg|\beta^{\frac12}({Q}_{A_1}\w_{\A,\Gamma},Q_{A_1}\xg)_{\T_h}+\beta^{\frac12}\l(\tau_1-\frac12\bz\cdot{\bf n})(U_{A_1}\w_{\A,\Gamma}-w_{A,\Gamma}),U_{A_1}\xg-(\widetilde\varphi_g-\varphi_g)\r_{\partial\T_h}\\
&\quad-\frac{\beta^{\frac12}}2(U_{A_1}\w_{\A,\Gamma},\bz\cdot\nabla U_{A_1}\xg)_{\T_h}+\frac{\beta^{\frac12}}2(U_{A_1}\xg,\bz\cdot\nabla U_{A_1}\w_{\A,\Gamma})_{\T_h}\\
&\quad+\frac{\beta^{\frac12}}2\l\bz\cdot{\bf n}w_{A,\Gamma},U_{A_1}\xg\r_{\partial\T_h}-\frac{\beta^{\frac12}}2
\l\bz\cdot{\bf n}(\widetilde\varphi_g-\varphi_g),U_{A_1}\w_{\A,\Gamma}\r_{\partial\T_h}-(U_{A_2}\w_{\A,\Gamma},U_{A_1}\xg)_{\T_h}\bigg|\\
&=\bigg|\beta^{\frac12}({Q}_{A_1}\w_{\A,\Gamma},Q_{A_1}\xg)_{\T_h}+\beta^{\frac12}\l(\tau_1-\frac12\bz\cdot{\bf n})(U_{A_1}\w_{\A,\Gamma}-w_{A,\Gamma}),U_{A_1}\xg-(\widetilde\varphi_g-\varphi_g)\r_{\partial\T_h}\\
&\quad-\frac{\beta^{\frac12}}2\l(\bz\cdot{\bf n})(U_{A_1}\w_{\A,\Gamma}-w_{A,\Gamma}),{U}_{A_1}\xg\r_{\partial\T_h}+\beta^{\frac12}(U_{A_1}\xg,\bz\cdot\nabla U_{A_1}\w_{\A,\Gamma})_{\T_h}\\
&\quad-\frac{\beta^{\frac12}}2\l(\bz\cdot{\bf n})(\widetilde\varphi_g-\varphi_g),U_{A_1}\w_{\A,\Gamma}-w_{A,\Gamma}\r_{\partial\T_h}-(U_{A_1}\xg,U_{A_2}\w_{\A,\Gamma})_{\T_h}-\frac{\beta^{\frac12}}2\l\bz\cdot{\bf n}w_{A,\Gamma},\widetilde\varphi_g-\varphi_g\r_{\partial\T_h}\bigg|\\
&\leq \|{\bf E}_{\C}(\w_{\A,\Gamma})\|_{\Aa}\|{\bf E}_{\C}({\xg})\|_{\Aa}+C(1+\beta^{\frac14})h^{\frac12}\|{\bf E}_{\C}(\w_{\A,\Gamma})\|_{\Ma}(\beta^{\frac14}h^{-\frac12}\|U_{A_1}\xg\|_{\T_h})\\
&\quad+C(1+\beta^{\frac14})\|{\bf E}_{\C}(\w_{\A,\Gamma})\|_{\Ma}\|U_{A_1}\xg\|_{\T_h}+C(1+\beta^{\frac14})h^{\frac12}\|{\bf E}_{\C}(\w_{\A,\Gamma})\|_{\Ma}(\beta^{\frac14}\|\widetilde\varphi_g-\varphi_g\|_{\partial\T_h})\\
&\quad+\|U_{A_1}\xg\|_{\T_h}\|U_{A_2}\w_{\A,\Gamma}\|_{\T_h}+CH\|{\bf E}_{\C}(\w_{\A,\Gamma})\|_{\Ma}\|{\bf E}_{\C}(\xg)\|_{\Aa}\\
&\leq \|{\bf E}_{\C}(\w_{\A,\Gamma})\|_{\Aa}\|{\bf E}_{\C}({\xg})\|_{\Aa}\\
&\quad+C\|{\bf E}_{\C}(\w_{\A,\Gamma})\|_{\Ma}\bigg((1+\beta^{\frac14})^2\|U_{A_1}\xg\|_{\T_h}+(1+\beta^{\frac14})\beta^{\frac14}\|\widetilde\varphi_g-\varphi_g\|_h+H\|{\bf E}_{\C}({\xg})\|_{\Aa}\bigg)\\
&\leq  \|{\bf E}_{\C}(\w_{\A,\Gamma})\|_{\Aa}\|{\bf E}_{\C}({\xg})\|_{\Aa}\\
&\quad+C\|{\bf E}_{\C}(\w_{\A,\Gamma})\|_{\Ma}\bigg((1+\beta^{\frac14})^2c_0\beta^{-\frac14}\|{\bf E}_{\C}({\xg})\|_{\Aa}+(1+\beta^{\frac14})\|{\bf E}_{\C}({\xg})\|_{\Aa}+H\|{\bf E}_{\C}({\xg})\|_{\Aa}\bigg)\\
&\leq  Cc_0(\beta^{\frac14}+\beta^{-\frac14})\|{\bf E}_{\C}(\w_{\A,\Gamma})\|_{\Ma}\|{\bf E}_{\C}({\xg})\|_{\Aa}\\
&\leq Cc_0H(1+\beta^{-\frac12})\|{\bf E}_{\C}(\w_{\A,\Gamma})\|_{\Ma}\|\g\|_{\T_h},
\end{align*}
where the first equality follows from the divergence theorem with $\nabla\cdot{\bz}=0$, namely
\[
(U_{A_1}\w_{\A,\Gamma},\bz\cdot\nabla U_{A_1}\xg)_{\T_h}=\l(\bz\cdot{\bf n})U_{A_1}\w_{\A,\Gamma},U_{A_1}\xg\r_{\partial\T_h}-(\bz\cdot\nabla U_{A_1}\w_{\A,\Gamma}, U_{A_1}\xg )_{\T_h},
\]
the first inequality follows from~\eqref{UAlambdaMestimate}, a trace inequality, ~\eqref{gradient1ThM} and Lemma~\ref{bdulambda}, the second inequality follows from ~\eqref{eq:Manorm} and the fact that $\|{\bf E}_{\C}(\w_{\A,\Gamma})\|_{\Aa}\leq \|{\bf E}_{\C}(\w_{\A,\Gamma})\|_{\Ma}$, the third to last inequality follows from~\eqref{Uestimate2Th}, Lemma~\ref{citeadlemma7.10} and~\eqref{Anormtriplrbar},  
the last inequality follows from~Lemma~\ref{Adualestimate}.
Therefore, we have
\begin{align*}
\|{\U}({\w}_{\A,\Gamma}-\lambdae_{\A,\Gamma})\|_{\T_h}&=\sup\limits_{\s\in\widetilde\Lambdae}\frac{({\U}({\w}_{\A,\Gamma}-\lambdae_{\A,\Gamma}),\U\s)_\LO2}{\|\U\s\|_{\LO2}}=\sup\limits_{\s\in\widetilde\Lambdae, \g=\U\s}\frac{({\U}({\w}_{\A,\Gamma}-\lambdae_{\A,\Gamma}),\g)_\LO2}{\|\g\|_{\LO2}}\\
&\leq Cc_0H(1+\beta^{-\frac12})\|{\bf E}_{\C}(\w_{\A,\Gamma})\|_{\Ma}
=Cc_0H(1+\beta^{-\frac12})\|\w_\Gamma\|_{\widetilde\M_\Gamma},
\end{align*}
{where the last equality follows from~\eqref{MaMGamma}.}
\end{proof}

\subsection{Estimates of $\|\mathcal{H}_{\A^{\ad}}(\w_{\Gamma}-\lambdae_\Gamma)\|_h$ and $\|\mathcal{H}_{\Aad}( {\bf T}\lambdae_{\Gamma} - \lambdae_{\Gamma})\|_h$.}
Given any $\lambdae \in \widetilde{\Lambdae}$, let $\lambdae^K$ denote the restriction of $\lambdae$ on $\partial K$, where $K$ is an element in the triangulation $\mathcal{T}_h$. Let $E_i$ represent the edge of $K$ that lies opposite to the $i$-th vertex of $K$. Following~\cite[Equations (7.33)-(7.35)]{TZAD2021}, we define the local lifting operator $S_i^K \lambdae^K$ in the polynomial space $P_{k+1}(K)\times P_{k+1}(K)$ such that:
\begin{align}
\langle S_i^K \lambdae^K, {\bm\eta} \rangle_{E_i} &= \langle \lambdae^K, {\bm\eta} \rangle_{E_i}, && \forall  {\bm\eta} \in P_{k+1}(E_i)\times P_{k+1}(E_i), \label{Snorm1} \\
(S_i^K \lambdae^K, \mathbf{v})_K &= ({\U}_{\Aad}({\lambdae^K}), \mathbf{v})_K, && \forall \, \mathbf{v} \in P_k(K)\times P_k(K),\label{Snorm2} 
\end{align} 
for $i = 1, \dots, n+1$, where $n+1$ denotes the total number of edges of $K$. We then define the global inner product and corresponding norm over the mesh $\mathcal{T}_h$ by:
\begin{equation} \label{normexpansion}
(\lambdae, \s)_S = \sum_{K \in \mathcal{T}_h} \frac{1}{n+1} \sum_{i=1}^{n+1} (S_i^K \lambdae, S_i^K \s)_K, 
\quad \text{and} \quad \|\lambdae\|_J^2 = (\lambdae, \lambdae)_S.
\end{equation}
Denote ${Q}_k$ as the {{$L^2$ projection from $P_{k+1}(K)\times P_{k+1}(K)$ into $P_k(K)\times P_k(K)$}}, following similar process as~\cite[Lemma 5.7]{CDGT2014}, we have:
\begin{lemma}
\begin{align}
&{\U}_{\Aad}(\lambdae^K) = Q_k(S_i^K \lambdae^K)_K,\label{S1}\\
&c \|\lambdae\|_h \leq \|\lambdae\|_S\label{S2},\\
&|S_i^K \lambdae|_{H^1(K)} \leq C  \interleave\lambdae\interleave_K\label{S3}.
\end{align}
\end{lemma}

Recall the definitions of bilinear forms ${\mathcal{B}}^{\ad}_h$ and $\widetilde{\mathcal{B}}^{\ad}$ in~\eqref{cadBh} and~\eqref{cBad}. Given ${\bf s}\in\widetilde{\Lambdae}$, let $\vps$  and $\vpst$ be the solutions of the following equations, 

\begin{align}
{\mathcal{B}}^{\ad}_h\big({\bf E}_{\Aad}(\lambdae),{\bf E}_{\Aad}(\vps)\big)&=({\bf U}_{\Aad}\lambdae,{\bf g})_{\Omega}\quad \forall \lambdae=(\lambda,\mu)^T \in\Lambdae,\\
\widetilde{\mathcal{B}}^{\ad}\big({\bf E}_{\Aad}(\widetilde\lambdae),{\bf E}_{\Aad}(\vpst)\big)&=({\bf U}_{\Aad}\widetilde\lambdae,{\bf g})_{\Omega}\quad \forall \lambdae=(\widetilde\lambda,\widetilde\mu)^T \in \widetilde\Lambdae,
\end{align}
where $\g$ is defined as
\begin{equation}\label{gs}
\g = \frac{1}{n+1} \sum_{i=1}^{n+1} Q_k({S}_i^K \s).
\end{equation}
Now, we are ready to provide the proof of Lemma~\ref{wdnormS}.

\subsection{Proof of Lemma~\ref{wdnormS}}

\begin{proof}
{
{
Given $\s= (s, t)^T \in \widetilde{\Lambdae}$, and defining ${\g}$ as in~\eqref{gs}, by~\eqref{normexpansion} we observe that
\begin{align*}
&(\mathcal{H}_{\Aad}( \w_{\Gamma} -\lambdae_{\Gamma}),\s)_{S}=\sum_{K \in \mathcal{T}_h} \frac{1}{n+1} \sum_{i=1}^{n+1} \big(S_i^K \mathcal{H}_{\Aad}( \w_{\Gamma} -\lambdae_{\Gamma}), S_i^K \s\big)_K,\\
&=\big(\U_{\Aad}(\mathcal{H}_{\A^{\ad}}(\w_{\Gamma}-\lambdae_\Gamma)),\g\big)_{\T_h}+ \sum_{K \in \mathcal{T}_h(\Omega)} \frac{1}{n+1} \sum_{i=1}^{n+1} 
\big( (I-{Q}_k) S_i^K ({\mathcal{H}}_{\A^{ad}}(\w_{\Gamma}-\lambdae_\Gamma)), (I-{Q}_k) S_i^K \s \big)_K.
\end{align*}
Therefore, we have 
\begin{equation}\label{HADSnorm}
\begin{aligned}
&\big| \big( \mathcal{H}_{\A^{ad}}(\w_{\Gamma} - \lambdae_\Gamma), \s \big)_{S} \big| 
\leq \big| \big( \U_{\Aad} \big( \mathcal{H}_{\A^{ad}}(\w_{\Gamma}-\lambdae_\Gamma) \big), \g \big)_{\T_h} \big| \\
&\quad + \bigg| \sum_{K \in \mathcal{T}_h(\Omega)} \frac{1}{n+1} \sum_{i=1}^{n+1} 
( (1 - Q_k) S^K_i \big( \mathcal{H}_{\A^{ad}}(\w_{\Gamma} -\lambdae_\Gamma) \big), 
(1 - Q_k) S_i^K \s \big)_K \bigg|.
\end{aligned}
\end{equation}
Denote $\mathcal{H}_{\A}(\lambdae_\Gamma)= \lambdae_{\A,\Gamma}$.
Following the similar argument as in \cite[Equation (7.58)]{TZAD2021}, we obtain
\begin{equation}\label{IQestimate}
\begin{aligned}
&\bigg| \sum_{K \in \mathcal{T}_h(\Omega)} \frac{1}{n+1} \sum_{i=1}^{n+1} 
\big( (I- Q_k) S^K_i \big( \mathcal{H}_{\A^{ad}}(\w_{\Gamma}-\lambdae_\Gamma) \big), 
(I- Q_k) S_i^K \s \big)_K \bigg| \\
&\leq h \interleave \mathcal{H}_{\A^{ad}}(\w_{\Gamma} -\lambdae_\Gamma) \interleave \|\s\|_S 
\leq C h \| {\bf E}_{\Aad} \left( \mathcal{H}_{\Aad}(\w_{\Gamma}-\lambdae_\Gamma) \right) \|_{\Aad} \|\s\|_S\\
&\leq Ch\beta^{-\frac14}\big\| {\bf E}_{\C} \big( \mathcal{H}_{\A}(\w_{\Gamma}-\lambdae_\Gamma) \big) \big\|_{\Aa} \|\s\|_S\leq Ch\beta^{-\frac14}\big(\|\w_\Gamma\|_{\widetilde{\S}_\Gamma}+\|\lambdae_\Gamma\|_{\widehat\S_\Gamma}\big) \|\s\|_S,
\end{aligned}
\end{equation}
where the first inequality follows from a Friedrichs inequality and ~\eqref{S3} with the fact that 
$$\big\|(I-Q_k)S^K_i (\mathcal{H}_{\A^{ad}}(\w_{\Gamma} -\lambdae_\Gamma))\big\|_{L^2(K)}\leq Ch\big|S^K_i \mathcal{H}_{\A^{ad}}(\w_{\Gamma} -\lambdae_\Gamma)\big|_{H^1(K)}
\leq Ch\interleave\mathcal{H}_{\A^{ad}}(\w_{\Gamma} - \lambdae_{\Gamma})\interleave_K,$$  
and $\|(I-Q_k)S^K_i\s\|_K\leq C\|S^K_i\s\|_K,$ the second inequality follows from~\eqref{ADestimateE}, the second to last inequality follows from~\eqref{great1AAd}, and the last inequality follows from Lemma~\ref{MSrelation}.

Next, we consider the estimate of $\big| \big( \U_{\Aad} \big( \mathcal{H}_{\A^{ad}}(\w_{\Gamma} -\widetilde{\R}_\Gamma \lambdae_\Gamma) \big), \g \big)_{\T_h} \big|
$. We have
\begin{equation}\label{uhnormad1}
\begin{aligned}
&\big|\left(\U_{\Aad}(\mathcal{H}_{\A^{\ad}}(\w_{\Gamma}-\lambdae_\Gamma)),\g\right)_{\T_h}\big|\\
&=\big|\big(\U_{\Aad}\big(\mathcal{H}_{\A^{\ad}}(\w_{\Gamma}-\lambdae_\Gamma)\big)-\U_{\Aad}\big(\mathcal{H}_{\A}\big(\w_{\Gamma}-\lambdae_\Gamma)\big), \g\big)_{\T_h}\\
&\quad+\big(\U_{\Aad}(\mathcal{H}_{\A}(\w_{\Gamma}-\lambdae_\Gamma))-\U\big(\mathcal{H}_{\A}(\w_{\Gamma}-\lambdae_\Gamma)\big), \g\big)_{\T_h}+\big(\U\big(\mathcal{H}_{\A}(\w_{\Gamma}-\lambdae_\Gamma)\big), \g\big)_{\T_h}\big|\\
&\leq \|\mathcal{H}_{\A^{\ad}}(\w_{\Gamma}-\lambdae_\Gamma)-\mathcal{H}_{\A}(\w_{\Gamma}-\lambdae_\Gamma)\|_h\|\g\|_{\T_h}+\|\U(\mathcal{H}_{\A}(\w_{\Gamma}-\lambdae_\Gamma))\|_{\T_h}\|\g\|_{\T_h}\\
&\quad+\big|\big(\U_{\Aad}\big(\mathcal{H}_{\A}(\w_{\Gamma}-\lambdae_\Gamma)\big)-\U\big(\mathcal{H}_{\A}(\w_{\Gamma}-\lambdae_\Gamma)\big), \g\big)_{\T_h}\big|,
\end{aligned}
\end{equation}
where the inequality follows from~\eqref{UA1L2} and~\eqref{UA2L2}.

We firstly estimate the term $\|\mathcal{H}_{\A^{\ad}}(\w_{\Gamma}-\lambdae_\Gamma)-\mathcal{H}_{\A}(\w_{\Gamma}-\lambdae_\Gamma)\|_h$. Since $\mathcal{H}_{\A^{\ad}}(\w_{\Gamma}-\lambdae_\Gamma)-\mathcal{H}_{\A}(\w_{\Gamma}-\lambdae_\Gamma)$ has ${\bf 0}$ values on the subdomain boundary, we obtain that  
\begin{equation}\label{uhnormad2}
\begin{aligned}
&\big\|\mathcal{H}_{\A^{\ad}}(\w_{\Gamma}-\lambdae_\Gamma)-\mathcal{H}_{\A}(\w_{\Gamma}-\lambdae_\Gamma)\big\|_h\\
&=\bigg(\sum\limits_{i=1}^N\|\mathcal{H}_{\A^{\ad}}(\w_{\Gamma}-\lambdae_\Gamma)-\mathcal{H}_{\A}(\w_{\Gamma}-\lambdae_\Gamma)\|^2_{h,\Omega_i}\bigg)^{\frac12}\\
&\leq\bigg(\sum\limits_{i=1}^NH^2\interleave\mathcal{H}_{\A^{\ad}}(\w_{\Gamma}-\lambdae_\Gamma)-\mathcal{H}_{\A}(\w_{\Gamma}-\lambdae_\Gamma)\interleave^2_{\Omega_i}\bigg)^{\frac12}\\
&\leq CH\interleave\mathcal{H}_{\A^{\ad}}(\w_{\Gamma}-\lambdae_\Gamma)-\mathcal{H}_{\A}(\w_{\Gamma}-\lambdae_\Gamma)\interleave\\
&\leq CH\|{\bf E}_{\Aad}(\mathcal{H}_{\A^{\ad}}(\w_{\Gamma}-\lambdae_\Gamma))\|_{\Aad}+CH\beta^{-\frac14}\|{\bf E}_{\C}(\mathcal{H}_{\A}(\w_{\Gamma}-\lambdae_\Gamma))\|_{\Aa}\\
&\leq CH\beta^{-\frac14}\|{\bf E}_{\C}(\mathcal{H}_{\A}(\w_{\Gamma}-\lambdae_\Gamma))\|_{\Aa}\\
&\leq CH\beta^{-\frac14}(\|\w_\Gamma\|_{\widetilde\S_\Gamma}+\|\lambdae_\Gamma\|_{\widehat\S_\Gamma}),
\end{aligned}
\end{equation}
where the first inequality follows from~Lemma~\ref{citeadlemma7.8}, the second inequality follows from~\eqref{ADestimateE} and~\eqref{Anormtriplrbar}, the second to last inequality follows from~\eqref{great1AAd}, and the last inequality follows from~Lemma~\ref{MSrelation}. 

Then, for the term $\big|\big(\U_{\Aad}\big(\mathcal{H}_{\A}(\w_{\Gamma}-\lambdae_\Gamma)\big)-\U\big(\mathcal{H}_{\A}(\w_{\Gamma}-\lambdae_\Gamma)\big), \g\big)_{\T_h}\big|$,  by~\eqref{UUaddifference}, we have 
\begin{equation}\label{uhnormad3}
\begin{aligned}
&\big|\big(\U_{\Aad}(\mathcal{H}_{\A}(\w_{\Gamma}-\lambdae_\Gamma))-\U\big(\mathcal{H}_{\A}(\w_{\Gamma}-\lambdae_\Gamma)), \g\big)_{\T_h}\big|\\
&\leq \big\|\U_{\Aad}\big(\mathcal{H}_{\A}(\w_{\Gamma}-\lambdae_\Gamma))-\U\big(\mathcal{H}_{\A}(\w_{\Gamma}-\lambdae_\Gamma\big)\big)\big\|_{\T_h}\|\g\|_{\T_h}\\
&\leq Ch\beta^{-\frac12}\|\U(\mathcal{H}_{\A}(\w_{\Gamma}-\lambdae_\Gamma))\|_{\T_h} \|\g\|_{\T_h}.
\end{aligned}
\end{equation}
Combining~\eqref{uhnormad1},~\eqref{uhnormad2},~\eqref{uhnormad3}, Lemma~\ref{Ulambdaw}, with the fact that $\|\g\|_{\T_h}\leq \|\s\|_S$ which is proved in~\cite[Lemma 7.20]{TZAD2021}, we have
\begin{equation}\label{UADwlambdae}
\begin{aligned}
&\big|\left(\U_{\Aad}(\mathcal{H}_{\A^{\ad}}(\w_{\Gamma}-\lambdae_\Gamma)),\g\right)_{\T_h}\big|\\
&\leq CH\beta^{-\frac14}(\|\w_\Gamma\|_{\widetilde\S_\Gamma}+\|\lambdae_\Gamma\|_{\widehat\S_\Gamma})\|\g\|_{\T_h}+\|\U\big(\mathcal{H}_{\A}(\w_{\Gamma}-\lambdae_\Gamma)\big)\|_{\T_h}\|\g\|_{\T_h}\\
&\quad+(h\beta^{-\frac12})\|\U\big(\mathcal{H}_{\A}(\w_{\Gamma}-\lambdae_\Gamma)\big)\|_{\T_h}\|\g\|_{\T_h}\\
&\leq CH\beta^{-\frac14}(\|\w_\Gamma\|_{\widetilde\S_\Gamma}+\|\lambdae_\Gamma\|_{\widehat\S_\Gamma})\|\g\|_{\T_h}+C(1+h\beta^{-\frac12})\|\U\big(\mathcal{H}_{\A}(\w_{\Gamma}-\lambdae_\Gamma)\big)\|_{\T_h}\|\g\|_{\T_h}\\
&\leq CH\beta^{-\frac14}(\|\w_\Gamma\|_{\widetilde\S_\Gamma}+\|\lambdae_\Gamma\|_{\widehat\S_\Gamma})\|\g\|_{\T_h}+C\bigg(H(1+\beta^{-\frac12})c^2_0\|{\bf w}_\Gamma\|_{\widetilde{\M}_\Gamma}\bigg)\|\g\|_{\T_h}\\
&\leq C\bigg((H\beta^{-\frac12}+H)c^2_0+H\beta^{-\frac14}\bigg)(\|\w\|_{\widetilde{\M}_\Gamma}+\|\lambdae_\Gamma\|_{\widehat\S_\Gamma})\|\s\|_{S}.
\end{aligned}
\end{equation}
Thus, by~\eqref{S2},~\eqref{HADSnorm},~\eqref{IQestimate} and ~\eqref{UADwlambdae}, we obtain that  
\begin{equation}\label{wlambdaadJ}
\begin{aligned}
\|{\mathcal{H}}_{\A^{ad}}(\w_{\Gamma}-\lambdae_\Gamma)\|_h&\leq \|{\mathcal{H}}_{\A^{ad}}(\w_{\Gamma}-\lambdae_\Gamma)\|_S=\sup\limits_{\s\in\widetilde\Lambdae}
\frac{ \left( {\mathcal{H}}_{\A^{ad}}(\w_{\Gamma} - \lambdae_\Gamma), \s \right)_{S} }{\|\s\|_{S}}\\
&\leq C\bigg((H\beta^{-\frac12}+H)c^2_0+H\beta^{-\frac14}\bigg)\big(\|\w\|_{\widetilde{\M}_\Gamma}+\|\lambdae_\Gamma\|_{\widehat\S_\Gamma}\big).
\end{aligned}
\end{equation}
Moreover, let $\widehat{\R}_{\Gamma}^{(i)}, \overline{\R}_{\Gamma}^{(i)} $ be the restriction operators 
from $\widehat{\Lambdae}_\Gamma$ and $\widetilde{\Lambdae}_\Gamma$ to $\Lambdae^{(i)}_\Gamma$ respectively. By the similar argument as in the proof of~\cite[Lemma 6.10]{TZAD2021}, we have
\begin{align*}
&\|\mathcal{H}_{\Aad}( {\bf T}\lambdae_{\Gamma} - \lambdae_{\Gamma})\|_h \\
&\leq C\bigg(\sum_{i=1}^{N} 
\big\| \mathcal{H}_{\Aad} \big( \widehat{\R}_{\Gamma}^{(i)} \widetilde{\R}_{\D,\Gamma}^{T} \w_{\Gamma} - \overline{\R}_{\Gamma}^{(i)} \w_{\Gamma} \big)\|^2_{h, \Omega_i}\bigg)^{\frac12} 
+C\bigg(\sum_{i=1}^{N}\| \mathcal{H}_{\Aad} \big( \overline{\R}_{\Gamma}^{(i)} \w_{\Gamma} - \widehat{\R}_{\Gamma}^{(i)} \lambdae_{\Gamma} \big)\|^2_{h, \Omega_i} 
\bigg)^{\frac12}\\
&:=I+II.
\end{align*}
For the estimate of $I$, since $ \widetilde{\R}_{\D,\Gamma}^{T} \w_{\Gamma}$ and $\w_{\Gamma}$ have the same edge average on the subdomain boundary $\partial\Omega_i$, then
\begin{align*}
I&\leq C\bigg(\sum\limits_{i=1}^N H^2\interleave\mathcal{H}_{\Aad} \left( \widehat{\R}_{\Gamma}^{(i)} \widetilde{\R}_{\D,\Gamma}^{T} \w_{\Gamma} - \overline{\R}_{\Gamma}^{(i)} \w_{\Gamma} \right) \interleave^2_{\T_h(\Omega_i)}\bigg)^{\frac12}\\
&\leq CH\beta^{-\frac14}\bigg(\sum\limits_{i=1}^N \|\widehat{\R}_{\Gamma}^{(i)} \widetilde{\R}_{\D,\Gamma}^{T} \w_{\Gamma} - \overline{\R}_{\Gamma}^{(i)} \w_{\Gamma}\|^2_{\widetilde\S_\Gamma}\bigg)^{\frac12}\leq CH\beta^{-\frac14}C_{ED}\|\w_\Gamma\|_{\widetilde\S_\Gamma},
\end{align*}
where the first inequality follows from~Lemma~\ref{citeadlemma7.8}, the second inequality follows from~\eqref{Anormtriplrbar} and Lemma~\ref{MSrelation}, 
and the last inequality follows from~\eqref{less1SGammav}. For the estimate of $II$, by~\eqref{wlambdaadJ} we have 
\begin{align*}
II\leq \|\mathcal{\H}_{\Aad}(\w_\Gamma-\lambdae_\Gamma)\|_h\leq
C\bigg((H\beta^{-\frac12}+H)c^2_0+H\beta^{-\frac14}\bigg)\big(\|\w\|_{\widetilde{\M}_\Gamma}+\|\lambdae_\Gamma\|_{\widehat\S_\Gamma}\big).
\end{align*} }}
Therefore, we obtain that 
\[
\|\mathcal{H}_{\Aad}( {\bf T}\lambdae_{\Gamma} - \lambdae_{\Gamma})\|_h \leq C\bigg((H\beta^{-\frac12}+H)c^2_0+H\beta^{-\frac14}C_{ED}\bigg)(\|\w\|_{\widetilde{\M}_\Gamma}+\|\lambdae_\Gamma\|_{\widehat\S_\Gamma}).
\]
\end{proof}

\subsection{Proof of Lemma~\ref{Chigamma}.}

\begin{proof}
{
We first give the proof of ~\eqref{wlambdaL1}, let $$\z_\Gamma=\widetilde{\L}^{-1}_\Gamma\widetilde{\bf R}_{\D,\Gamma}{\L}_\Gamma\lambdae_\Gamma=
\widetilde{\L}^{-1}_\Gamma\widetilde{\bf R}_{\D,\Gamma}\widetilde{\R}^T_\Gamma\widetilde{\L}_\Gamma\widetilde{\R}_\Gamma\lambdae_\Gamma
=\widetilde{\L}^{-1}_\Gamma {\bf E}^T_\D\widetilde{\L}_\Gamma\widetilde{\R}_\Gamma\lambdae_\Gamma,$$ then
\begin{align*}
\l\w_\Gamma,\z_{\Gamma}\r_{\widetilde{\L}_\Gamma}-\l\w_\Gamma,\widetilde{\R}_\Gamma\lambdae_\Gamma\r_{\widetilde{\L}_\Gamma}&=
\lambdae^T_\Gamma\widetilde{\R}^T_\Gamma\widetilde{\L}^T_\Gamma{\bf E}_\D{\bf w}_\Gamma-\l\w_\Gamma,\widetilde{\R}_\Gamma\lambdae_\Gamma\r_{\widetilde{\L}_\Gamma}=\l{\bf E}_\D{\bf w}_\Gamma,\widetilde{\R}_\Gamma\lambdae_\Gamma\r_{\widetilde{\L}_\Gamma}-\l\w_\Gamma,\widetilde{\R}_\Gamma\lambdae_\Gamma\r_{\widetilde{\L}_\Gamma}\\
&=\l{\bf E}_\D{\bf w}_\Gamma-{\bf w}_\Gamma,\widetilde{\R}_\Gamma\lambdae_\Gamma\r_{\widetilde{\L}_\Gamma}.
\end{align*}
We have
\begin{align*}
&\l{\bf E}_\D{\bf w}_\Gamma-{\bf w}_\Gamma,\widetilde{\R}_\Gamma\lambdae_\Gamma\r_{\widetilde{\L}_\Gamma}\\
&\leq 
\|\widetilde{\R}_\Gamma\lambdae_\Gamma\|_{\widetilde{\L}_\Gamma}\|{\bf E}_\D{\bf w}_\Gamma-{\bf w}_\Gamma\|_{\widetilde{\L}_\Gamma}=
\|\lambdae_\Gamma\|_{{\L}_\Gamma}\|{\bf E}_\D{\bf w}_\Gamma-{\bf w}_\Gamma\|_{\widetilde{\L}_\Gamma}=\|\lambdae_\Gamma\|_{{\L}_\Gamma}\|{\bf U}(\mathcal{H}_{\A}({\bf E}_\D{\bf w}_\Gamma-{\bf w}_\Gamma))\|_{\T_h}\\
&\leq \|\lambdae_\Gamma\|_{{\L}_\Gamma}\left(c_0\|\mathcal{H}_{\A}({\bf E}_\D{\bf w}_\Gamma-{\bf w}_\Gamma)\|_{h}\right)\leq  C\|\lambdae_\Gamma\|_{{\L}_\Gamma}\left(c_0 H \interleave\mathcal{H}_{\A,\Gamma}({\bf E}_{\D}{\bf w}_\Gamma-{\bf w}_\Gamma)\interleave\right)\\
&\leq C\|\lambdae_\Gamma\|_{{\L}_\Gamma}(c_0 H\beta^{-\frac14}\|{\bf E}_{\D}{\bf w}_\Gamma-{\bf w}_\Gamma\|_{\widetilde{\S}_\Gamma})\leq Cc_0 H\beta^{-\frac14}C_{ED}
\|\lambdae_\Gamma\|_{{\L}_\Gamma}\|{\bf w}_\Gamma\|_{\widetilde{\S}_\Gamma}\\
&= Cc_1\|\lambdae_\Gamma\|_{{\L}_\Gamma}\|{\bf w}_\Gamma\|_{\widetilde{\S}_\Gamma}\leq Cc_1\|\lambdae_\Gamma\|_{{\M}_\Gamma}\|{\bf w}_\Gamma\|_{\widetilde{\M}_\Gamma},
\end{align*}
where $c_1:=H\beta^{-\frac14}C_{ED}$,
the second inequality follows from~\eqref{Uestimate3} in Lemma~\ref{UL2K}, the third inequality follows from Lemma \ref{citeadlemma7.10}, the fourth inequality follows from~\eqref{Anormtriplrbar}, the second to last inequality follows from~\eqref{CED}, and the last inequality follows from~Lemma~\ref{MSrelation}.}

Next, we give the proof of \eqref{wlambdaL4}. We have
\begin{equation*}
\begin{aligned}
&\langle \lambdae_\Gamma, {\bf T}\lambdae_{\Gamma} - \lambdae_{\Gamma} \rangle_{\Z_\Gamma} \\
&\leq C\gamma_1\|\lambdae_\Gamma\|_{{\M}_\Gamma}\left(C(H\beta^{-\frac14})\|{\bf T}\lambdae_\Gamma - \lambdae_\Gamma \|_{\widehat\S_\Gamma}+\|\H_{\A^{ad}}({\bf T}\lambdae_\Gamma - \lambdae_\Gamma) \|_{h}\right)\\
&\leq CH\beta^{-\frac14}\gamma_1\|\lambdae_\Gamma\|_{{\M}_\Gamma}\|{\bf T}\lambdae_\Gamma - \lambdae_\Gamma \|_{\widehat\S_\Gamma}+C\gamma_1\|\lambdae_\Gamma\|_{{\M}_\Gamma}\|\H_{\A^{ad}}({\bf T}\lambdae_\Gamma - \lambdae_\Gamma) \|_{h}\\
&\leq CH\beta^{-\frac14}\gamma_1\|\lambdae_\Gamma\|_{{\M}_\Gamma}\big(C_{ED}\|{\w}_\Gamma\|_{\widetilde{\S}_{\Gamma}}+\|\lambdae_\Gamma\|_{\widehat\S_\Gamma}\big)+C\gamma_1\|\lambdae_\Gamma\|_{{\M}_\Gamma}
\big( \alpha_2 (\|\w_\Gamma\|_{\widetilde\M_\Gamma}+\|\lambdae_\Gamma\|_{\wS_\Gamma})\big) \\
&\leq C c_2\|\lambdae\|_{\M_\Gamma} (\|\w_\Gamma\|_{\widetilde\M_\Gamma}+\|\lambdae_\Gamma\|_{\M_\Gamma})
\end{aligned}
\end{equation*}
where $c_2:=\gamma_1\alpha_2$, the first inequality follows from~Lemma~\ref{Zhatestimatetilde}, 
the second to last inequality follows from~\eqref{TAdnormSw2}, a triangle inequality with~\eqref{CED} and the fact that 
\begin{align*}
\|{\bf T}\lambdae_\Gamma\|_{\widehat\S_\Gamma}=\langle {\bf T}\lambdae_{\Gamma}, {\bf T}\lambdae_{\Gamma} \rangle^{\frac12}_{\widehat\S_{\Gamma}} &= \langle \widetilde{\R}^T_{\D,\Gamma} \widetilde{\S}_{\Gamma}^{-1} \widetilde{\R}_{\D,\Gamma} \S_{\Gamma} \lambdae_{\Gamma}, \widetilde{\R}^T_{\D,\Gamma} \widetilde{\S}_{\Gamma}^{-1} \widetilde{\R}_{\D,\Gamma} \S_{\Gamma} \lambdae_{\Gamma} \rangle^{\frac12}_{{\widehat{\S}}_{\Gamma}} \\
&= \langle \E_{\D}\w_{\Gamma}, \E_{\D}\w_{\Gamma} \rangle^{\frac12}_{\widetilde{\S}_{\Gamma}}= \|\E_{\D}\w_{\Gamma}\|_{\widetilde{\S}_{\Gamma}}\leq CC_{ED}\|{\w}_\Gamma\|_{\widetilde{\S}_{\Gamma}},
\end{align*}
and the last inequality follows from the definition of $\alpha_2$ and {{Lemma \ref{MSrelation}}}.

{{Then, \eqref{wlambdaL2} can be obtained by Lemma~\ref{Ulambdaw} with the fact that 
\begin{align*}
\|\w_\Gamma - \widetilde{\R}_\Gamma \lambdae_\Gamma\|_{\widetilde{\bf L}_\Gamma}=\|\U(\w_{\A,\Gamma}- \widetilde{\R} \lambdae_{\A,\Gamma})\|_{\T_h}\leq Cc_3\|\w_{\Gamma}\|_{\widetilde\M_\Gamma},
\end{align*}  
where $c_3:=c_0H(1+\beta^{-\frac12}).$}}

Finally, we give the proof of~\eqref{wlambdaL3}, since
\begin{align*}
&\left|\langle \w_\Gamma, \widetilde{\R}_\Gamma \lambdae_\Gamma - \w_\Gamma \rangle_{\widetilde{\Z}_\Gamma} \right| \\
&\leq C\gamma_1\|\w_\Gamma\|_{\widetilde{\M}_\Gamma}\left(C(H\beta^{-\frac14})\| \widetilde{\R}_\Gamma \lambdae_\Gamma - \w_\Gamma \|_{\widetilde\S_\Gamma}+\|\H_{\A^{ad}}(\w_\Gamma - \widetilde{\R}_\Gamma \lambdae_\Gamma) \|_{h}\right)\\
&\leq CH\beta^{-\frac14}\gamma_1\|\w_\Gamma\|_{\widetilde{\M}_\Gamma}\big(\|\lambdae_\Gamma\|_{\widehat\S_\Gamma}+\|\w_\Gamma\|_{\widetilde\S_\Gamma}\big)+
C{{ \gamma_1 \alpha_1 \|\w\|_{\widetilde{\M}_\Gamma}\big( \|\w_\Gamma\|_{\widetilde{\M}_\Gamma}+ \|\lambdae\|_{\wS_\Gamma} \big) }}\\
&\leq Cc_4\|\w\|_{\widetilde{\M}_\Gamma}(\|\w\|_{\widetilde{\M}_\Gamma}+\|\lambdae_{\Gamma}\|_{\M_\Gamma}),
\end{align*}
where $c_4:=\gamma_1 \alpha_1$,
the first inequality follows from~~Lemma~\ref{Zhatestimatetilde}, the second inequality follows from~\eqref{TAdnormSw1}, and the last inequality follows from~Lemma~\ref{MSrelation}.
\end{proof}
\section{Numerical Experiments}\label{sec:numerics}
We decompose the domain $\Omega=[0,1]\times[0,1]$ into nonoverlapping subdomains and each subdomain into triangles. Using the discretization system \eqref{eq:hdg2} with HDG polynomial degrees $k=1,2$, we consider two velocity fields ${\bm \zeta}=(1,0)$ and ${\bm \zeta}=(x_2,-x_1)$. 
For each case, the regularization parameter $\beta$ is set to $1,10^{-4},10^{-6}, 10^{-8}.$
The right-hand side functions $f$ and $g$ are computed based on the exact solution $(y,p)$ as $${y}=\sin^3(\pi x_1) \sin^2(\pi x_2) \cos(\pi x_2),\quad {p}=-\sin^2(\pi x_1)\sin^2(\pi x_2)\cos(\pi x_1).$$
The stabilization parameters are chosen as $\tau_1=\max(\sup\limits_{{x\in\mathcal{E}}}({\bm\zeta}\cdot{\bf n}),0)+1, \tau_2=\tau_1-\bm{\zeta}\cdot\bn, \forall \mathcal{E}\subset\partial K, \forall K\in\mathcal{T}_h$. These choices satisfy Assumption \ref{eq:assump}.
We also employ the edge average constraints, the edge flux weighted average constraints, and the edge flux weighted first moment constraints from Assumption \ref{assumpconstrain}. We use the GMRES method without restart, and the iteration is stopped when the residual is reduced by $10^{-11}$.

{{We report our numerical results in Table~\ref{tab:table1} and Table~\ref{tab:table2}. In Table~\ref{tab:table1}, we fix the subdomain problem size and vary the number of subdomains.  
When \( k = 1 \), the basis functions are linear polynomials. With a fixed value of $\beta$, we observe that the number of iterations is independent of the number of subdomains. Moreover, the iteration count is robust with respect to changes in $\beta$.  
When \( k = 2 \), the basis functions are quadratic polynomials. In this case, the number of iterations is higher compared to the $ k = 1$ case. Nevertheless, the iteration count continues to be independent of the number of subdomains and exhibits good scalability, even for small values of $\beta$.}} 

{{In Table~\ref{tab:table2}, we fix the number of subdomains and vary the subdomain problem size. We observe that the iteration count increases slightly as \( H/h \) increases. This behavior is consistent with those results in symmetric positive definite problems~\cite{Tu2007BDDC, TuWang2016HDG, TWZstokes2020}.}}
More numerical results can be found in \cite{liuzhang2025}.

\begin{table}[htbp]
  \caption{Iteration counts for varying number of subdomains with fixed subdomain problem size \( H/h = 6 \)} 
  \label{tab:table1}
  \centering
  \scalebox{1}{
  \begin{tabular}{|c||c|c|c|c||c|c|c|c||}
    \hline
    \# of Sub. & $4^2$ & $8^2$ & $16^2$ & $32^2$ & $4^2$ & $8^2$ & $16^2$ & $32^2$ \\
    \hline \hline 
    \multicolumn{1}{|c||}{$\beta$ (Test I)} & \multicolumn{4}{c||}{$k=1$} & \multicolumn{4}{c||}{$k=2$} \\
    \hline
    $1$  & 19 & 21 & 18 & 17 & 24 & 27 & 23 & 22  \\
    $10^{-4}$ & 25 & 21 & 18 & 16 & 27 & 21 & 22 & 26  \\
    $10^{-6}$ & 17 & 21 & 18 & 15 & 22 & 22 & 18 & 21  \\
    $10^{-8}$ & 8  & 11 & 15 & 18 & 10 & 15 & 18 & 21  \\
    \hline
    \multicolumn{1}{|c||}{$\beta$ (Test II)} & \multicolumn{4}{c||}{$k=1$} & \multicolumn{4}{c||}{$k=2$} \\
    \hline
    $1$  & 7 & 6 & 6 & 5 & 11 & 10 & 10 & 12 \\
    $10^{-4}$ & 7 & 6 & 6 & 5 & 10 & 10 & 10 & 10 \\
    $10^{-6}$ & 7 & 6 & 6 & 5 & 10 & 9  & 9  & 9  \\
    $10^{-8}$ & 6 & 6 & 6 & 5 & 6  & 8  & 8  & 8  \\
    \hline
  \end{tabular}
  }
\end{table}

\begin{table}[h!]
\captionof{table}{ Iteration counts for changing subdomain problem size with $6\times 6$ subdomains}
\label{tab:table2} 
\scalebox{1}{
\begin{tabular}{|c||c|c|c|c||c|c|c|c||}
\hline
$H/h$ & $4$ & $8$ & $16$ & $20$  & $4$ & $8$ & $16$ & $20$\\
\hline \hline 
 \multicolumn{1}{|c||}{$\beta$ (Test I)}& \multicolumn{4}{c||}{$k=1$} & \multicolumn{4}{c||}{$k=2$}   \\
\hline
$1$    &18  & 23  &28   &29  &23  &29  &32   &33 \\
$10^{-4}$   &18  &22   &27   &29  &24  &28  &32   &33\\   
$10^{-6}$   &16  &21   &25   &27  &22  &23  &25   &26 \\ 
$10^{-8}$   &8   &11    &14   &15  & 11   &16  &18  &19  \\ 
\hline
$\beta$(Test II)& \multicolumn{4}{c||}{$k=1$} & \multicolumn{4}{c||}{$k=2$}   \\
\hline
$1$  &5   & 7  &9   &9 &10  &11  &10   &10 \\
$10^{-4}$ &5   &7   &9  &9  &11  &9  &10   &14\\   
$10^{-6}$ &6   &7   &9    &9      &10    &9   &8    &8 \\ 
$10^{-8}$ &5   &7   & 9   &9   &   6 & 7    &7    &7 \\ 
\hline
\end{tabular}
}
\end{table}

\section{Concluding Remarks}\label{sec:conclude}

In this work, we perform a thorough analysis of the BDDC preconditioner applied to the HDG discretization of an elliptic optimal control problem constrained by a advection-diffusion equation. We prove that the BDDC preconditioner is robust with respect to $\beta$, provided $H$ is sufficiently small. The analysis in this work can be easily extended to the case where the state equation is more general, for example, a non-divergence-free advection field. The results in this work could potentially be extended to the case of a advection-dominated state equation (cf. \cite{liu2024multigrid,TZAD2021}), or to problems with control or state constraints (cf. \cite{brenner2023multigrid,brenner2021p1,boyana2025convergence,engel2011multigrid}), where the BDDC preconditioner can be applied to the subsystem that must be solved during the outer iterative method.

\section*{Acknowledgment}
The authors would like to thank Prof. Xuemin Tu for the helpful discussion regarding this project. This material is based upon work supported by the National Science Foundation under Grant No. DMS-1929284 while SL was in residence at the Institute for Computational and Experimental Research in Mathematics in Providence, RI, during the Numerical PDEs: Analysis, Algorithms, and Data Challenges semester program.

\bibliographystyle{plain}
\bibliography{references}

\end{document}